\documentstyle[10pt,amscd,xypic,amssymb,combelow,color,enumitem]{amsart}

\xyoption{all}
\CompileMatrices

\makeatletter
\@addtoreset{equation}{section}
\makeatother

\newenvironment{proof}{{\noindent \textbf{Proof}\,\,}}{\hspace*{\fill}$\Box$\medskip}

\newtheorem{theorem}[subsection]{Theorem}

\newtheorem{proposition}[subsection]{Proposition}

\theoremstyle{definition}

\newtheorem{definition}[subsection]{Definition}

\theoremstyle{remark}

\newtheorem{claim}[subsection]{Claim}

\newtheorem{remark}[subsection]{Remark}

\def\loccitt{\emph{loc. cit.}}
\def\loccit{\emph{loc. cit. }}

\def\fa{{\mathfrak{a}}}

\def\fp{{\mathfrak{p}}}

\def\fsl{{\mathfrak{sl}}}
\def\fgl{{\mathfrak{gl}}}

\def\hgl{{\widehat{\fgl}}}

\def\fC{{\mathfrak{C}}}

\def\fW{{\mathfrak{W}}}
\def\fZ{{\mathfrak{Z}}}

\def\BC{{\mathbb{C}}}

\def\BN{{\mathbb{N}}}
\def\BF{{\mathbb{F}}}
\def\BP{{\mathbb{P}}}

\def\BQ{{\mathbb{Q}}}
\def\BZ{{\mathbb{Z}}}

\def\CA{{\mathcal{A}}}

\def\CE{{\mathcal{E}}}
\def\CF{{\mathcal{F}}}

\def\CL{{\mathcal{L}}}
\def\CM{{\mathcal{M}}}
\def\CN{{\mathcal{N}}}
\def\CO{{\mathcal{O}}}

\def\CV{{\mathcal{V}}}

\def\tCF{\widetilde{\CF}}

\def\oZ{\overline{Z}}
\def\ofZ{\overline{\fZ}}

\def\Hom{\textrm{Hom}}

\def\col{\textrm{col }}
\def\ecol{\emph{col }}

\def\vs{\varsigma}

\def\pt{\textrm{pt}}
\def\and{\textrm{ }\&\textrm{ }}

\def\sym{\textrm{Sym}}

\def\oP{\overline{P}}
\def\op{\overline{p}}

\def\ofZ{\overline{\fZ}}
\def\tCF{\widetilde{\CF}}

\def\sym{\textrm{Sym}}

\def\nn{{{\BN}}^n}

\def\su{{U_q(\dot{\fsl}_n)}}

\def\uui{{U_q(\dot{\fgl}_1)}}

\def\uu{{U_q(\dot{\fgl}_n)}}

\def\uup{{U_q^+(\dot{\fgl}_n)}}
\def\uum{{U_q^-(\dot{\fgl}_n)}}

\def\UU{{U_{q,\oq}(\ddot{\fgl}_n)}}

\def\UUp{{U^+_{q, \oq}(\ddot{\fgl}_n)}}

\def\bd{{\mathbf{d}}}

\def\bk{{\mathbf{k}}}

\def\bu{{\mathbf{u}}}
\def\bs{{\boldsymbol{\vs}}}

\def\la{{\lambda}}
\def\sq{{\square}}
\def\bsq{{\blacksquare}}

\def\lamu{{\lambda \backslash \mu}}
\def\lanu{{\lambda \backslash \nu}}

\def\numu{{\nu \backslash \mu}}

\def\bla{{\boldsymbol{\la}}}
\def\bmu{{\boldsymbol{\mu}}}
\def\bnu{{\boldsymbol{\nu}}}

\def\blamu{{\boldsymbol{\lamu}}}
\def\blanu{{\boldsymbol{\lanu}}}

\def\bnumu{{\boldsymbol{\numu}}}

\def\bA{{\mathbf{A}}}
\def\bB{{\mathbf{B}}}

\def\bV{{\mathbf{V}}}
\def\bX{{\mathbf{X}}}
\def\bY{{\mathbf{Y}}}

\def\of{{\overline{f}}}

\def\syt{{\text{SYT}}}
\def\asyt{{\text{ASYT}}}

\def\oeta{\overline{\eta}}

\def\oQ{{\overline{Q}}}
\def\oq{{\overline{q}}}

\def\bari{\bar{i}}
\def\barj{\bar{j}}
\def\bark{\bar{k}}

\begin{document}

\title[Affine Laumon spaces and a conjecture of Kuznetsov]{\Large{\textbf{Affine Laumon spaces and a conjecture of Kuznetsov}}}

\author[Andrei Negu\cb t]{Andrei Negu\cb t}
\address{MIT, Department of Mathematics, Cambridge, MA, USA}
\address{Simion Stoilow Institute of Mathematics, Bucharest, Romania}
\email{andrei.negut@@gmail.com}

\maketitle

\begin{abstract}

We prove a conjecture of Kuznetsov stating that the equivariant $K$--theory of affine Laumon spaces is the universal Verma module for the quantum affine algebra $\uu$. We do so by reinterpreting the action of the quantum toroidal algebra $\UU$ on the $K$--theory from \cite{T} in terms of the shuffle algebra studied in \cite{Tor}, which constructs an embedding $\uu \hookrightarrow \UU$.

\end{abstract}

\section{Introduction}

\noindent Laumon spaces for the group $GL_n$ parametrize flags of torsion-free sheaves on $\BP^1$:
\begin{equation}
\label{eqn:flag finite}
\CF_1 \subset ... \subset \CF_{n-1} \subset \CO_{\BP^1}^{\oplus n}
\end{equation}
whose fibers near $\infty \subset \BP^1$ match a fixed full flag of subspaces of $\BC^n$. Laumon spaces are disconnected, with components indexed by vectors $\bd = (d_1,...,d_{n-1}) \in \BN^{n-1}$ that keep track of the first Chern classes of the sheaves in \eqref{eqn:flag finite}. The component indexed by $\bd$ coincides with the space of framed degree $\bd$ quasimaps into the complete flag variety, hence the interest in these objects in representation theory. \\

\noindent We will denote quantum affine and quantum toroidal algebras by $\uu$, $\UU$, respectively (we use dots instead of the more customary hats in order to avoid double hats on our symbols). The equivariant $K$--theory groups of Laumon spaces have been studied and identified with the universal Verma module of $U_q(\fsl_n)$ in \cite{BF}. In the present paper, we study an ``affine" \footnote{The word ``affine" here does not refer to geometric properties of the spaces in question, but it refers to the fact that their cohomology/$K$--theory groups are controlled by $\hgl_n$ instead of $\fsl_n$} version of these spaces, denoted by:
$$
\CM = \bigsqcup_{\bd = (d_1,...,d_n) \in \nn} \CM_{\bd}
$$
whose definition we will recall in Subsection \ref{sub:affdef}. The reference \cite{BF} constructs an action of $\su \curvearrowright K = K_{\text{equiv}}(\CM)$ and recalls a conjecture of Kuznetsov that this action can be extended to $\uu \curvearrowright K$. Our main purpose is to prove this fact: \\

\begin{theorem}
\label{thm:main}

There is a geometric action of the affine quantum algebra $\uu$ on $K$, with the latter being isomorphic to the universal Verma module. \\

\end{theorem}

\noindent Affine Laumon spaces appear naturally in mathematical physics, geometry and representation theory as semismall resolutions of singularities of Uhlenbeck spaces for $\dot{\fgl}_n$. In \cite{CM}, we use Theorem \ref{thm:main} to prove a conjecture of Braverman that relates the Nekrasov partition function of ${{\mathcal{N}}}=2$ supersymmetric $U(n)$ gauge theory with bifundamental matter in the presence of a complete surface operator to the elliptic Calogero-Moser system. The fact that the $K$--theory group of $\CM$ is isomorphic to the universal Verma module of $\uu$ is crucial to our proof. \\

\noindent Reference \cite{T} constructed an action of the bigger algebra $\UU$ on $K$, by generators and relations. We will recast this action in terms of the shuffle algebra realization of $\UU$, see \cite{Tor}. Thus, any element of the shuffle algebra gives rise to an operator on $K$. In particular, we have introduced in \loccit the elements:
\begin{equation}
\label{eqn:operator}
S_m^\pm, T_m^\pm \in \UU \qquad \Rightarrow \qquad S_m^\pm, T_m^\pm \curvearrowright K
\end{equation}
for any Laurent polynomial $m(z_i,...,z_{j-1})$ and any pair of integers $i<j$. When $m$ is the constant Laurent polynomial $1$, the operators \eqref{eqn:operator} will give rise to the action of the root generators of $\uu$ on $K$, and we will use this fact to prove Theorem \ref{thm:main}. \\

\noindent The word ``geometric" in the statement of Theorem \ref{thm:main} comes from the fact that the operators \eqref{eqn:operator} will be given by certain explicit correspondences. Specifically, we will define three types of correspondences in Section \ref{sec:corr}: \\

\begin{itemize}
	
\item the fine correspondences $\fZ_{[i;j)}$ in Definition \ref{def:fine} \\

\item the eccentric correspondences $\ofZ_{[i;j)}$ in Definition \ref{def:eccentric} \\

\item the smooth correspondences $\fW_{(k,...,k)}$ in \eqref{eqn:rum} \\

\end{itemize}

\noindent The fine and eccentric correspondences are equipped with tautological line bundles $\CL_i,...,\CL_{j-1}$. With this in mind, we may identify the operators \eqref{eqn:operator} with those induced by the fine and eccentric correspondences: \\

\begin{theorem}
\label{thm:geom}

The shuffle element $S^\pm_m$ \emph{(}resp. $T_m^\pm$\emph{)} $\in \UU$ acts on $K$ via:
$$
m \left (\CL_i,...,\CL_{j-1} \right) \quad \emph{on} \quad \fZ_{[i;j)} \ \ (\text{resp. }\ofZ_{[i;j)})
$$
interpreted as an operator on $K$ in \eqref{eqn:opop1} \emph{(}resp. \eqref{eqn:opop2}\emph{)}. Similarly, the shuffle element $G_{\pm (k,...,k)}$ of \eqref{eqn:g} acts on $K$ via the smooth correspondence $\fW_{(k,...,k)}$. \\

\end{theorem}

\noindent The content of the present paper and of \cite{Tor} is synthesized in \cite{thesis}, with additional details, in the related context when affine Laumon spaces are replaced by Nakajima cyclic quiver varieties. Indeed, affine Laumon spaces may be interpreted as ``chainsaw quiver varieties" (see \cite{FR}), which we review in Subsection \ref{sub:quiver2}. \\

\noindent I would like to thank Alexander Braverman, Michael Finkelberg and Andrei Okounkov for teaching me all about affine Laumon spaces, and for numerous very useful talks along the way. I am also grateful to Boris Feigin, Sachin Gautam, Valerio Toledano Laredo, Francesco Sala, and especially Alexander Tsymbaliuk for their interest and help. I gratefully acknowledge the support of NSF grant DMS--1600375. 

\newpage 

\section{The quantum toroidal and shuffle algebras}\label{sec:algebra}

\subsection{} \label{sub:colors} Let us now review the main constructions of \cite{Tor}, and introduce certain notation and results that will be used. In the present paper, we will encounter variables:
$$
x_{i1}, x_{i2},...
$$
for arbitrary $i \in \{1,...,n\}$. The first index is called the \textbf{color} of the variable $x_{ia}$, and is denoted by $i = \col x_{ia}$. Although colors come from the set $\{1,...,n\}$, in many of our formulas we will encounter arbitrary colors $i \in \BZ$, by the convention:
\begin{equation}
\label{eqn:color change}
\Big( z \text{ of color }i \Big) \quad \text{is identified with} \quad \Big( z\oq^{-2 \left \lfloor \frac {i-1}n \right \rfloor} \text{ of color }\bari \Big)
\end{equation}
where $\bari$ denotes the residue class of $i \in \BZ$ in the set $\{1,...,n\}$. As an example, let us consider the following color-dependent rational function:
\begin{equation}
\label{eqn:def zeta}
\zeta \left( \frac zw \right) = \left( \frac {zq \oq^{2 \left \lceil \frac {i-j}n \right\rceil} - wq^{-1}}{z \oq^{2 \left \lceil \frac {i-j}n \right\rceil} - w}\right)^{\delta_{i - j \text{ mod } n}^0 - \delta_{i - j + 1 \text{ mod } n}^0}
\end{equation}
for variables $z,w$ of colors $i,j \in \BZ$. If we wanted to convert the right-hand sides of \eqref{eqn:def zeta} into an expression that only involves variables of colors $\bari, \barj \in \{1,...,n\}$, then:
\begin{equation}
\label{eqn:def zeta exp}
\zeta \left( \frac zw \right)  = \begin{cases} \frac {zq - wq^{-1}}{z-w} & \text{if } \bari = \barj \\
\frac {z-w}{zq - wq^{-1}} &\text{if } \bari + 1 = \barj \\ \frac {z\oq^2-w}{zq\oq^2 - wq^{-1}} &\text{if } \bari = n, \barj = 1 \\ 1 & \text{otherwise} \end{cases}
\end{equation}

\subsection{} Let us now introduce the trigonometric shuffle algebra of type $\widehat{A}_n$ (see \cite{FO} for the original inspiration for shuffle algebras in the context of elliptic quantum algebras of finite type). A rational function:
\begin{equation}
\label{eqn:shuffle}
R(...,z_{i1},...,z_{ik_i},...)
\end{equation}
($i$ goes from 1 to $n$) will be called \textbf{color-symmetric} if it is symmetric in $z_{i1},...,z_{ik_i}$ for all $i$ separately. We will refer to the vector $\bk = (k_1,...,k_n) \in \nn$ as the degree of $R$, and we will write $\bk! = k_1!...k_n!$. Let $\BF = \BQ(q,\oq)$ and define the vector space:
\begin{equation}
\label{eqn:big}
\CV = \bigoplus_{\bk \in \nn} \BF(...,z_{i1},...,z_{ik_i},...)^{\sym}
_{1 \leq i \leq n} 
\end{equation}
where the superscript $\sym$ means that we only consider color-symmetric rational functions. We make \eqref{eqn:big} into a $\BF-$algebra via the \textbf{shuffle product}:
\begin{equation}
\label{eqn:mult}
R(...,z_{i1},...,z_{ik_i},...) * R'(...,z_{i1},...,z_{ik'_i},...) = \frac 1{\bk! \cdot \bk'!} \cdot
\end{equation}
$$
\textrm{Sym} \left[ R(...,z_{i1},...,z_{ik_i},...) R'(...,z_{i,k_i+1},...,z_{i,k_i+k'_i},...) \prod_{i,i'=1}^{n} \prod^{1 \leq j \leq k_i}_{k_{i'} + 1 \leq j' \leq k_{i'} + k'_{i'}} \zeta \left( \frac {z_{ij}}{z_{i'j'}} \right) \right] 
$$
for all rational functions $R$ and $R'$ in $\bk$ and $\bk'$ variables, respectively. \\

\begin{definition}
\label{def:shuf}

The \textbf{shuffle algebra} $\CA^+ \subset \CV$ consists of rational functions:
\begin{equation}
\label{eqn:shuf}
R(...,z_{i1},...,z_{ik_i},...) = \frac {r(...,z_{i1},...,z_{ik_i},...)}{\prod_{i=1}^{n} \prod_{1\leq b \leq k_{i+1}}^{1\leq a \leq k_{i}} (z_{ia} q - z_{i+1,b} q^{-1})} 
\end{equation}
where $r$ goes over all color-symmetric Laurent polynomials satisfying the conditions:
\begin{equation}
\label{eqn:wheel}
r(...,z_{ia},...) \Big |_{z_{i1} \mapsto w, \ z_{i2} \mapsto wq^{\pm 2}, \ z_{i \mp 1,1} \mapsto w} = 0
\end{equation}
for all $i \in \{1,...,n\}$.\footnote{We note that when $i \in \{1,n\}$, we would encounter the variables $z_{01}$ or $z_{n+1,1}$ in \eqref{eqn:wheel}; in this case, we apply the convention \eqref{eqn:color change} to convert these variables into $z_{n1}\oq^2$ and $z_{11}\oq^{-2}$, respectively} \\

\end{definition}

\noindent Implicit in the terminology ``shuffle algebra" is that the vector subspace $\CA^+ \subset \CV$ defined by properties \eqref{eqn:shuf}--\eqref{eqn:wheel} is a subalgebra with respect to the multiplication \eqref{eqn:mult}. This is a straightforward exercise, proved almost word for word as Proposition 2.3 of \cite{Shuf}, so we leave it to the interested reader. \\

\subsection{} It was shown in \cite{Tor} that the shuffle algebra coincides with the subalgebra of $\CV$ generated by the one-variable rational functions $z_{i1}^d$, $\forall i \in \{1,...,n\}, d \in \BZ$. This fact allowed us to prove that there exists an isomorphism:
\begin{equation}
\label{eqn:iso plus}
\CA^+ \cong \UUp 
\end{equation}
between the shuffle algebra and the positive half of the quantum toroidal algebra (the precise definition of the latter will not be important to us, but can be found in \cite{Tor}, where we also introduce a Hopf algebra structure and pairing on the shuffle algebra). The isomorphism \eqref{eqn:iso plus} induces an isomorphism between Drinfeld doubles:
\begin{equation}
\label{eqn:iso}
\CA \cong \UU
\end{equation}
The algebra $\CA$ will be called the double shuffle algebra, and it is defined as:
\begin{equation}
\label{eqn:double}
\CA = \CA^- \otimes \CA^0 \otimes \CA^+ 
\end{equation}
where:
\begin{align*}
&\CA^- = (\CA^+)^{\text{op}} \\
&\CA^0 = \BF \left[c^{\pm 1},\psi_{i,d}^\pm \right]^{d \in \BN \sqcup 0}_{1\leq i \leq n} \Big / \left( \psi_{i,0}^+ \psi_{i,0}^- - 1 \right)
\end{align*}
and the element $c \in \CA$ is central. We refer the reader to \cite{Tor} for details on the multiplicative relations between the three factors of \eqref{eqn:double}. Let us write $\psi_i = \psi_{i,0}^+$ and note that the definition of $\CA^0$ implies that $\psi_i$ is invertible. We will write $R^-$ or $R^+$ when specifying that a particular rational function $R$ as in \eqref{eqn:shuffle} lies in either of the opposite algebras $\CA^-$ or $\CA^+$, respectively. There exists a grading:
$$
\CA^\pm = \bigoplus_{\bk \in \nn} \CA_{\pm \bk}
$$
which assigns to a rational function $R^\pm \in \CA^\pm$ its degree $\pm \bk = \pm (k_1,...,k_n)$. Thus, the degree of $R$ keeps track of the number of variables of each color of $R$. The subalgebra $\CA^0$ is concentrated in degree 0, and the grading extends to the algebra $\CA$. \\

\subsection{} For any $i<j$, let $[i;j) \in \nn$ denote the vector whose $a$--th entry is the number of integers in the set $\{i,...,j-1\}$ that are congruent to $a$ mod $n$. An important role in establishing the isomorphism \eqref{eqn:iso} was played by the rational functions:
\begin{align}
&S^\pm_m(z_i,...,z_{j-1}) = \sym \left[ \frac {m(z_i,...,z_{j-1})}{\left(1 - \frac {z_{i+1}}{z_{i}q^2}  \right) ... \left(1 - \frac {z_{j-1}}{z_{j-2}q^2} \right)} \prod_{i\leq a < b < j} \zeta \left( \frac {z_b}{z_a} \right)  \right] \in \CA_{\pm [i;j)}  \label{eqn:s} \\ &T^\pm_m(z_i,...,z_{j-1}) = \sym \left[ \frac {m(z_i,...,z_{j-1})}{\left(1 - \frac {z_{i}}{z_{i+1}}  \right) ... \left(1 - \frac {z_{j-2}}{z_{j-1}} \right)} \prod_{i\leq a < b < j} \zeta \left( \frac {z_a}{z_b} \right)  \right] \in \CA_{\pm [i;j)} \label{eqn:t} 
\end{align}
defined for all $i < j$ and all $m \in \BF[z_i^{\pm 1},...,z_{j-1}^{\pm 1}]$ (see \eqref{eqn:color change} on how to change the variables $z_i,...,z_{j-1}$ of colors $i,...,j-1$ into variables of color $\in \{1,...,n\}$, so that the right-hand sides of \eqref{eqn:s}--\eqref{eqn:t} are well-defined elements of \eqref{eqn:big}). In particular:
\begin{align}
&E_{[i;j)} := S_1^+ \label{eqn:root e} \\
&F_{[i;j)} := T_1^- \label{eqn:root f}
\end{align}
were shown in \cite{Tor} to give rise to an embedding:
\begin{equation}
\label{eqn:quantum affine}
\uu \hookrightarrow \CA
\end{equation}
$$
 e_{[i;j)} \leadsto E_{[i;j)}, \qquad f_{[i;j)} \leadsto F_{[i;j)}
 $$
$e_{[i;j)}, f_{[i;j)}$ are the root generators of the subalgebras $\uup, \uum \subset \uu$, respectively, \footnote{Note that $e_{[i;i+1)}, f_{[i;i+1)}$ are the usual Drinfeld-Jimbo generators of the subalgebra $\su \subset \uu$, which we will recount in the next Subsection} and the relations between them were deduced in \loccit based on the RTT presentation of the quantum affine algebra (akin to \cite{DF}, \cite{GM}). The indexing sets go over all integers $i<j$, where we identify the pair $[i,j)$ with $[i+n,j+n)$. \\ 

\subsection{}

Consider the following elements of the shuffle algebra, defined for all $\bk \in \nn$:
\begin{align}
&G_{\bk} =  \frac 1{(q^{-1}-q)^{|\bk|}} \prod_{1 \leq i \leq n}  \frac {\prod^{1\leq a \leq k_i}_{1\leq b \leq k_i} \left(\frac 1q - \frac {z_{ia}q}{z_{ib}} \right)}{\prod^{1\leq a \leq k_i}_{1\leq b \leq k_{i+1}} \left(\frac 1q - \frac {z_{ia}q}{z_{i+1,b}} \right)} \in \CA_{\bk} \label{eqn:g plus} \\
&G_{-\bk} =  \frac 1{(1 - q^{-2})^{|\bk|}} \prod_{1 \leq i \leq n}   \frac {\prod^{1\leq a \leq k_i}_{1\leq b \leq k_i} \left(\frac 1q - \frac {z_{ia}q}{z_{ib}} \right)}{\prod^{1\leq a \leq k_i}_{1\leq b \leq k_{i+1}} \left(\frac 1q - \frac {z_{ia}q}{z_{i+1,b}} \right)} \in \CA_{-\bk} \label{eqn:g minus}
\end{align}
where $|\bk| = k_1+...+k_n$. It was shown in \loccit that there exists an isomorphism:
$$
\su \otimes \uui \cong \uu 
$$
as well as an injective homomorphism:
$$
\uu \hookrightarrow \CA 
$$
given by sending:
\begin{align}
&\text{Drinfeld-Jimbo generators of } \su \leadsto E_{[i;i+1)}, F_{[i;i+1)} \label{eqn:dj} \\
&\text{group-like generators } g_{\pm k} \text{ of } \uui \leadsto \oq^{k^2} G_{\pm (k,...,k)} \label{eqn:g}
\end{align}
$\forall k \in \BN$. We refer to \loccit for a discussion of the Heisenberg algebra $\uui$ and its embedding into $\uu$, in terms of which the defining relation takes the form:
\begin{equation}
\label{eqn:comm heis}
[P_k, P_l] = \delta_{k+l}^0 k \cdot \frac {(q^{nk} - q^{-nk})(c^k - c^{-k})}{(\oq^{k} - \oq^{-k})(q^{nk} \oq^{k} - q^{-nk} \oq^{-k})}
\end{equation}
where $\sum_{k=0}^\infty \oq^{k^2} G_{\pm (k,...,k)} x^k  =\exp \left( \sum_{k=1}^\infty \frac {P_{\pm k}x^k}k \right)$ and $c$ is the central element of $\CA$. \\

\section{Laumon Spaces}\label{sec:laumon}

\subsection{}\label{sub:affdef}

\noindent To define affine Laumon spaces, consider the surface $\BP^1 \times \BP^1$ and the divisors:
$$
D = \BP^1 \times \{0\}, \qquad \qquad \infty = \BP^1 \times \{\infty\} \cup \{\infty\} \times \BP^1
$$
A rank $n$ \textbf{parabolic sheaf} $\CF$ is a flag of rank $n$ torsion free sheaves: 
\begin{equation}
\label{eqn:flag0}
\CF_\bullet = \Big\{ \CF_n(-D) \subset \CF_1 \subset ... \subset \CF_{n-1} \subset \CF_n \Big\}
\end{equation}
on $\BP^1 \times \BP^1$, together with a collection of isomorphisms:
$$
\xymatrix{\CF_n(-D)|_{\infty} \ar[r] \ar[d]^\cong & \CF_1|_{\infty} \ar[r] \ar[d]^\cong &  ... \ar[r] \ar[d]^\cong & \CF_{n-1}|_{\infty} \ar[r] \ar[d]^\cong& \CF_n|_{\infty} \ar[d]^\cong \\
\CO_{\infty}^{\oplus n}(-D) \ar[r] & \CO_{\infty} \oplus \CO_{\infty}^{\oplus n-1}(-D) \ar[r] & ... \ar[r] & \CO_{\infty}^{\oplus n-1} \oplus \CO_{\infty}(-D) \ar[r] & \CO_{\infty}^{\oplus n}} \qquad 
$$
The vertical isomorphism are called \textbf{framing}, and they force $c_1(\CF_i) = -(n-i)D$. On the other hand, $c_2(\CF_i) =: d_i$ can vary over all non-negative integers, and we call the vector $\bd=(d_1,...,d_n) \in \nn$ the \textbf{degree} of the parabolic sheaf \eqref{eqn:flag0}. \\

\begin{definition} 
\label{def:laumon}
	
The {\bf affine Laumon space} $\CM_{\bd}$ is the moduli space of rank $n$, degree $\bd$ parabolic sheaves as above (see the footnote on page 1 for the etymology of the word ``affine"). \\ 

\end{definition}

\noindent Affine Laumon spaces are smooth and quasi-projective varieties of dimension $2|\bd| = 2(d_1+...+d_n)$. It will be convenient to extend \eqref{eqn:flag0} to an infinite flag of sheaves, by setting $\CF_{i+n} = \CF_i(D)$ for all $i \in \BZ$. \\

\begin{remark}
\label{rem:laumon}

When $d_n=0$, one has an isomorphism:
$$
\CF_n \cong \CO^{\oplus n}_{\BP^1 \times \BP^1}
$$
for any parabolic sheaf \eqref{eqn:flag0}. In this case, the parabolic sheaf \eqref{eqn:flag0} is completely determined by its restriction to the divisor $D$, so affine Laumon spaces for $d_n = 0$ parametrize flags of sheaves \eqref{eqn:flag finite} on a projective line. \\

\end{remark}

\subsection{}\label{sub:fixed}

An alternative presentation of affine Laumon spaces was given in \cite{FFNR}, inspired by independent constructions of Biswas and Okounkov. Consider the $n$-to-1 map:
$$
\sigma : \BP^1 \times \BP^1 \longrightarrow \BP^1 \times \BP^1, \qquad \qquad \sigma(x,y) = (x,y^n) 
$$
To any flag \eqref{eqn:flag0}, we may associate the following rank $n$ torsion-free sheaf on $\BP^1 \times \BP^1$:
\begin{equation}
\label{eqn:flag}
\CF = \sigma^*(\CF_0) + \sigma^*(\CF_1)(-D) + ... + \sigma^*(\CF_{n-1})(-(n-1)D)
\end{equation}
Let us stress the fact that the right hand side is not a direct sum, but simply the linear span of the sheaves $\{ \sigma^*(\CF_i)(-iD) \}_{0\leq i < n}$ inside the sheaf $\sigma^*(\CF_{n-1})$. Because of the framing condition, we observe that:
\begin{equation}
\label{eqn:framing}
\CF|_\infty = \CO_\infty(-D) \oplus ... \oplus \CO_\infty(-nD)
\end{equation}
It is straightforward to show that $c_2(\CF) = |\bd|$. The authors of \cite{FR} modify the framing of $\CF$ at $\BP^1 \times \{\infty\} \subset \infty$ and produce a one-to-one correspondence:  
\begin{equation}
\label{eqn:misha}
(\CF_n(-D) \subset \CF_1 \subset ... \subset \CF_{n-1} \subset \CF_n) \quad \rightsquigarrow \quad \tCF
\end{equation}
where $\tCF$ matches $\CF$ off $\infty$, but the framing condition is changed to $\tCF|_\infty = \CO_\infty^{\oplus n}$. The advantage of this choice is that one can directly match the construction with the usual moduli space of framed sheaves, as defined below. \\

\begin{definition} 

Consider the moduli space $\CN_d$ which parametrizes rank $n$ torsion free sheaves $\tCF$ on $\BP^1 \times \BP^1$, with $c_2(\tCF) = d$ and a framing isomorphism $\tCF|_\infty \cong \CO_\infty^{\oplus n}$. \\
	
\end{definition}

\noindent As explained in \loccitt, any sheaf $\tCF$ on $\BP^1 \times \BP^1$ which arises from the correspondence \eqref{eqn:misha} is invariant under the action: 
\begin{equation}
\label{eqn:act1}
\BZ/n\BZ \ \curvearrowright \ \BP^1 \times \BP^1, \qquad e^{\frac {2\pi i}n}\cdot (x,y) = (x,e^{-\frac {2\pi i}n} y)
\end{equation}
and $\BZ/n\BZ \curvearrowright \CO_\infty^{\oplus n}$ via:
\begin{equation}
\label{eqn:act2}
e^{\frac {2\pi i}n} \ \mapsto \ \text{the matrix } \ \delta = \text{diag}\left(e^{\frac {2\pi i}n} ,..., e^{\frac {2\pi i(n-1)}n}, e^{\frac {2\pi in}n} \right)
\end{equation}
Conversely, any $\BZ/n\BZ-$invariant rank $n$ torsion free sheaf on $\BP^1 \times \BP^1$ is of the form \eqref{eqn:flag} for some parabolic sheaf \eqref{eqn:flag0}. We conclude that:
\begin{equation}
\label{eqn:fixedlocus}
\CN_d^{\BZ/n\BZ} \ = \ \bigsqcup_{\bd \in \nn}^{|\bd| = d} \CM_\bd
\end{equation}

\subsection{}\label{sub:quiver1}

We note that the moduli space $\CN_d$ of framed sheaves on $\BP^1 \times \BP^1$ is isomorphic to the moduli space of framed sheaves on $\BP^2$ (see \cite{BFG}, Section 4). Therefore, we conclude that $\CN_d$ is a Nakajima quiver variety, which by \cite{Nak} can be presented as the set of quadruples of linear maps $(X,Y,A,B)$ which satisfy certain properties. In more detail, consider the double framed Jordan quiver: \\

\begin{picture}(200,120)(30,-50)

\label{pic:jordan}

\put(200,31){\circle*{10}}
\put(200,51){\circle{60}}
\put(200,50){\circle{30}}

\put(192,36){\vector(4,-1){5}}
\put(208,32){\vector(-4,-1){5}}

\put(198,-10){\vector(0,1){36}}
\put(202,26){\vector(0,-1){36}}

\put(195,-20){\line(1,0){10}}
\put(195,-10){\line(1,0){10}}
\put(195,-20){\line(0,1){10}}
\put(205,-20){\line(0,1){10}}

\put(208,23){$\color{red}{V \cong \BC^d}$}
\put(208,-20){$\color{blue}{W \cong \BC^n}$}
\put(188,6){$A$}
\put(204,6){$B$}
\put(185,47){$X$}
\put(223,47){$Y$}

\put(180,-40){\text{Figure \ref{pic:jordan}}}

\end{picture} 

\noindent Specifically, the picture represents the fact that we fix vector spaces $V \cong \BC^d$ and $W \cong \BC^n$ and consider the following vector space of linear maps between them:
$$
N_d = \text{Hom}(V,V) \oplus \text{Hom}(V,V) \oplus \text{Hom}(W,V) \oplus \text{Hom}(V,W)  
$$
Elements of this vector space are quadruples of linear maps $(X,Y,A,B)$, whose domains and codomains are depicted in Figure \ref{pic:jordan}. Consider the quadratic map:
$$
N_d \stackrel{\mu}\longrightarrow \text{End}(V), \qquad \mu(X,Y,A,B) = [X,Y]+AB
$$
and the action of $GL_d := GL(V)$ on $N_d$ by conjugation. The well-known ADHM presentation of the moduli space of framed sheaves (see \loccitt) asserts that:
\begin{equation}
\label{eqn:adhm0}
\CN_d \cong \mu^{-1}(0)^s / GL_d
\end{equation}
Here and throughout this paper, the superscript $s$ means that we intersect $\mu^{-1}(0)$ with the open set of {\bf stable} quadruples $(X,Y,A,B)$, i.e. those for which the vector space $V$ is generated by $X$ and $Y$ acting on $\text{Im }A$. \\

\begin{remark}
\label{rem:stable}

The quotient \eqref{eqn:adhm0} is geometric because the stable locus defined as above matches the one prescribed by Geometric Invariant Theory (with respect to the determinant character of $GL_d$) and the $GL_d$ action on $\mu^{-1}(0)$ is free. \\

\end{remark}

\subsection{}\label{sub:quiver2}

We will now recall the \textbf{chainsaw quiver} construction from \cite{FR}. In terms of quadruples, the action $\BZ/n\BZ \curvearrowright \CN_d$ from \eqref{eqn:act1}--\eqref{eqn:act2} is given by:
$$
e^{\frac {2\pi}n} \cdot (X,Y,A,B) = (X, e^{\frac {2\pi}n} Y, A \delta, e^{\frac {2\pi}n} \delta^{-1} B)
$$
For such a quadruple to be $\BZ/n\BZ-$fixed, i.e. for the quadruple to correspond to a point in affine Laumon space by \eqref{eqn:fixedlocus}, there must exist some $g \in GL_d$ such that:
\begin{equation}
\label{eqn:china}
(X, e^{\frac {2\pi}n} Y, A\delta , e^{\frac {2\pi}n} \delta^{-1} B) = (gXg^{-1}, gYg^{-1}, gA, Bg^{-1})
\end{equation}
If we decompose $V = \bigoplus_{\lambda \in \BC} V(\lambda)$ in terms of the generalized eigenspaces of $g$, then \eqref{eqn:china} specifies that the linear maps $X,Y,A,B$ can only act non-trivially between the following pairs of eigenspaces:
\begin{align}
&V(\lambda) \stackrel{X}\rightarrow V(\lambda) \qquad \ \quad V(\lambda) \stackrel{Y}\rightarrow V \left(\lambda \cdot e^{\frac {2\pi}n} \right) \label{eqn:act 1} \\ &w_k \stackrel{A}\rightarrow V\left(e^{\frac {2\pi i k}n}\right) \qquad V\left(e^{\frac {2\pi i k}n}\right) \stackrel{B}\rightarrow w_{k+1} \label{eqn:act 2}
\end{align}
where $w_k$ denotes the basis vector of $W$ which is acted on by the character $e^{\frac {2\pi i k}n}$ in the action \eqref{eqn:act2}. Because we only consider stable quadruples, then \eqref{eqn:act 1}--\eqref{eqn:act 2} imply that the only non-zero eigenspaces are $V_k := V(e^{\frac {2\pi i k}n})$. Therefore, the maps $(X,Y,A,B)$ in a $\BZ/n\BZ-$fixed quadruple split up according to the following diagram: \\

\begin{picture}(200,130)(30,-60)\label{pic:chainsaw}

\put(43,31){\dots}
\put(343,31){\dots}

\put(60,31){\vector(1,0){45}}
\put(72,34){$Y_{i-2}$}

\put(115,31){\vector(1,0){50}}
\put(135,34){$Y_{i-1}$}

\put(175,31){\vector(1,0){50}}
\put(195,34){$Y_{i}$}

\put(235,31){\vector(1,0){50}}
\put(255,34){$Y_{i+1}$}

\put(295,31){\vector(1,0){50}}
\put(315,34){$Y_{i+2}$}

\put(110,31){\circle*{10}}
\put(110,50){\circle{30}}
\put(95,47){$X_{i-2}$}
\put(117,22){$\color{red}{V_{i-2}}$}
\put(102,36){\vector(4,-1){5}}

\put(170,31){\circle*{10}}
\put(170,50){\circle{30}}
\put(155,47){$X_{i-1}$}
\put(177,22){$\color{red}{V_{i-1}}$}
\put(162,36){\vector(4,-1){5}}

\put(230,31){\circle*{10}}
\put(230,50){\circle{30}}
\put(215,47){$X_{i}$}
\put(237,22){$\color{red}{V_{i}}$}
\put(222,36){\vector(4,-1){5}}

\put(290,31){\circle*{10}}
\put(290,50){\circle{30}}
\put(275,47){$X_{i+1}$}
\put(297,22){$\color{red}{V_{i+1}}$}
\put(282,36){\vector(4,-1){5}}

\put(75,-20){\line(1,0){10}}
\put(75,-10){\line(1,0){10}}
\put(75,-20){\line(0,1){10}}
\put(85,-20){\line(0,1){10}}
\put(88,-20){$\color{blue}{\BC w_{i-2}}$}

\put(51,27){\vector(2,-3){25}}
\put(85,-10){\vector(2,3){25}}
\put(68,4){$B_{i-2}$}
\put(98,4){$A_{i-2}$}

\put(135,-20){\line(1,0){10}}
\put(135,-10){\line(1,0){10}}
\put(135,-20){\line(0,1){10}}
\put(145,-20){\line(0,1){10}}
\put(148,-20){$\color{blue}{\BC w_{i-1}}$}

\put(111,27){\vector(2,-3){25}}
\put(145,-10){\vector(2,3){25}}
\put(128,4){$B_{i-1}$}
\put(158,4){$A_{i-1}$}

\put(195,-20){\line(1,0){10}}
\put(195,-10){\line(1,0){10}}
\put(195,-20){\line(0,1){10}}
\put(205,-20){\line(0,1){10}}
\put(208,-20){$\color{blue}{\BC w_{i}}$}

\put(171,27){\vector(2,-3){25}}
\put(205,-10){\vector(2,3){25}}
\put(188,4){$B_{i}$}
\put(218,4){$A_{i}$}

\put(255,-20){\line(1,0){10}}
\put(255,-10){\line(1,0){10}}
\put(255,-20){\line(0,1){10}}
\put(265,-20){\line(0,1){10}}
\put(268,-20){$\color{blue}{\BC w_{i+1}}$}

\put(231,27){\vector(2,-3){25}}
\put(265,-10){\vector(2,3){25}}
\put(248,4){$B_{i+1}$}
\put(278,4){$A_{i+1}$}

\put(315,-20){\line(1,0){10}}
\put(315,-10){\line(1,0){10}}
\put(315,-20){\line(0,1){10}}
\put(325,-20){\line(0,1){10}}
\put(328,-20){$\color{blue}{\BC w_{i+2}}$}

\put(291,27){\vector(2,-3){25}}
\put(325,-10){\vector(2,3){25}}
\put(308,4){$B_{i+2}$}
\put(338,4){$A_{i+2}$}

\put(180,-45){\text{Figure \ref{pic:chainsaw}}}

\end{picture} 

\noindent More precisely, we fix vector spaces $V_1,...,V_n$ of dimensions $d_1,...,d_n$ that sum up to $d$, identify $V_0$ with $V_n$, and consider the vector space of linear maps:
\begin{equation}
\label{eqn:affine}
M_\bd = \bigoplus_{i=1}^{n} \Hom(V_i,V_i) \bigoplus_{i=1}^{n} \Hom(V_{i-1},V_i) \bigoplus_{i=1}^{n} \Hom(W_i,V_i) \bigoplus_{i=1}^{n} \Hom(V_{i-1},W_i) \qquad
\end{equation}
where $\bd = (d_1,...,d_n) \in \nn$. Elements of the vector space $M_\bd$ will be quadruples $(X_i,Y_i,A_i,B_i)_{1 \leq i \leq n}$. Consider the quadratic map:
\begin{equation}
\label{eqn:moment}
M_\bd \stackrel{\nu}\longrightarrow \bigoplus_{i=1}^{n} \Hom(V_{i-1},V_i)
\end{equation}
$$
\nu(X_i,Y_i,A_i,B_i)_{1 \leq i \leq n} = \bigoplus_{i=1}^{n} \left( X_iY_i - Y_iX_{i-1}+A_iB_i \right)
$$
We let $GL_\bd := \prod_{i=1}^{n} GL(V_i)$ act on the vector space \eqref{eqn:affine} by conjugation, and with this in mind, \eqref{eqn:fixedlocus}--\eqref{eqn:adhm0} imply the following description of affine Laumon spaces:
\begin{equation}
\label{eqn:adhm1}
\CM_\bd \cong \nu^{-1}(0)^s/GL_\bd
\end{equation}
The superscript $s$ refers to the open subset of stable quadruples, i.e. those where the vector spaces $V_i$ are generated by the $X$ and $Y$ maps acting on the images of the $A$ maps. Therefore, points of affine Laumon spaces are stable quadruples of linear maps as in Figure \ref{pic:chainsaw} mod conjugation, satisfying the moment map equation $\nu = 0$. It is well-known that the moment map \eqref{eqn:moment} is a submersion on the open subset of stable quadruples, which implies the smoothness of \eqref{eqn:adhm1}. Moreover, a straightforward analogue of Remark \ref{rem:stable} shows that the quotient \eqref{eqn:adhm1} is geometric. \\

\subsection{}\label{sub:toract}

The maximal torus $T_n \subset GL_n$ and the rank 2 torus $\BC^* \times \BC^*$ act on $\CM_\bd$ by changing the trivialization at $\infty$, respectively by multiplying the base $\BP^1 \times \BP^1$. In terms of quadruples \eqref{eqn:moment}, this action is given by:
$$
(U_1,...,U_n,Q,\oQ) \cdot (..., X_i,Y_i,A_i,B_i,...) = \left(..., Q^2 X_i, \oQ^{2\delta_i^1} Y_i, A_iU^2_i, \frac {Q^2 \oQ^{2\delta_i^1} B_i}{U^2_{i}},... \right)
$$
for all $(U_1,...,U_n,Q,\oQ) \in T_n \times \BC^* \times \BC^*$. We choose even powers in the expression above in order to avoid square roots later on (in other words, the torus action we consider is a $2^{n+2}-$fold cover of the usual one). Write $T = T_n \times \BC^* \times \BC^*$, and let
\begin{equation}
\label{eqn:coordinates}
K_{T}(\pt) = \text{Rep}(T) = \BZ \left[u_1^{\pm 1},...,u_n^{\pm 1}, q^{\pm 1}, \oq^{\pm 1} \right] 
\end{equation}
(where $u_1,...,u_n,q,\oq$ denote the characters on $T$ which are dual to $U_1,...,U_n,Q,\oQ$ above). In the present paper, we will study the $T-$equivariant algebraic $K-$theory groups of affine Laumon spaces:
$$
K_\bd^{\text{int}} := K_{T}(\CM_{\bd})
$$
which are all modules over the ring \eqref{eqn:coordinates}. The superscript ``int" refers to integral $K-$theory, as opposed from the localized $K-$theory that we will mostly focus on:
$$
K_\bd := K_{T}(\CM_{\bd}) \bigotimes_{K_{T}(\pt)} \text{Frac}(K_{T}(\pt))
$$
Our main actor is the $\nn-$graded vector space:
\begin{equation}
\label{eqn:ktheory}
K = \bigoplus_{\bd \in \nn} K_\bd
\end{equation}
over the field $\BF_\bu := \text{Frac}(K_{T}(\pt)) = \BQ \left(u_1,...,u_n,q,\oq\right) $. \\

\begin{remark}
	
We claim that the natural localization map $K_{\bd}^{\text{int}} \hookrightarrow K_{\bd}$ is injective, which is simply another way of saying that $K_{\bd}^{\text{int}}$ is a torsion-free $K_T(\pt)$--module. Coupled with the fact that the algebra action $\CA \curvearrowright K$ constructed in the present paper is given by geometric correspondences, this means that (an appropriate normalization of) this action preserves $K^{\text{int}}$. However, we see no benefit in working with integral over localized $K$--theory, so we stick to the latter to keep things concise. \\

\noindent As for why $K_{\bd}^{\text{int}}$ is a torsion-free $K_T(\pt)$--module, this holds (by a standard excision argument) for all smooth quasi-projective varieties which have a $T$--equivariant affine cell decomposition. A standard generalization of the Bialynicki-Birula theorem shows that, in order to have such a cell decomposition, the only thing one needs to prove is that there exists a 1-parameter subgroup $\sigma : \BC^* \rightarrow T$ such that:
$$
\lim_{t \rightarrow 0} \sigma(t) \cdot x
$$
exists and is a $T$--fixed point of $\CM_\bd$, for all $x \in \CM_\bd$. Since $\CM_\bd  \subset \CN_{|\bd|}^{\BZ/n\BZ}$ and the $\BZ/n\BZ$--action commutes with the $T$--action, this follows from the analogous property for the smooth quasi-projective varieties $\CN_d$ (in which case, see Theorem 3.7 of \cite{NY}). \\
	
\end{remark}

\subsection{} \label{sub:tangent}

For any $\bd \in \nn$, consider the ring of color-symmetric Laurent polynomials:
$$
\Lambda_\bd = \BF_\bu[...,x_{ia}^{\pm 1},...]^{\sym}_{1 \leq i \leq n, 1\leq a \leq d_i}
$$
namely those polynomials $f$ which are symmetric in $x_{i1},...,x_{id_i}$ for each $i$ separately. Let $\pi : \nu^{-1}(0)^s \rightarrow \text{pt}$ denote the standard map, and consider the composition:
$$
\Lambda_\bd \cong K_{T \times GL_\bd}(\text{pt}) \stackrel{\pi^*}\longrightarrow K_{T \times GL_\bd}(\nu^{-1}(0)^s) \stackrel{\eqref{eqn:adhm1}}= K_{T}(\CM_\bd)
$$
\begin{equation}
\label{eqn:kirwan}
f \leadsto \of
\end{equation}
Elements of the image of the map \eqref{eqn:kirwan} will be called \textbf{tautological classes}. When $f(...,x_{ia},...) = x_{i1}+...+x_{id_i}$ is the first power sum function, then:
\begin{equation}
\label{eqn:chernroots}
\of = [\CV_i]
\end{equation}
is the class of the {\bf tautological vector bundle} $\CV_i$, whose fiber over a quadruple $\in M_\bd$ is the vector space $V_i$ itself. To compute the tangent space to $\CM_{\bd}$, we recall that the tangent space to a vector space is the vector space itself, hence:
\begin{equation}
\label{eqn:mich}
[TM_\bd] = \sum_{i = 1}^{n} \left( \frac {\CV_i}{\CV_i q^2} + \frac {\CV_i}{\CV_{i-1}} + \frac {\CV_i}{u_i^2} + \frac {u_i^2}{\CV_{i-1}q^2} \right) 
\end{equation}
The four sums above are the contributions of the $X,Y,A,B$ linear maps, respectively. To keep formulas simple, here and throughout this paper we abuse notation and write $\CV$ instead of $[\CV]$ for the class of a vector bundle. We also write:
\begin{equation}
\label{eqn:weird}
\frac {\CV'}{\CV} \quad \text{instead of} \quad [\CV' \otimes \CV^\vee]
\end{equation}
and set:
\begin{align}
&\CV_i = \CV_{\bari} \cdot \oq^{-2\left \lfloor \frac {i-1}n \right \rfloor} \label{eqn:change 1} \\
&u_i = u_{\bari} \cdot \oq^{-\left \lfloor \frac {i-1}n \right \rfloor} \label{eqn:change 2} 
\end{align}
for all $i\in \BZ$, which absorbs the dependence of \eqref{eqn:mich} on the equivariant parameter $\oq$. Since the map \eqref{eqn:moment} is a submersion on the stable locus, the $K$--theory class of $T\CM_\bd$ is equal to \eqref{eqn:mich} minus the $K$--theory classes of the tangent bundles to the codomain of $\mu$ and to the gauge group $GL_\bd$, which are respectively:
$$
\sum_{i = 1}^{n} \frac {\CV_i}{\CV_{i-1}q^2} \qquad \text{and} \qquad \sum_{i = 1}^{n} \frac {\CV_i}{\CV_i}
$$
We obtain the following formula for the tangent bundle to affine Laumon spaces:
\begin{equation}
\label{eqn:tangentspace}
[T\CM_\bd] = \sum_{i=1}^{n} \left[\left(1 - \frac 1{q^2} \right) \left(\frac {\CV_i}{\CV_{i-1}} -  \frac {\CV_i}{\CV_i} \right) + \frac {\CV_i}{u_i^2} + \frac {u_{i+1}^2}{\CV_i q^2} \right] 
\end{equation}
In particular, the rank of the tangent bundle, i.e. the dimension of $\CM_\bd$, is: 
$$
2 \sum_{i=1}^{n} \text{rk } \CV_i = 2(d_1+...+d_n) = 2|\bd|
$$

\subsection{}\label{sub:lambda}

One expects that tautological classes generate the integral $K-$theory ring $K_\bd^{\text{int}}$. We do not have a proof of this claim, but it becomes quite simple once we replace integral $K-$theory with localized $K-$theory. In other words, the following Proposition establishes the fact that the map \eqref{eqn:kirwan} is surjective upon localization: \\

\begin{proposition}
\label{prop:gen}

For any $\bd \in \BN^n$, the $\BF_\bu$--vector space $K_\bd$ is spanned by the tautological classes $\of$, as $f \in \Lambda_\bd$ goes over all color-symmetric Laurent polynomials. \\

\end{proposition}

\noindent The Proposition above will be proved in Subsection \ref{sub:restrictions}. An advantage of working with tautological classes is that we can replace the ring $\Lambda_\bd$ of color-symmetric Laurent polynomials in finitely many variables by the ring:
$$
\Lambda = \BF_\bu[...,x_{ia}^{\pm 1},...]_{1 \leq i \leq n, a\in \BN}^{\sym}
$$
of color-symmetric Laurent polynomials in infinitely many variables over the field $\BF_\bu = \BQ(u_1,...,u_n,q,\oq)$. Then we may collect the maps \eqref{eqn:kirwan} together for all $\bd$:
$$
\Lambda \longrightarrow \prod_{\bd \in \nn} K_\bd, \qquad f \leadsto \of
$$
We abuse notation by writing $\of$ either for a $K-$theory class on a single moduli space $\CM_\bd$, or for the collection of these $K-$theory classes over all $\bd \in \nn$. We will often write color-symmetric functions $f\in \Lambda$ by using the shorthand notation:
\begin{equation}
\label{eqn:notation}
f(X) = f(...,x_{ia},...)^{1 \leq i \leq n}_{a\in \BN}
\end{equation}
in terms of the alphabet $X = \sum_{i=1}^{n} \sum_{a \in \BN} x_{ia}$. Recalling the color-symmetric rational function \eqref{eqn:def zeta}, we may write expressions such as:
\begin{align}
&\zeta \left( \frac zX \right) = \prod_{j,a} \zeta \left(\frac z{x_{ja}} \right) = \prod_{a=1}^\infty \frac {zq - x_{ia}q^{-1}}{z-x_{ia}} \prod_{a=1}^\infty \frac {z - x_{i+1,a}}{zq - x_{i+1,a}q^{-1}} \label{eqn:zeta1} \\
&\zeta \left( \frac Xz \right)^{-1} = \prod_{j,a} \zeta \left(\frac {x_{ja}}z \right)^{-1} = \prod_{a=1}^\infty \frac {x_{ia} - z}{x_{ia}q - z q^{-1}} \prod_{a=1}^\infty \frac {x_{i-1,a} q - z q^{-1}}{x_{i-1,a} - z} \label{eqn:zeta2}
\end{align}
for any variable $z$ of color $i$. The right-hand sides are interpreted as color-symmetric functions in the $x$ variables by expanding $z$ around either 0 or $\infty$, depending on the situation. Moreover, the notation \eqref{eqn:notation} allows us to define the so-called {\bf plethysm} homomorphism with respect to any variable $z$ of any color $i$:
\begin{equation}
\label{eqn:pleth1}
f(X) \leadsto f(X+z) = f(x_{11},...,\underbrace{z,x_{i1},x_{i2},...}_{\text{color }i \text{ variables}},..., x_{n1},...)
\end{equation}
for any $f\in \Lambda$. The inverse operation to \eqref{eqn:pleth1} is denoted by:
\begin{equation}
\label{eqn:pleth2}
f(X) \leadsto f(X-z)
\end{equation}
It does not have a closed formula akin to \eqref{eqn:pleth1}, but if $z$ is a variable of color $i$ and $f = x_{j1}^k+x_{j2}^k+...$ is a {\bf power sum function} of variables of color $j$, then:
$$
f(X \pm z) = f(X) \pm \delta_i^j z^k
$$
Since all elements of $\Lambda$ are polynomials in the power sum functions, the equation above determines the ring homomorphisms \eqref{eqn:pleth1} and \eqref{eqn:pleth2} completely. \\

\subsection{}\label{sub:action}

Plethysms can also be defined by adding or subtracting whole alphabets:
\begin{equation}
\label{eqn:xyz}
Z = \sum^{1 \leq i \leq n}_{1\leq a \leq k_i} z_{ia}
\end{equation}
by successively adding or subtracting the individual variables $z_{ia}$. We will write:
$$
R^\pm(Z) \ = \ R^\pm(...,z_{ia},...)^{1 \leq i \leq n}_{1\leq a \leq k_i} 
$$
for arbitrary shuffle elements $R^\pm \in \CA_{\pm \bk}$. We define:
\begin{align}
&\tau_+(Z) = \prod^{1 \leq i \leq n}_{1\leq a \leq k_i} \left(\frac {u_{i+1}}q - \frac {z_{ia}q}{u_{i+1}}\right) \label{eqn:deftau 1} \\
&\tau_-(Z) = \prod^{1 \leq i \leq n}_{1\leq a \leq k_i} \left(u_{i} - \frac {z_{ia}}{u_{i}}\right) \label{eqn:deftau 2}
\end{align}
and make the convention that: 
$$
\zeta \left( \frac ZZ \right) \ = \ \prod_{i,j = 1}^{n} \prod_{a \leq k_i, b \leq k_j}^{(i,a) \neq (j,b)} \zeta \left( \frac {z_{ia}}{z_{jb}} \right)
$$
since $\zeta(z/z')$ has a pole at $z=z'$ if $\col z = \col z'$. Similarly, we define:
$$
DZ = \prod^{1 \leq i \leq n}_{1 \leq a \leq k_i} Dz_{ia}, \quad \text{where} \quad Dz = \frac {dz}{2\pi i z}
$$

\begin{definition}
\label{def:normal}
	
For a rational function $F(Z) = F(...,z_{ia},...)^{1 \leq i \leq n}_{1\leq a \leq k_i}$ with poles:
\begin{equation}
\label{eqn:factor denominator}
z_{ia} q^s \oq^t - z_{jb} q^u \oq^v
\end{equation}
for various integers $s,t,u,v$, we define:
\begin{equation}
\label{eqn:intplus}
\int^+ F(Z) DZ = \sum^{\text{functions}}_{\sigma : \{(i,a)\} \rightarrow \{1,-1\}}  \int^{|q|^{\pm 1}, |\oq|^{\pm 1} < 1}_{|z_{ia}| = \gamma^{\sigma(i,a)}  \oq^{-\frac {2i}n}} F(...,z_{ia},...) \prod_{(i,a)} \sigma(i,a) Dz_{ia} \qquad 
\end{equation}
\begin{equation}
\label{eqn:intminus}
\int^- F(Z) DZ = \sum^{\text{functions}}_{\sigma : \{(i,a)\} \rightarrow \{1,-1\}}  \int^{|q|^{\pm 1}, |\oq|^{\pm 1} > 1}_{|z_{ia}| = \gamma^{\sigma(i,a)} \oq^{-\frac {2i}n}} F(...,z_{ia},...) \prod_{(i,a)} \sigma(i,a) Dz_{ia} \qquad
\end{equation}
for some positive real $\gamma \ll 1$. In each summand, each variable $z_{ia}$ is integrated over either a very small circle of radius $\gamma\oq^{\frac {-2i}n}$ or a very large circle of radius $\gamma^{-1} \oq^{\frac {-2i}n}$. The orientation of the contours is such that the residue theorem reads:
\begin{multline}
\int_{|z| = \gamma} f(z) Dz - \int_{|z| = \gamma^{-1}} f(z) Dz = \\ = - \sum_{\text{poles } \alpha \neq \{0,\infty\}} \underset{z = \alpha}{\text{Res}} \ \frac {f(z)}z = \underset{z = 0}{\text{Res}} \ \frac {f(z)}z + \underset{z = \infty}{\text{Res}} \ \frac {f(z)}z \label{eqn:residue theorem}
\end{multline}
The meaning of the superscripts ``$|q|^{\pm 1}, |\oq|^{\pm 1} < 1$" that adorn the integral \eqref{eqn:intplus} is the following. In the summand corresponding to a particular function $\sigma$, if:
$$
\sigma(i,a) = \sigma(j,b) = 1 \qquad (\text{respectively} \quad \sigma(i,a) = \sigma(j,b) = - 1) 
$$
then the variables $z_{ia}$ and $z_{jb}$ are both integrated over the small (respectively large) circle. For any factor of the form \eqref{eqn:factor denominator} appearing in the denominator of $F(Z)$, we make the assumption that $|q|, |\oq| < 1$ (respectively $|q|, |\oq| > 1$) when evaluating the integral via residues.\footnote{As far as Theorem \ref{thm:act} is concerned, where the notion of normal-ordered integrals will be applied, the procedure in Definition \ref{def:normal} will produce the same answer regardless of the relative sizes of $q$ and $\bar{q}$, as long as they are both either $>1$ or $<1$ in absolute value} If $\sigma(i,a) \neq \sigma(j,b)$, then we need not assume anything about the sizes of $q$ and $\oq$. One defines \eqref{eqn:intminus} analogously. \\

\end{definition}

\subsection{} We may now restate Theorem 4.13 of \cite{T} in terms of the shuffle algebra: \\

\begin{theorem}
\label{thm:act}

There is an action $\CA \curvearrowright K$, where elements $R^\pm \in \CA_{\pm \bk}$ act by:
\begin{equation}
\label{eqn:p}
R^\pm * \of = \frac {q^{-|\bk|\delta_\pm^-}}{\bk!} \int^\pm \frac {R^\pm(Z)}{\zeta(Z/Z)} \overline{f(X \mp Z) \Big[ \zeta\left( \frac {Z^{\pm 1}}{X^{\pm 1}} \right)}  \tau_\pm(Z) \Big]^{\pm 1} DZ 
\end{equation}
for all $f(X) \in \Lambda$. The commuting elements $\psi_{i,d}^\pm \in \CA^0$ act by: 
\begin{equation}
\label{eqn:psi1}
\sum_{d=0}^\infty \frac {\psi_{i,d}^\pm}{z^{\pm d}} = \text{multiplication by } q^{\pm i} \left(\frac 1{u^{\pm 1}_i} - \frac {u^{\pm 1}_i}{z^{\pm 1}}\right) \cdot \overline{\zeta \left( \frac Xz \right)}
\end{equation}
expanded around $z^{\pm 1} = \infty$. \\

\end{theorem}

\begin{remark}
\label{rem:lowest}

The central element $c$ acts by $q^n \oq$, and the leading term of \eqref{eqn:psi1} is:
\begin{equation}
\label{eqn:psi2}
\psi_i = \text{multiplication by } \frac {q^{i + d_{i-1} - d_i}}{u_i}
\end{equation}
on the graded component $K_\bd \subset K$. This tells us that the various $K_\bd$ are the weight subspaces of $K$, and the weights are prescribed by the integers $d_1,...,d_n$ and the equivariant parameters $u_1,...,u_n$. \\	
\end{remark}

\begin{proof} When $R = z_{i1}^d$, for various $i\in \{1,...,n\}$ and $d\in \BZ$, formula \eqref{eqn:p} reads:
\begin{equation}
\label{eqn:z formula}
\left(z_{i1}^d\right)^\pm * \of = 
\end{equation}
$$
= q^{-\delta_\pm^-} \sum_{\pm'} \pm' \int_{|z_{i1}| = \gamma^{\pm' 1} \oq^{-\frac {2i}n}} z_{i1}^d \cdot \overline{f(X \mp z_{i1}) \Big[ \zeta \left ( \frac {z^{\pm 1}_{i1}}{X^{\pm 1}} \right)} \tau_\pm(z_{i1}) \Big]^{\pm 1} Dz_{i1} 
$$
(because we are integrating a function of a single variable in \eqref{eqn:z formula}, we do not need to assume anything about the sizes of $q$ and $\oq$). Using the formal series $\delta(z) = \sum_{d \in \BZ} z^d$, one may repackage the formula above as:
\begin{equation}
\label{eqn:e}
\delta\left(\frac {z_{i1}}z \right)^\pm * \of = q^{-\delta_\pm^-} \overline{f(X\mp z) \Big[ \zeta \left( \frac {z^{\pm 1}}{X^{\pm 1}} \right)} \tau_\pm(z) \Big]^{\pm 1}
\end{equation}
for any $i \in \{1,...,n\}$. We claim that formulas \eqref{eqn:psi1} and \eqref{eqn:e} induce an action: 
\begin{equation}
\label{eqn:action toroidal}
\UU \curvearrowright K
\end{equation}
Indeed, the defining relations in the algebra $\UU$ are spelled out in (2.70), (2.71), (2.72), (2.81), (2.82), (2.83) of \cite{Tor}, and the proof that they are satisfied by \eqref{eqn:psi1} and \eqref{eqn:e} is analogous to Theorem II.9 of \cite{thesis}. Alternatively, using Proposition \ref{prop:coeffs}, one can show that formulas \eqref{eqn:e} are equivalent (in the basis of torus fixed points that we will recall in Subsection \ref{sub:fixed point basis}) with the formulas for the action \eqref{eqn:action toroidal} defined previously in \cite{T}. By invoking \eqref{eqn:iso}, we thus obtain an action $\CA \curvearrowright K$. So far, we have showed that this action is given by \eqref{eqn:p} when $R = z_{i1}^d$ is a rational function in a single variable. Because the shuffle algebras $\CA^\pm$ are generated by the $z_{i1}^d$'s, it remains to show that formula \eqref{eqn:p} is multiplicative, i.e.:
$$
\text{if \eqref{eqn:p} holds for }R_1^\pm \text{ and }R_2^\pm, \text{ then \eqref{eqn:p} holds for }R_1^\pm * R_2^\pm
$$
We will prove this assertion in the case $\pm = +$, as the $\pm = -$ case is analogous. Assume $\gamma_2 \ll \gamma_1 \ll 1$ are positive real numbers. Applying \eqref{eqn:p} twice gives us:
\begin{multline*}
R^+_1 * R^+_2 * \of = \frac 1{\bk_1!\bk_2!} \int^{+}_{|Z_1| = \gamma_1^{\pm 1}, |Z_2|=\gamma_2^{\pm 1}} \frac {R^+_1(Z_1)R_2^+(Z_2)}{\zeta(Z_1 / Z_1)\zeta(Z_2 / Z_2)} \cdot \\ \cdot \overline{f(X-Z_1-Z_2) \zeta \left( \frac {Z_1}X \right) \zeta \left( \frac {Z_2}X \right)} \zeta \left( \frac {Z_2}{Z_1} \right)^{-1} \tau_+(Z_1)\tau_+(Z_2) DZ_1DZ_2 
\end{multline*}
where $|Z| = \gamma$ is shorthand for $|z_{ia}| = \gamma \oq^{-\frac {2i}n}$ for all $(i,a)$. The poles that involve both $Z_1$ and $Z_2$ in the integral all come from the rational function $\zeta(Z_2/Z_1)^{-1}$. The choice of contours and the assumption on the parameters $q,\oq$ in \eqref{eqn:intplus} were made in such a way that these poles do not hinder us from moving the contours together, i.e. to assume $\gamma_1 = \gamma_2 =: \gamma \ll 1$. With this in mind, $R^+_1 * R^+_2 * \of$ becomes:
$$
 \frac 1{\bk_1!\bk_2!} \int^{+}_{|Z_1| = |Z_2| = \gamma^{\pm 1}} \frac {R^+_1(Z_1)R_2^+(Z_2)\zeta \left( \frac {Z_1}{Z_2} \right)}{\zeta(Z/Z)} \cdot \overline{f(X-Z) \zeta\left(\frac ZX \right)} \tau_+(Z) DZ
$$
Since the contours are symmetric, we can replace the integrand by its symmetrization in the $Z$ variables, which precisely amounts to \eqref{eqn:p} for $R_1^+ * R_2^+$. 
	
\end{proof}
	
\subsection{}\label{sub:fixedpoints}

Let us now describe the fixed points of the action $T \curvearrowright \CM_\bd$. Because \eqref{eqn:fixedlocus} realizes Laumon spaces as $\BZ/n\BZ-$fixed loci of $\CN_{|\bd|}$, and $\BZ/n\BZ \subset T$, such points are among the torus fixed points of the moduli space of torsion-free sheaves $\CN_{|\bd|}$. It is well-known that the latter are indexed by $n-$tuples of partitions:
\begin{equation}
\label{eqn:tuple}
\bla = (\la^1,...,\la^{n}) \qquad \text{where} \qquad \la^i = (\la^i_0 \geq \la^i_1 \geq ...) 
\end{equation}
are usual partitions whose sizes sum up to $|\bd|$. Recall that partitions are in one-to-one correspondence with Young diagrams (see \cite{thesis} for an introduction to the terminology of partitions), so we will often say ``the boxes of the partition $\la^i$" instead of ``the boxes of the Young diagram corresponding to the partition $\la^i$". By extension, we will refer to the boxes of $\bla$ as the disjoint union of the boxes of the constituent partitions $\la^i$. The quadruple $(X,Y,A,B) \in \CN_{|\bd|}$ that corresponds to $\bla$ has:
$$
V = \bigoplus_{\sq \in \bla} \BC \cdot v_\sq
$$ 
and: 
\begin{equation}
\label{eqn:enrique1}
X \cdot v_{\sq} = v_{\text{box directly to the right of }\sq}
\end{equation}
\begin{equation}
\label{eqn:enrique2}
Y \cdot v_{\sq} = v_{\text{box directly above }\sq} 
\end{equation}
\begin{equation}
\label{eqn:enrique3}
A \cdot w_i = v_{\text{corner of }\la^i}
\end{equation}
as well as $B=0$. The same description applies to $\bla$ as a fixed point of $\CM_\bd$, where:
\begin{equation}
\label{eqn:iglesias}
V_i = \bigoplus_{k=1}^{n} \  \bigoplus^{i \equiv y+k \text{ mod }n}_{\sq = (x,y) \in \la^k} \BC \cdot v_\sq
\end{equation}
The maps $X_i,Y_i,A_i$ are given by restricting \eqref{eqn:enrique1}--\eqref{eqn:enrique3} to the subspaces \eqref{eqn:iglesias}, according to Figure \ref{pic:chainsaw}. To make the definition \eqref{eqn:iglesias} simpler, for the box $\sq = (x,y)$ in the partition $\la^k$, we will refer to $y+k$ as the {\bf color} of $\sq$. Then a basis of $V_i$ is given by those boxes of color $i$ mod $n$ in the $n-$tuple of partitions $\bla$. Write:
$$
\bla \vdash \bd \quad \Leftrightarrow \quad |\bla| = \bd
$$
where $\bd = (d_1,...,d_n)$ is the $n$--tuple of non-negative integers defined by the property that $\bla$ contains $d_i$ boxes of color $\equiv i$ modulo $n$, for all $i\in \{1,...,n\}$. \\

\subsection{}\label{sub:fixed point basis}

Recall the following equivariant localization formula, true for any smooth variety $X \curvearrowleft T$ acted on by a torus with isolated fixed points:
\begin{equation}
\label{eqn:eqloc0}
c = \sum_{x\in X^{T}} \frac {c|_x \cdot [x]}{[\wedge^\bullet(T_x^\vee X)]}
\end{equation}
where $[x] \in K_{T}^{\text{int}}(X)$ denotes the class of the skyscraper sheaf at $x$. For notational convenience, we renormalize $[x]$ by the exterior class of the tangent bundle:
\begin{equation}
\label{eqn:renormalize}
|x \rangle := \frac {[x]}{[\wedge^\bullet(T_x^\vee X)]} \ \in  \ K_{T}(X) := K_{T}(X)_{\text{loc}}
\end{equation}
in terms of which the equivariant localization formula becomes:
\begin{equation}
\label{eqn:eqloc}
c = \sum_{x\in X^{T}} c|_x \cdot |x\rangle \qquad \forall c \in K_{T}(X)
\end{equation}
In order to use this formula for $X = \CM_\bd$, we recall from the previous Subsection that the fixed points are indexed by $n$--tuples of partitions $\bla$. Therefore, we have: \\

\begin{proposition}
\label{prop:basis}

As a $\BF_\bu = \emph{Frac}(K_T(\emph{pt}))$ vector space, we have:
$$
K_\bd = \bigoplus_{\bla \vdash \bd} \BF_\bu \cdot |\bla \rangle
$$
for all $\bd \in \nn$. Summing over all $\bd$, we conclude that:
$$
K = \bigoplus_{\bla} \BF_\bu \cdot |\bla \rangle
$$
where the direct sum goes over all $n$--tuples of partitions. \\

\end{proposition}

\begin{proof} {\bf of Theorem \ref{thm:main}:} According to \eqref{eqn:root e} and \eqref{eqn:root f}, the assignment:
\begin{align*}
&e_{[i;j)} \leadsto E_{[i;j)} = S_1^+ \curvearrowright K \\
&f_{[i;j)} \leadsto F_{[i;j)} = T_1^- \curvearrowright K
\end{align*}
yields an action of $\uu \hookrightarrow \CA$ on $K$ (the fact that this action is ``geometric", i.e. given by explicit correspondences, will be proved as part of Theorem \ref{thm:geom}). Moreover, the Cartan elements $\psi_i$ of the quantum affine algebra act on $K$ by formulas \eqref{eqn:psi2}. The unit class $|\emptyset \rangle \in K_{(0,...,0)} = K_T(\pt)$ is a lowest weight vector, in the sense that:
\begin{align}
f_{[i;j)} \cdot |\emptyset \rangle &= 0 \label{eqn:lowest 1} \\
\psi_i \cdot |\emptyset \rangle &= \frac {q^i}{u_i} |\emptyset \rangle \label{eqn:lowest 2}
\end{align}
Letting $M$ denote the universal Verma module of $\uu$ over $\BF_\bu$, which is generated by a lowest weight vector $|\emptyset \rangle$ satisfying properties \eqref{eqn:lowest 1}--\eqref{eqn:lowest 2}, we obtain a morphism:
$$
M \rightarrow K
$$
of $\uu-$modules. Because the universal Verma module is irreducible for generic $u_1,...,u_n$, proving that $M \cong K$ reduces to showing that the modules $M$ and $K$ have the same dimension in every graded component. As a graded vector space:
$$
M \cong \uup
$$
and so that the dimension of $M_\bd$ equals the number of partitions of the degree vector $\bd$ into arcs $[i;j)$ (\cite{Tor}). The dimension of $K_\bd$ is the number of fixed points of Laumon spaces, which were interpreted in \cite{FFNR} as collections $(d_{j,i})_{i \leq j \in \BZ}$ such that:
$$
d_{j,i} \geq d_{j',i} \text{ if } j \leq j', \qquad d_{j+n,i+n} = d_{j,i}, \qquad d_{j,i} = 0 \text{ for } j-i \gg 0, \qquad d_j = \sum_{i \leq j} d_{j,i}
$$
To such a collection, we associate the following partition of $\bd$ into arcs:
$$
\bd = \sum_{i<j} [i;j) \cdot (d_{j-1,i} - d_{j,i}) 
$$
Conversely, to a partition of $\bd$ into arcs, we associate the collection:
$$
d_{j,i} = \# \text{ of arcs } [i;a) \text{ with } a > j
$$
It is easy to see that these two assignments are inverses of each other, and therefore produce the required bijection between partitions of $\bd$ into arcs and collections $(d_{j,i})$. Therefore $\dim M_\bd = \dim K_\bd$, $\forall \bd \in \nn$, and so we conclude that $M \cong K$. \\
\end{proof}

\subsection{}\label{sub:restrictions}

Later in our paper, we will need the restriction of important $K-$theory classes to the fixed points $\bla$. For the tautological vector bundle $\CV_i$, we have:
\begin{equation}
\label{eqn:tautfixed0}
\CV_i|_\bla = \sum_{k=1}^{n} \ \sum_{\sq = (x,y) \in \la^k}^{y + k \equiv i \text{ mod }n} \chi_\sq \cdot \oq^{2 \frac {y+k-i}n} 
\end{equation}
where the {\bf weight} of a box $\sq = (x,y)$ situated in the partition $\la^k \in \bla$ is:
\begin{equation}
\label{eqn:weight}
\chi_\sq := u_k^2 q^{2x}
\end{equation}
Along the same lines, the tautological class associated to any $f\in \Lambda$ has restriction to the fixed point $\bla$ given by:
\begin{equation}
\label{eqn:tautfixed1}
\of|_\bla = f(\bla) := f \left(...,\chi_\sq,... \right)_{\sq \in \bla}
\end{equation}
In the right-hand side, for any box $\sq \in \bla$ of color $i$, we plug the weight $\chi_\sq$ into an argument of the function $f$ of color $\bari$ (one must remember to take into account rule \eqref{eqn:color change} to change colors within a certain residue class modulo $n$). \\

\begin{proof} {\bf of Proposition \ref{prop:gen}}: By Proposition \ref{prop:basis}, the classes $|\bla\rangle$ span the $\BF_\bu$-vector space $K_\bd$. Then Proposition \ref{prop:gen} is a consequence of the following: \\
		
\begin{claim} Consider any vector space $V = \textrm{span}(v_1,...,v_p)$. Take a vector $v = \sum \alpha_i v_i$ with all $\alpha_i \neq 0$, and a collection of endomorphisms $A_1,...,A_n$ that are diagonal in the basis $\{v_1,...,v_p\}$. The collection of vectors:
$$
A_1^{\beta_1}... A_n^{\beta_n}\cdot v, \quad \beta_1,...,\beta_n \in \BN
$$ 
generates the vector space $V$ if for any two different basis vectors $v_i$ and $v_j$, there exists an $l\in \{1,...,n\}$ such that $A_l$ has distinct eigenvalues on $v_i$ and $v_j$. \\
\end{claim} 

\noindent The claim is an easy exercise which relies on the fact that the Vandermonde determinant is non-zero. We will apply this situation when $V=K_\bd$, $v=\overline{1}$ is the class of the structure sheaf, and $\{v_1,...,v_p\}$ is the basis of torus fixed points. We will choose $A_k$ to be the operator of multiplication by the class $[\CV_k]$. Since for any two distinct fixed points $\bla \neq \bmu$ of $\CM_\bd$, there will be a box $\sq \in \blamu$, this ensures that the hypothesis of the Claim holds. 

\end{proof}

\subsection{}\label{sub:funkyzeta}

In the current Subsection, we will study the color-symmetric rational functions \eqref{eqn:zeta1} and \eqref{eqn:zeta2}, and particularly the $K$--theory classes they give rise to under the map \eqref{eqn:kirwan}. In Theorem \ref{thm:act}, these $K$--theory classes are multiplied by the power series $\tau_\pm(z)^{\pm 1}$, and the following Proposition computes their restrictions to the fixed points. We will use multiplicative notation akin to \eqref{eqn:tautfixed1}, specifically:
\begin{equation}
\label{eqn:not}
\zeta \left(\frac z{\chi_\bla} \right) := \prod_{\sq \in \bla} \zeta \left(\frac z{\chi_\sq} \right)
\end{equation}
Given an $n$--tuple of partitions $\bla$, an \textbf{inner corner} will refer to any box $\sq \notin \bla$ such that $\bla \sqcup \sq$ is a valid $n$--tuple of partitions. Similarly, an \textbf{outer corner} of $\bla$ refers to any box $\sq \in \bla$ such that $\bla \backslash \sq$ is a valid $n$--tuple of partitions. \\

\begin{proposition}
\label{prop:restrictions}

For any variable $z$ and for any $n-$tuple of partitions $\bla$:
\begin{align}
\zeta \left( \frac z{\chi_\bla} \right) \tau_+(z) &= \frac {\prod^{\ecol \sq = \ecol z +1}_{\sq \emph{ inner corner of }\bla} \left(\frac {\sqrt{\chi_\sq}}q - \frac {zq}{\sqrt{\chi_\sq}} \right)}{\prod^{\ecol \sq = \ecol z}_{\sq \emph{ outer corner of }\bla} \left(\sqrt{\chi_\sq} - \frac {z}{\sqrt{\chi_\sq}} \right)} \label{eqn:Gamma+} \\ \zeta \left( \frac {\chi_\bla} z \right)^{-1} \tau_-(z)^{-1} &= \frac {\prod^{\ecol \sq = \ecol z - 1}_{\sq \emph{ outer corner of }\bla}\left(\sqrt{\chi_\sq}q - \frac {z}{\sqrt{\chi_\sq}q} \right)}{\prod^{\ecol \sq = \ecol z}_{\sq \emph{ inner corner of }\bla} \left(\sqrt{\chi_\sq} - \frac {z}{\sqrt{\chi_\sq}} \right)} \label{eqn:Gamma-}
\end{align}
where the outer/inner corners of a partition $\bla$ are those boxes which can be removed/added from/to $\bla$ in such a way as to produce another valid partition. \\

\end{proposition}

\begin{proof} Let us first prove \eqref{eqn:Gamma+} and leave \eqref{eqn:Gamma-} as an exercise for the interested reader. Write $i = \col z$, and let us recall that the rational function $\zeta(z/X)$ is explicitly given by formula \eqref{eqn:zeta1}. The corresponding $K-$theory class comes about by replacing the variables $x_{ia}$ by the Chern roots of the tautological vector bundle $\CV_i$. When one restricts this to the fixed point $\bla$, one must replace these Chern roots by the weights of the boxes of color $i$ in $\bla$ (one must remember to apply rule \eqref{eqn:color change} to change the colors of boxes within a single residue class modulo $n$). We conclude:
$$
\zeta \left( \frac z {\chi_\bla} \right) \tau_+(z)  = f\left(\frac {u_{i+1}}q \right) \prod_{\sq \in \bla}^{\col \sq = i} \frac {f\left(\sqrt{\chi_\sq} /q\right)}{f\left(\sqrt{\chi_\sq} \right)} \prod_{\sq \in \bla}^{\col \sq = i+1} \frac {f\left(\sqrt{\chi_\sq}\right)}{f\left(\sqrt{\chi_\sq}/q\right)}
$$	
where we abbreviate $f(x) = x - \frac zx$. The boxes in the $y-$th row of $\la^k$ have weights:
$$
u_k^2 \oq^{2\left \lfloor \frac {y+k-1}n \right \rfloor},..., u_k^2 q^{2(\la^k_y - 1)} \oq^{2\left \lfloor \frac {y+k-1}n \right \rfloor}
$$
if we think of them as having color $i = \overline{y+k}$. Therefore, we have:
$$
\zeta \left( \frac z{\chi_\bla} \right) \tau_+(z)  = f \left(\frac {u_{i+1}}q\right) \prod_{y \equiv i-k}^{1 \leq k \leq n} \frac {f \left( u_k q^{-1} \oq^{\frac {y+k-i}n} \right)}{f \left(u_k q^{\la^k_y-1}\oq^{\frac {y+k-i}n} \right)} \prod_{y \equiv i+1-k}^{1 \leq k \leq n} \frac {f \left(u_k q^{\la^k_y-1} \oq^{\frac {y+k-i-1}n} \right)}{f \left( u_k q^{-1} \oq^{\frac {y+k-i-1}n} \right)} 
$$
Changing the index $y \mapsto y+1$ in the last product implies the equation above is:
$$
= f \left(\frac {u_{i+1}}q\right) \prod_{y \equiv i-k, \ y \geq 0}^{1 \leq k \leq n} \frac {f \left( u_k q^{-1} \oq^{\frac {y+k-i}n} \right) }{f \left(u_k q^{\la^k_y-1}\oq^{\frac {y+k-i}n} \right)} \prod_{y \equiv i-k, \ y \geq -1}^{1 \leq k \leq n} \frac {f \left(u_k q^{\la^k_{y+1}-1}\oq^{\frac {y+k-i}n} \right)}{f \left( u_k q^{-1} \oq^{\frac {y+k-i}n} \right)} = 
$$
$$
= \prod_{y \equiv i-k}^{1 \leq k \leq n} \frac {f \left(u_k q^{\la^k_{y+1}-1}\oq^{\frac {y+k-i}n} \right) }{f \left(u_k q^{\la^k_y-1}\oq^{\frac {y+k-i}n} \right)} 
$$
The fraction on the second row equals 1 when $\la_y^k = \la_{y+1}^k$, so it can be $\neq 1$ only if $\la_y^k > \la_{y+1}^k$. This corresponds to the existence of an outer corner of color $i$ on the $y$-th row, and an inner corner of color $i+1$ on the $(y+1)$-th row, and thus contributes precisely the fraction claimed in \eqref{eqn:Gamma+}. \\
\end{proof}

\subsection{}\label{sub:actionfixed}

For any pair of usual partitions $\la$ and $\mu$, we will write:
$$
\la \geq \mu \qquad \text{if} \qquad \la_i \geq \mu_i, \ \ \forall i \geq 0
$$
In other words, $\la \geq \mu$ if the Young diagram of $\la$ contains that of $\mu$. In this case, the set of boxes $\lamu$ is called a {\bf skew Young diagram}. The analogous notions apply for fixed points \eqref{eqn:tuple}, which are nothing but $n-$tuples of usual partitions: 
$$
\bla \geq \bmu \qquad \text{if} \qquad \la^1 \geq \mu^1, \ ..., \ \la^{n} \geq \mu^{n}
$$
and so $\blamu$ can be thought of as an $n-$tuple of skew Young diagrams. We will think of such skew Young diagrams as sets of boxes colored modulo $n$. We will write $|\blamu| = (k_1,...,k_{n}) \in \nn$ if there are $k_i$ boxes of color $i$ modulo $n$ in $\blamu$. By analogy with \eqref{eqn:tautfixed1}, we will write:
$$
R(\blamu) \ := \ R(...,\chi_\sq,...)_{\sq \in \blamu}
$$
for all color-symmetric rational functions $R$ of degree $|\blamu|$. As before, one must plug the weight $\chi_\sq$ into an argument of $R$ of the same color as $\sq$. We are now ready to write down the matrix coefficients of the operators $R^\pm$ of Theorem \ref{thm:act} in the basis of renormalized fixed points $|\bla\rangle \in K$. \\

\begin{proposition}
\label{prop:coeffs}
	
For any shuffle element $R^\pm \in \CA_{\pm \bk} \curvearrowright K$, we have:
\begin{align}
&\langle \bla | R^+ | \bmu \rangle =  R^+(\blamu) \prod_{\bsq \in \blamu} \left[\left(q^{-1} - q \right) \zeta \left( \frac {\chi_\bsq}{\chi_\bmu} \right) \tau_+(\chi_\bsq) \right] \label{eqn:coeff+} \\ &\langle \bmu | R^- | \bla \rangle = R^-(\blamu) \prod_{\bsq \in \blamu} \left[\left(1 - q^{-2} \right) \zeta \left( \frac {\chi_\bla}{\chi_\bsq} \right)^{-1} \tau_-(\chi_\bsq)^{-1} \right] \label{eqn:coeff-}
\end{align}
The right-hand sides only make sense if $\bla \geq \bmu$ and $|\blamu| = \bk$. If either of these conditions fails to hold, then the corresponding matrix coefficient of $R^\pm$ is 0. \\
	
\end{proposition}

\begin{proof} Let us show that if \eqref{eqn:coeff+} holds for $R^+_1$ and $R^+_2$, then it holds for $R^+_1 * R^+_2$. Indeed, by the definition of the shuffle product, we have:
\begin{equation}
\label{eqn:fish}
R^+_1 * R^+_2 (\blamu) = \sum_{\blamu = A \sqcup B} R^+_1(A) R^+_2(B) \prod^{\sq \in A}_{\bsq \in B} \zeta \left( \frac {\chi_\sq}{\chi_\bsq} \right)
\end{equation}
A priori, the sum is over all ways to partition the set of boxes $\blamu$ into two disjoint parts $A$ and $B$. But recall the fact that $\zeta(x/y)|_{xq^2 \mapsto y} = 0$ when $\col y = \col x$ and $\zeta(x/y)|_{x \mapsto y} = 0$ when $\col y = \col x + 1$. This implies that if there were a box $\sq \in A$ to the left or below a box $\bsq \in B$, then the corresponding summand of \eqref{eqn:fish} would be 0. Therefore, the only non-zero summands of \eqref{eqn:fish} occur when $A = \blanu$ and $B = \bnumu$ for some intermediate partition $\bmu \leq \bnu \leq \bla$. Hence $\langle \bla | R^+_1 * R^+_2 | \bmu \rangle$ equals:
$$
\sum_{\bmu \leq \bnu \leq \bla} R^+_1(\blanu) R^+_2(\bnumu) \prod^{\sq \in \blanu}_{\bsq \in \bnumu} \zeta \left( \frac {\chi_\sq}{\chi_\bsq} \right) \prod_{\bsq \in \blamu} \left[\left(q^{-1} - q \right) \zeta \left( \frac {\chi_\bsq}{\chi_\bmu} \right) \tau_+(\chi_\bsq)  \right]
$$
The right-hand side is precisely obtained by iterating \eqref{eqn:coeff+} for $R^+_1$ and $R^+_2$. Similarly, one proves that if \eqref{eqn:coeff-} holds for $R^-_1$ and $R^-_2$, then it holds for $R_1^- * R_2^-$. \\

\noindent With the previous paragraphs in mind, \eqref{eqn:iso} reduces the proof of formulas \eqref{eqn:coeff+}--\eqref{eqn:coeff-} to the case of the shuffle elements $R^\pm = z_{i1}^d$. Let us first prove the $\pm = +$ case. Consider a fixed point $\bmu$, and pick an arbitrary element $f\in \Lambda$ such that:
\begin{equation}
\label{eqn:deltafunction}
f(\bnu) = \delta_\bnu^\bmu 
\end{equation}
for all fixed points $\bnu$ of a given size. Such an $f$ exists because \eqref{eqn:deltafunction} imposes only finitely many linear relations on the coefficients of the symmetric polynomial $f$ in infinitely many variables. Relation \eqref{eqn:z formula} implies:
$$
z_{i1}^d | \bmu \rangle = \sum_{\bla} |\bla \rangle \cdot \int z_{i1}^d \cdot \overline {f(X - z_{i1})\zeta\left( \frac {z_{i1}}X \right)} \Big|_{\bla}  \tau_+(z_{i1})  Dz_{i1} 
$$
where the integral is taken over the difference between a very small circle and a very large circle. Using \eqref{eqn:Gamma+}, the formula above becomes:
$$
\langle \bla | z_{i1}^d |\bmu \rangle = \int z_{i1}^d \cdot f(\bla - z_{i1}) \frac {\prod^{\col \sq = i+1}_{\sq \text{ inner corner of }\bla} \left(\frac {\sqrt{\chi_\sq}}q - \frac {z_{i1}q}{\sqrt{\chi_\sq}} \right)}{\prod^{\col \sq = i}_{\sq \text{ outer corner of }\bla} \left(\sqrt{\chi_\sq} - \frac {z_{i1}}{\sqrt{\chi_\sq}}\right)} Dz_{i1}
$$
We may think of the contour as surrounding the poles of the fraction, which are all of the form $z_{i1} = \chi_\bsq$ for an outer corner $\bsq \in \bla$ of color $i$. We obtain:
$$
\langle \bla | z_{i1}^d |\bmu \rangle = \sum^{\bsq \ =\text{ outer corner}}_{\text{of }\bla\text{ of color }i} \chi_\bsq^d \cdot f(\bla - \bsq) \frac {\prod^{\col \sq = i+1}_{\sq \text{ inner corner of }\bla} \left(\frac {\sqrt{\chi_\sq}}q - \frac {\chi_\bsq q}{\sqrt{\chi_\sq}} \right)}{\sqrt{\chi_\bsq} \prod^{\col \sq = i, \ \sq \neq \bsq}_{\sq \text{ outer corner of }\bla} \left(\sqrt{\chi_\sq} - \frac {\chi_\bsq}{\sqrt{\chi_\sq}}\right)} 
$$
Since $\bla - \bsq = \bnu$ must be a partition, condition \eqref{eqn:deltafunction} requires that this partition coincide with $\bmu$. We may then rewrite the above expression by changing the product over corners of $\bla$ to a product over corners of $\bmu = \bla \backslash \bsq$:
\begin{align*}
\langle \bla | z_{i1}^d |\bmu \rangle &= \chi_\bsq^d \cdot \left(q^{-1} - q\right) \frac {\prod^{\col \sq = i+1}_{\sq \text{ inner corner of }\bmu} \left(\frac {\sqrt{\chi_\sq}}q - \frac {\chi_\bsq q}{\sqrt{\chi_\sq}}\right)}{\prod^{\col \sq = i}_{\sq \text{ outer corner of }\bmu} \left(\sqrt{\chi_\sq} - \frac {\chi_\bsq}{\sqrt{\chi_\sq}}\right)} = \\ &= \chi_\bsq^d \cdot \left(q^{-1} - q \right) \ \zeta \left( \frac {\chi_\bsq}{\chi_\bmu} \right) \tau_+(\chi_\bsq)
\end{align*}
where $\bsq$ is the unique box of $\blamu$. This is precisely \eqref{eqn:coeff+} for $R^+ = z_{i1}^d$. \\

\noindent Let us now prove \eqref{eqn:coeff-} when $R^- = z_{i1}^d$. Consider a fixed point $\bla$, and pick an arbitrary element $f\in \Lambda$ such that:
\begin{equation}
\label{eqn:deltafunction 2}
f(\bnu) = \delta_\bnu^\bla 
\end{equation}
for all fixed points $\bnu$ of a given size. Relation \eqref{eqn:z formula} implies:
$$
z_{i1}^d | \bla \rangle = \frac 1q \sum_{\bmu} |\bmu \rangle \cdot \int z_{i1}^d \cdot \overline {f(X + z_{i1})\zeta\left( \frac X{z_{i1}} \right)^{-1}} \Big|_{\bmu}  \tau_-(z_{i1})^{-1}  Dz_{i1} 
$$
where the integral is taken over the difference between a very small circle and a very large circle. Using \eqref{eqn:Gamma-}, the formula above becomes:
$$
\langle \bmu | z_{i1}^d |\bla \rangle = \frac 1q \int z_{i1}^d \cdot f(\bmu + z_{i1}) \frac {\prod^{\col \sq = i-1}_{\sq \text{ outer corner of }\bmu} \left(\sqrt{\chi_\sq} q - \frac {z_{i1}}{\sqrt{\chi_\sq} q} \right)}{\prod^{\col \sq = i}_{\sq \text{ inner corner of }\bmu} \left(\sqrt{\chi_\sq} - \frac {z_{i1}}{\sqrt{\chi_\sq}}\right)} Dz_{i1}
$$
We may think of the contour as surrounding the poles of the fraction, which are all of the form $z_{i1} = \chi_\bsq$ for an inner corner $\bsq \in \bmu$ of color $i$. We obtain:
$$
\langle \bmu | z_{i1}^d |\bla \rangle = \frac 1q \sum^{\bsq \ =\text{ inner corner}}_{\text{of }\bmu \text{ of color }i} \chi_\bsq^d \cdot f(\bmu + \bsq) \frac {\prod^{\col \sq = i-1}_{\sq \text{ outer corner of }\bmu} \left(\sqrt{\chi_\sq} q - \frac {\chi_\bsq}{\sqrt{\chi_\sq}q} \right)}{\sqrt{\chi_\bsq} \prod^{\col \sq = i, \ \sq \neq \bsq}_{\sq \text{ inner corner of }\bmu} \left(\sqrt{\chi_\sq} - \frac {\chi_\bsq}{\sqrt{\chi_\sq}}\right)} 
$$
Since $\bmu + \bsq = \bnu$ must be a partition, condition \eqref{eqn:deltafunction 2} requires that this partition coincide with $\bla$. We may then rewrite the above expression by changing the product over corners of $\bmu$ to a product over corners of $\bla = \bmu \sqcup \bsq$:
\begin{align*}
\langle \bmu | z_{i1}^d |\bla \rangle &= \frac {\chi_\bsq^d}q \cdot \left(q - q^{-1}\right) \frac {\prod^{\col \sq = i-1}_{\sq \text{ outer corner of }\bla} \left(\sqrt{\chi_\sq} q - \frac {\chi_\bsq}{\sqrt{\chi_\sq} q}\right)}{\prod^{\col \sq = i}_{\sq \text{ inner corner of }\bla} \left(\sqrt{\chi_\sq} - \frac {\chi_\bsq}{\sqrt{\chi_\sq}}\right)} = \\ &= \chi_\bsq^d \cdot \left(1 -  q^{-2} \right) \left[ \zeta \left( \frac {\chi_\bla}{\chi_\bsq} \right) \tau_-(\chi_\bsq) \right]^{-1}
\end{align*}

\end{proof}

\section{Geometric Correspondences}
\label{sec:corr} 

\subsection{}\label{sub:basic}

In the present Section, we will prove Theorem \ref{thm:geom}. Specifically, we construct geometric correspondences which act on $K$ in the same way as the shuffle elements $S^\pm_m$, $T^\pm_m$, $G_{\pm \bk}$ of \eqref{eqn:s}, \eqref{eqn:t}, \eqref{eqn:g plus}, \eqref{eqn:g minus} act via Theorem \ref{thm:act}. Since the root generators \eqref{eqn:root e}--\eqref{eqn:root f} of $\uu$ are particular cases of the shuffle elements $S^+_m$ and $T^-_m$, this will also give a geometric interpretation of:
$$
\uu \subset \UU \cong \CA
$$
acting on $K$, thus completing the proof of Theorem \ref{thm:main}. Historically, the conjecture of Kuznetsov that motivated Theorem \ref{thm:main} (see \cite{FFNR} for some background) suggests that there should be an action of $\dot{\fgl}_n$ on the cohomology of affine Laumon spaces, given by the correspondences:
\begin{equation}
\label{eqn:kuznetsov}
\fC_{[i;j)} = \Big\{ ( \CF^+_\bullet \subset^{[i;j)}_\circ \CF^-_\bullet )\Big\}
\end{equation}
Above, $\circ = (0,0)$ denotes the origin of $\BP^1 \times \BP^1 \backslash \infty$, and the notation $\CF_\bullet^+ \subset^{[i;j)}_\circ \CF_\bullet^-$ means that the individual component sheaves of $\CF^+_\bullet$ and $\CF^-_\bullet$ are contained inside each other, with the quotient giving a type $[i;j)$ indecomposable representation of the cyclic quiver, supported at the origin $\circ$. The main problem with the correspondences $\fC_{[i;j)}$ is that they are not local complete intersections, and thus their structure sheaves do not give the ``right" operators in $K-$theory. Even worse, it is not clear how to define an appropriate virtual structure sheaf on $\fC_{[i;j)}$. \\

\subsection{}\label{sub:definition} 

Instead, in Definition \ref{def:fine} we will introduce an analogue of the fine correspondence studied in \cite{mod}. Afterwards, in Definition \ref{def:eccentric} we will introduce an analogue of the eccentric correspondence studied in \cite{thesis}. These two types of constructions will be the correct replacements for \eqref{eqn:kuznetsov} in $K-$theory: \\

\begin{definition}\label{def:fine}
	
The {\bf fine correspondence} is the locus of collections of $j-i+1$ parabolic sheaves, sitting inside each other: 
\begin{equation}
\label{eqn:z}
\fZ_{[i;j)} = \Big\{ ( \CF^{j}_\bullet \subset^{j-1}_\circ \CF^{j-1}_\bullet \subset^{j-2}_\circ ... \subset^{i+1}_\circ \CF^{i+1}_\bullet \subset^{i}_\circ \CF^{i}_\bullet )\Big\}
\end{equation}
where the notation $\CF_\bullet' \subset^k_\circ \CF_\bullet$ is shorthand for $\CF'_\bullet \subset^{[k;k+1)}_\circ \CF_\bullet$. In other words, this notation means that $\CF'_l \subset \CF_l$ for all integers $l$, and the quotient is non-trivial only if $l\equiv k$ mod $n$, in which case it is a length 1 skyscraper sheaf supported at the origin $\circ$. \\

\end{definition}

\noindent Because the successive quotients in a flag \eqref{eqn:z} are length one, the parabolic sheaves $\CF^j_\bullet,...,\CF^i_\bullet$ that make up such a flag must have degrees $\bd^j,...,\bd^i$, respectively, where:
$$
\bd^{k+1} = \bd^k + \bs^k \qquad \qquad \forall k \in \{i,...,j-1\}
$$
and $\bs^k \in \nn$ is the vector with entry 1 on position $\bark$, and 0 everywhere else. We will often write $\bd^j = \bd^+$ and $\bd^i = \bd^-$, so we must have $\bd^+ = \bd^- + [i;j)$. Although $\fZ_{[i;j)}$ is also quite far from being smooth, one expects that the forgetful map:
$$
\fZ_{[i;j)} \rightarrow \fC_{[i;j)} \qquad \qquad (\CF^{j}_\bullet \subset^{j-1}_\circ ... \subset^{i}_\circ \CF^{i}_\bullet) \rightarrow (\CF^j_\bullet \subset^{[i;j)}_\circ \CF^i_\bullet)
$$
(which is an isomorphism if $j-i \leq n$, but a non-trivial projective morphism otherwise) behaves like a resolution of singularities, and thus the correspondences $\fZ_{[i;j)}$ and $\fC_{[i;j)}$ give rise to the same operator in cohomology. However, this argument does not work in $K$--theory, and to fix this issue, one needs to define an appropriate virtual fundamental class on the correspondence $\fZ_{[i;j)}$. We will do so now by appealing to the quiver description of this correspondence. \\

\subsection{}\label{sub:quivercorr}

Recall from Subsection \ref{sub:quiver2} that Laumon spaces $\CM_\bd$ parametrize quadruples:
\begin{equation}
\label{eqn:duran}
(\bX,\bY,\bA,\bB) = (...,X_k,Y_k,A_k,B_k,...)_{1 \leq k \leq n}
\end{equation}
of linear maps between $n$--tuples of vector spaces $\bV = (V_1,...,V_n)$ and $\bigoplus_{i=1}^n \BC w_i$. It is easy to see that $\fZ_i := \fZ_{[i;i+1)}$ parameterizes quadruples of maps \eqref{eqn:duran} which preserve a fixed quotient $\bV^+ \twoheadrightarrow_i \bV^-$, by which we mean a fixed collection of quotients of vector spaces $V^+_k \twoheadrightarrow V^-_k$, of codimension $\delta_k^i$. Then consider the diagram:

\begin{picture}(200,180)(60,-70)\label{pic:simple}

\put(98,31){\dots}
\put(343,31){\dots}

\put(115,31){\vector(1,0){50}}
\put(135,34){$Y_{i-1}$}

\put(173,29){\vector(2,-1){53}}
\put(195,20){$Y_{i}$}

\put(234,58){\vector(2,-1){52}}
\put(244,39){$Y_{i+1}$}

\put(295,31){\vector(1,0){50}}
\put(310,20){$Y_{i+2}$}

\put(170,31){\circle*{10}}
\put(170,50){\circle{30}}
\put(155,47){$X_{i-1}$}
\put(145,20){$\color{red}{V_{i-1}}$}
\put(162,36){\vector(4,-1){5}}

\put(230,61){\circle*{10}}
\put(230,80){\circle{30}}
\put(215,77){$X^-_{i}$}
\put(212,50){$\color{red}{V^-_{i}}$}
\put(222,66){\vector(4,-1){5}}

\multiput(230,6)(0,8){7}{\line(0,1){5}}
\put(230,53){\vector(0,1){3}}
\put(230,47){\vector(0,1){3}}

\put(230,1){\circle*{10}}
\put(230,-19){\circle{30}}
\put(230,-23){$X^+_{i}$}
\put(232,8){$\color{red}{V^+_{i}}$}
\put(237,-4){\vector(-4,1){5}}

\put(290,31){\circle*{10}}
\put(290,12){\circle{30}}
\put(284,12){$X_{i+1}$}
\put(297,34){$\color{red}{V_{i+1}}$}
\put(297,27){\vector(-4,1){5}}

\put(135,-20){\line(1,0){10}}
\put(135,-10){\line(1,0){10}}
\put(135,-20){\line(0,1){10}}
\put(145,-20){\line(0,1){10}}
\put(135,-28){$\color{blue}{w_{i-1}}$}

\put(111,27){\vector(2,-3){25}}
\put(145,-10){\vector(2,3){25}}
\put(128,4){$B_{i-1}$}
\put(158,4){$A_{i-1}$}

\put(195,-20){\line(1,0){10}}
\put(195,-10){\line(1,0){10}}
\put(195,-20){\line(0,1){10}}
\put(205,-20){\line(0,1){10}}
\put(195,-28){$\color{blue}{w_{i}}$}

\put(171,27){\vector(2,-3){25}}
\put(205,-10){\vector(2,1){20}}
\put(188,4){$B_{i}$}
\put(205,-2){$A_{i}$}

\put(255,70){\line(1,0){10}}
\put(255,80){\line(1,0){10}}
\put(255,70){\line(0,1){10}}
\put(265,70){\line(0,1){10}}
\put(255,83){$\color{blue}{w_{i+1}}$}

\put(231,58){\vector(2,1){24}}
\put(265,70){\vector(2,-3){23}}
\put(244,57){$B_{i+1}$}
\put(271,63){$A_{i+1}$}

\put(315,70){\line(1,0){10}}
\put(315,80){\line(1,0){10}}
\put(315,70){\line(0,1){10}}
\put(325,70){\line(0,1){10}}
\put(315,83){$\color{blue}{w_{i+2}}$}

\put(291,33){\vector(2,3){25}}
\put(325,70){\vector(2,-3){23}}
\put(305,48){$B_{i+2}$}
\put(337,55){$A_{i+2}$}

\put(210,-55){\text{Figure \ref{pic:simple}}}

\end{picture} 

\noindent We claim that a collection of maps as in the picture above amounts to a point of $\fZ_i$. In more detail, we use the codimension 1 surjection (the vertical dotted arrow, which we henceforth call $\iota$) to produce the maps $Y_i,A_i$ to $V_i^-$ and the maps $Y_{i+1}, B_{i+1}$ from $V_i^+$. The endomorphisms $X_i^+$ and $X_i^-$ are required to satisfy $X_i^- \circ \iota = \iota \circ X_i^+$, and induce the zero map on the one-dimensional kernel of $\iota$. \\  

\subsection{}\label{sub:fine}

Iterating the above argument, points of $\fZ_{[i;j)}$ are quadruples \eqref{eqn:duran} which preserve a fixed flag of subspaces: 
\begin{equation}
\label{eqn:angela}
\bV^+ = \bV^j \twoheadrightarrow_{j-1} \bV^{j-1} \twoheadrightarrow_{j-2} ... \twoheadrightarrow_{i+1} \bV^{i+1} \twoheadrightarrow_i \bV^i = \bV^-
\end{equation}
If we fix a surjection of $n-$tuples of vector spaces $\bV^+ \twoheadrightarrow \bV^-$ of codimension $[i;j)$, then the datum \eqref{eqn:angela} is equivalent to a full flag of subspaces:
\begin{equation}
\label{eqn:funk}
0 = U^{[j;j)} \subset U^{[j-1;j)} \subset U^{[j-2;j)} \subset ... \subset U^{[i;j)} = \text{Ker}(\bV^+ \twoheadrightarrow \bV^-)
\end{equation}
Here, $U^{[a;j)}$ is itself an $n-$tuple of vector spaces of dimension $[a;j)$, and the cokernel $U^{[a;j)} / U^{[a+1;j)}$ is an $n-$tuple of vector spaces of dimension $\bs^a$. This means that the cokernel is non-zero only on the $a-$th place, where it is one-dimensional:
\begin{equation}
\label{eqn:cok1}
L_a := U^{[a;j)}_a/U^{[a+1;j)}_a
\end{equation}
Then the $n-$tuple of vector spaces $\bV^a$ which appears in \eqref{eqn:angela} is precisely $\bV^+/U^{[a;j)}$. Recall the vector space $M_{\bd^+}$ of quadruples of linear maps $(\bX,\bY,\bA,\bB)$ involving the $n-$tuple of vector spaces $\bV^+$ and $\oplus_{k=1}^n \BC w_k$. Consider the affine space:
$$
Z_{[i;j)} \subset M_{\bd^+}
$$
of quadruples that preserve the quotient $\bV^+ \twoheadrightarrow \bV^-$, and such that: \\

\begin{itemize}[leftmargin=*]
	
\item the $\bX$ maps preserve the flag \eqref{eqn:funk} and act nilpotently on it:
$$
\bX  \left( U^{[a;j)} \right) \ \subset \ U^{[a+1;j)} \qquad \forall a\in \{i,...,j-2\}
$$
	
\item the $\bY$ maps preserve the flag \eqref{eqn:funk}. Since $V_{a-1} \stackrel{Y_a}\longrightarrow V_a$ for all $a$, this forces:
$$
\bY \left( U^{[a;j)} \right) \ \subset \ U^{[a+1;j)} \qquad \forall a\in \{i,...,j-2\}
$$
	
\end{itemize}

\noindent As a consequence of the bullets above, we note that the commutator $[\bX,\bY]$ maps the vector space $U^{[a;j)}$ to $U^{[a+2;j)}$ for all $a$. We then define the linear map $\eta$:
$$
\xymatrix{
Z_{[i;j)} \ar@{^{(}->}[d] \ar[r]^-\eta & \fa_{[i;j)} \ar@{^{(}->}[d] \\
M_{\bd^+} \ar[r]^-\nu & \oplus_{k=1}^{n} \text{Hom}(V_{k-1}^+, V_k^+)}
$$
in order to make the diagram commute, where $\fa_{[i;j)}$ consists of those $n-$tuples of homomorphisms $V_{k-1}^+ \rightarrow V_k^+$ which not only preserve the flag \eqref{eqn:funk}, but map $U^{[a;j)}$ to $U^{[a+2;j)}$ for all $a$. Finally, consider the subgroup $P_{[i;j)} \subset GL_{\bd^+}$ of automorphisms of $\bV^+$ which preserves the flag \eqref{eqn:funk}. The above discussion establishes:
\begin{equation}
\label{eqn:fine}
\fZ_{[i;j)} = \eta^{-1}(0)^s/P_{[i;j)}
\end{equation}
where the superscript ``s" denotes, as always, the open subset of stable quadruples. A direct analogue of Remark \ref{rem:stable} implies that the quotient \eqref{eqn:fine} is geometric. \\

\subsection{}\label{sub:eccentric}

We will now give a related, but different construction, which admits a description as in Subsection \ref{sub:fine}, but not as in Subsection \ref{sub:definition}. Consider a quotient of $n-$tuples of vector spaces $\bV^+ \twoheadrightarrow \bV^-$ of codimension $[i;j)$, and fix a full flag of subspaces:
\begin{equation}
\label{eqn:zeit}
0 = U^{[i;i)} \subset U^{[i;i+1)} \subset U^{[i;i+2)} \subset ... \subset U^{[i;j)} = \text{Ker}(\bV^+ \twoheadrightarrow \bV^-)
\end{equation}
Note the difference between the flag above and \eqref{eqn:funk}. In \eqref{eqn:zeit}, $U^{[i;a)}$ is itself an $n-$tuple of vector spaces of dimension vector $[i;a)$, and the cokernel $U^{[i;a+1)} / U^{[i;a)}$ is an $n-$tuple of vector spaces of dimension vector $\bs^a$. This means that the cokernel is non-zero only on the $a-$th place, where it is one-dimensional:
\begin{equation}
\label{eqn:cok2}
L_a := U^{[i;a+1)}_a/U^{[i;a)}_a
\end{equation}
Recall the vector space $M_{\bd^+}$ of quadruples of linear maps $(\bX,\bY,\bA,\bB)$ involving the $n-$tuple of vector spaces $\bV^+$ and $\oplus_{k=1}^n \BC w_k$. Consider the affine space:
$$
\oZ_{[i;j)} \subset M_{\bd^+}
$$
of quadruples that preserve the quotient $\bV^+ \twoheadrightarrow \bV^-$, such that: \\

\begin{itemize}[leftmargin=*]

\item the $\bX$ maps preserve the flag \eqref{eqn:zeit} and act nilpotently on it:
$$
\bX \left( U^{[i;a)} \right) \subset U^{[i;a-1)} \qquad \forall a\in \{i+1,...,j-1\}
$$

\item the $\bY$ maps ``almost" preserve the flag \eqref{eqn:zeit}, in the sense that:
\begin{equation}
\label{eqn:almost}
\bY \left( U^{[i;a)} \right) \subset U^{[i;a+1)} \qquad \forall a\in \{i+1,...,j-1\}
\end{equation}

\end{itemize}

\noindent As a consequence of the above bullets, note that the commutator $[\bX,\bY]$ takes $U^{[i;a)}$ to $U^{[i;a)}$ for all $a$, and so preserves the flag \eqref{eqn:zeit}. We define the linear map $\oeta$: \\
$$
\xymatrix{	{\oZ}_{[i;j)} \ar@{^{(}->}[d] \ar[r]^-{\oeta} & {\overline{\fa}}_{[i;j)} \ar@{^{(}->}[d] \\
M_{\bd^+} \ar[r]^-\nu & \oplus_{k=1}^{n} \text{Hom}(V_{k-1}^+, V_k^+)}
$$
in order to make the diagram commute, where $\overline{\fa}_{[i;j)}$ consists of those $n-$tuples of homomorphisms $V_{k-1}^+ \rightarrow V_k^+$ which preserve the flag \eqref{eqn:zeit}. Consider the subgroup $\oP_{[i;j)} \subset GL_{\bd^+}$ of automorphisms of the $n-$tuple of vector spaces $\bV^+$, which preserve the flag of subspaces \eqref{eqn:zeit}. Then we define:
\begin{equation}
\label{eqn:eccentric}
\ofZ_{[i;j)} = \oeta^{-1}(0)^s/\oP_{[i;j)}
\end{equation}
where the superscript ``s" denotes the open subset of stable quadruples. A straightforward analogue of Remark \ref{rem:stable} implies that the quotient \eqref{eqn:fine} is geometric. \\

\begin{definition}
\label{def:eccentric} 

Call the variety $\ofZ_{[i;j)}$ of \eqref{eqn:eccentric} an {\bf eccentric correspondence}. \\

\end{definition}

\subsection{}\label{sub:koszul}

In Subsections \ref{sub:fine} and \ref{sub:eccentric}, we started from affine spaces $Z_{[i;j)}$ and $\oZ_{[i;j)}$, imposed on them the equations $\eta = 0$ and $\oeta = 0$, respectively, and then took the quotient under a parabolic group action. The latter operation does not affect the smoothness of the spaces in question, because the action is free on the open locus of stable points, and stability in our case is precisely the notion derived from GIT. But a priori, imposing a number of equations on an affine space may yield a ``bad" space. \\

\begin{definition}
\label{def:koszul}

Consider a section $s$ of a vector bundle $E$ on a smooth variety $X$. The structure sheaf of the scheme-theoretic zero locus $\iota : Z = \{s = 0\} \hookrightarrow X$ is isomorphic to the right-most cohomology group $\CO_X/\text{Im }s^\vee$ of the Koszul complex:
\begin{equation}
\label{eqn:koszul}
\wedge^\bullet(E,s) = \Big[ ... \stackrel{s^\vee}\longrightarrow \wedge^2 E^\vee \stackrel{s^\vee}\longrightarrow E^\vee \stackrel{s^\vee}\longrightarrow \CO_X  \Big]
\end{equation}
(if the section $s$ is regular, then all other cohomology groups vanish). In general, we define the virtual structure sheaf as the $K$--theory class:
$$
[Z] = \sum_{k=0}^{\text{rank }E} (-1)^k [\wedge^k E^\vee] \in \text{Im} \Big( K(Z) \stackrel{\iota_*}\longrightarrow K(X) \Big)
$$
In the following, we will abuse notation and refer to $[Z]$ as an element in $K(Z)$, namely the alternating sum of the cohomology groups of the complex \eqref{eqn:koszul}. \\

\end{definition}

\noindent Applying Definition \ref{def:koszul} gives us virtual structure sheaves $[\fZ_{[i;j)}]$ and $[\ofZ_{[i;j)}]$ on the fine and eccentric correspondences of \eqref{eqn:fine} and \eqref{eqn:eccentric}, respectively. \\

\begin{definition}
\label{def:line}

The {\bf tautological line bundles} $\CL_i$, ..., $\CL_{j-1}$ on $\fZ_{[i;j)}$, $\ofZ_{[i;j)}$ have fibers given by the one-dimensional quotients \eqref{eqn:cok1}, \eqref{eqn:cok2}, respectively. \\

\end{definition}

\noindent We will adjust the virtual structure sheaves $[\fZ_{[i;j)}]$ and $[\ofZ_{[i;j)}]$ by multiplying them with certain line bundles from Definition \ref{def:line} and equivariant constants. This is simply a cosmetic change, meant to match our shuffle algebra conventions:
\begin{align}
&[\fZ_{[i;j)}^+] = [\fZ_{[i;j)}] \cdot (-1)^{j-i-1} u_{i+1}...u_j \frac {\CL_i}{\CL_{j-1}}  \frac {q^{\left \lceil \frac {j-i}n \right \rceil - 2}}{q^{d^+_i-d^+_j}} \label{eqn:adjust1} \\
&[\fZ_{[i;j)}^-] = [\fZ_{[i;j)}] \cdot (-1)^{j-i-1}  \frac {u_i...u_{j-1}}{\CL_i...\CL_{j-1}} \frac {\CL_i}{\CL_{j-1}} \frac {q^{\left \lceil \frac {j-i}n \right \rceil - 2}}{q^{d^-_{i-1} - d^-_{j-1}}} \label{eqn:adjust2} \\
&[\ofZ_{[i;j)}^+] = [\ofZ_{[i;j)}] \cdot u_{i+1}...u_j \frac {q^{-j + i - \left \lfloor \frac {j-i}n \right \rfloor}}{q^{d^+_i-d^+_j}} \label{eqn:adjust3} \\
&[\ofZ_{[i;j)}^-] = [\ofZ_{[i;j)}] \cdot \frac {u_i...u_{j-1}}{\CL_i...\CL_{j-1}} \frac {q^{- j + i -\left \lfloor \frac {j-i}n \right \rfloor}}{q^{d^-_{i-1}-d^-_{j-1}}} \label{eqn:adjust4}
\end{align}
where $d_k^\pm$ are the dimensions of the vector spaces $V_k^\pm$ in the definition of $\fZ_{[i;j)}$, $\ofZ_{[i;j)}$. \\

\subsection{}\label{sub:operators}

Consider the following projection maps:
\begin{equation}
\label{eqn:diagram}
\xymatrix{ & \fZ_{[i;j)} \text{ or } \ofZ_{[i;j)} \ar[ld]_{p^+ \text{ or } \op^+} \ar[rd]^{p^- \text{ or } \op^-}  & \\
\CM_{\bd^+} & & \CM_{\bd^-}}
\end{equation}
which remember only the quadruple of linear maps on the $n-$tuples of vector spaces $\bV^+$ or $\bV^-$. For any Laurent polynomial $m(z_i,...,z_{j-1})$, we define the operators:
\begin{align}
&s_m^\pm : K \rightarrow K, \quad \alpha \mapsto p^{\pm}_*\Big(m \left( \CL_i ,..., \CL_{j-1} \right) \cdot [\fZ_{[i;j)}^\pm] \cdot p^{\mp *}(\alpha)\Big)  \label{eqn:opop1} \\
&t_m^\pm : K \rightarrow K, \quad \alpha \mapsto  \op^{\pm}_*\Big(m \left( \CL_i,..., \CL_{j-1} \right) \cdot [\ofZ_{[i;j)}^\pm] \cdot \op^{\mp *}(\alpha)\Big) \label{eqn:opop2}
\end{align}
The way the virtual structure sheaves in \eqref{eqn:opop1}--\eqref{eqn:opop2} give rise to endomorphisms of $K$ is well-known. In more detail, recall from \eqref{eqn:fine} that $\fZ_{[i;j)}$ is cut out from the affine space $Z_{[i;j)}$ (modulo the group $P_{[i;j)}$) by the section $\eta$ of the vector bundle with fibers $\fa_{[i;j)}$. Therefore, Definition \ref{def:koszul} gives rise to a virtual fundamental class inside the $P_{[i;j)}$ equivariant $K$--theory of $Z_{[i;j)}$, which then maps to the equivariant $K$--theory group of $\CM_{\bd^+} \times \CM_{\bd^-}$, via the projection maps \eqref{eqn:diagram}. This $K$--theory class produces the operator \eqref{eqn:opop1}, and \eqref{eqn:opop2} is defined analogously. \\

\begin{remark}
\label{rem:simple}

The correspondence $\fZ_i := \fZ_{[i;i+1)} = \ofZ_{[i;i+1)}$ is smooth of middle dimension in $\CM_{\bd^+} \times \CM_{\bd^-}$. We call $\fZ_i$ a {\bf simple correspondence}, and note that it was used in \cite{FFNR} and \cite{T} to construct operators in the cohomology and $K-$theory of affine Laumon spaces, respectively. In more detail, \loccit used the operators \eqref{eqn:opop1} and \eqref{eqn:opop2} for $j=i+1$ to construct an action $\UU \curvearrowright K$, precisely the action that we invoked in the proof of Theorem \ref{thm:act}. \\

\end{remark}

\subsection{}\label{sub:dimcount}

The equivariant localization formula \eqref{eqn:eqloc0} may be adapted to the case when the smooth variety $X$ is replaced by the zero locus $Z \hookrightarrow X$ of a regular section of a vector bundle $E$ (as in Definition \ref{def:koszul}), and it reads:
\begin{equation}
\label{eqn:eqloc vir}
c = \sum_{z \in Z^T} c|_z \cdot [z] \cdot \frac {[\wedge^\bullet(E_z^\vee)]}{[\wedge^\bullet(T_z^\vee X)]} = \frac {c|_z \cdot [z]}{[\wedge^\bullet(T_z^\vee X - E_z^\vee)]}
\end{equation}
The difference $T Z := T X - E$ is considered to be the virtual tangent space to the zero locus $Z \hookrightarrow X$. Therefore, formula \eqref{eqn:eqloc} still holds for $X$ replaced with $Z$, if the $K$--theory classes of the skyscraper sheaves $[z]$ are replaced by:
\begin{equation}
\label{eqn:renormalization vir}
| z \rangle = \frac {[z]}{[\wedge^\bullet(T_z^\vee Z)]} = \frac {[z]}{[\wedge^\bullet(T_z^\vee X - E_z^\vee)]}
\end{equation}
Even when the section is non-regular, \eqref{eqn:eqloc vir} and \eqref{eqn:renormalization vir} continue to hold, at the price of replacing the scheme-theoretic zero locus $Z$ by the virtual structure sheaf of Definition \ref{def:koszul}. In the particular cases of the fine and eccentric correspondences, it follows from \eqref{eqn:fine} and \eqref{eqn:eccentric} that their virtual tangent spaces are given by:
\begin{align}
&[T \fZ_{[i;j)}] = [TZ_{[i;j)]}] - [\fa_{[i;j)}] - [\fp_{[i;j)}] \label{eqn:vir fine} \\
&[T \ofZ_{[i;j)}] = [T\oZ_{[i;j)]}] - [\overline{\fa}_{[i;j)}] - [\overline{\fp}_{[i;j)}] \label{eqn:vir eccentric}
\end{align} 
Since the Lie algebra of an algebraic group is the vector space of invariant vector fields, the reason why we subtract the Lie algebras $\fp_{[i;j)}$ and $\overline{\fp}_{[i;j)}$ from \eqref{eqn:vir fine}--\eqref{eqn:vir eccentric} is to account for killing those tangent vectors that come from the $P_{[i;j)}$ and $\oP_{[i;j)}$ actions in the quotients \eqref{eqn:fine} and \eqref{eqn:eccentric}, respectively. All the summands in \eqref{eqn:vir fine} and \eqref{eqn:vir eccentric} may be expressed in terms of the tautological vector bundles: 
$$
\{\CV_k^+\}_{1 \leq k \leq n} \quad \text{and} \quad \{\CV_k^-\}_{1 \leq k \leq n}
$$
pulled back from Laumon spaces via the maps $p^\pm$ or $\op^\pm$ of \eqref{eqn:diagram}, and in terms of the tautological line bundles $\CL_i,...,\CL_{j-1}$ from Definition \ref{def:line}. The precise formula is given in the following Proposition, which will be proved in the Appendix. \\

\begin{proposition}
\label{prop:tangent}
	
The virtual tangent space to the fine correspondence $\fZ_{[i;j)}$ is:
$$
\left[ T \fZ_{[i;j)} \right] - \left[ T\CM_{\bd^+} \right] = i - j + \left(1 - \frac 1{q^2}\right) \left( \sum_{i \leq a < j} \frac {\CV^+_a}{\CL_a} - \sum_{i \leq a < j} \frac {\CV^+_{a+1}}{\CL_a} - \right.
$$ 
\begin{equation}
\label{eqn:form1}
\left. - \sum_{i \leq a < b < j}^{a \equiv b} \frac {\CL_b}{\CL_a}  + \sum_{i \leq a < b < j}^{a \equiv b - 1} \frac {\CL_b}{\CL_a} \right) + \sum_{a=i+1}^{j-1} \frac {\CL_a}{\CL_{a-1} q^2} - \sum_{a=i}^{j-1} \frac {u^2_{a+1}}{\CL_aq^2}
\end{equation}
and the virtual tangent space to the eccentric correspondence $\ofZ_{[i;j)}$ is: 
$$
\left[ T \ofZ_{[i;j)} \right] - \left[ T\CM_{\bd^+} \right] = i - j + \left(1 - \frac 1{q^2}\right) \left( \sum_{i \leq a < j} \frac {\CV^+_a}{\CL_a} - \sum_{i \leq a < j} \frac {\CV^+_{a+1}}{\CL_a} - \right.
$$ 
\begin{equation}
\label{eqn:form3}
\left. - \sum_{i \leq a < b < j}^{a \equiv b} \frac {\CL_a}{\CL_b}  + \sum_{i \leq a < b < j}^{a \equiv b + 1} \frac {\CL_{a}}{\CL_b} \right) + \sum_{a=i+1}^{j-1} \frac {\CL_a}{\CL_{a-1}} - \sum_{a=i}^{j-1} \frac {u^2_{a+1}}{\CL_aq^2}
\end{equation}

\end{proposition}

\noindent Since the $K$--theory class of the tangent bundle to $\CM_{\bd^+}$ is computed by \eqref{eqn:tangentspace}, formulas \eqref{eqn:form1} and \eqref{eqn:form3} allow one to compute the tangent space to the fine and eccentric correspondences directly. However, for the purpose of localization computations later on, it makes more sense to leave these formulas in the form above. We leave the following Proposition as an exercise to the interested reader, but note that it readily follows by combining \eqref{eqn:form1}--\eqref{eqn:form3} with the difference between $[T\CM_{\bd^+}]$ and $[T\CM_{\bd^-}]$, which can be computed using \eqref{eqn:tangentspace}. \\

\begin{proposition}
\label{prop:tangent 2}

The following formulas are equivalent to \eqref{eqn:form1} and \eqref{eqn:form3}:
$$
\left[ T \fZ_{[i;j)} \right] - \left[ T\CM_{\bd^-} \right] = i - j  + \left(1 - \frac 1{q^2}\right) \left( - \sum_{i \leq a < j} \frac {\CL_a}{\CV^-_a} + \sum_{i \leq a < j} \frac {\CL_a}{\CV^-_{a-1}} - \right.
$$ 
\begin{equation}
\label{eqn:form2}
\left. - \sum_{i \leq a < b < j}^{a \equiv b} \frac {\CL_b}{\CL_a}  + \sum_{i \leq a < b < j}^{a \equiv b - 1} \frac {\CL_b}{\CL_a} \right) + \sum_{a=i+1}^{j-1} \frac {\CL_a}{\CL_{a-1} q^2} + \sum_{a=i}^{j-1} \frac {\CL_a}{u^2_{a}}
\end{equation}
$$
\left[ T \ofZ_{[i;j)} \right] - \left[ T\CM_{\bd^-} \right] = i - j + \left(1 - \frac 1{q^2}\right) \left( - \sum_{i \leq a < j} \frac {\CL_a}{\CV^-_a} + \sum_{i \leq a < j} \frac {\CL_a}{\CV^-_{a-1}} - \right.
$$ 
\begin{equation}
\label{eqn:form4}
\left. - \sum_{i \leq a < b < j}^{a \equiv b} \frac {\CL_a}{\CL_b}  + \sum_{i \leq a < b < j}^{a \equiv b + 1} \frac {\CL_{a}}{\CL_b} \right) + \sum_{a=i+1}^{j-1} \frac {\CL_a}{\CL_{a-1}} + \sum_{a=i}^{j-1} \frac {\CL_a}{u^2_{a}}
\end{equation}
respectively. \\

\end{proposition}

\subsection{}\label{sub:new}

We will introduce another geometric operator, which will be shown to correspond to the shuffle elements \eqref{eqn:g plus}--\eqref{eqn:g minus}. Consider the locus:
\begin{equation}
\label{eqn:graham}
\fW_{\bd^+,\bd^-} \subset \CM_{\bd^+} \times \CM_{\bd^-}
\end{equation}
consisting of pairs of parabolic sheaves $\CF^+_\bullet \subset \CF^-_\bullet$ such that the quotient $\CF^-_\bullet/\CF^+_\bullet$ is scheme-theoretically supported on the divisor $D' = \{0\} \times \BP^1$, i.e.:
$$
\CF^+_i \subset \CF^-_i \subset \CF^+_i(D')
$$
for all $i \in \BZ$. Note that $\fW_{\bd^+,\bd^-}$ is the $\BZ/n\BZ$--fixed locus of the space ${\mathfrak{V}}^{|\bd^+|-|\bd^-|}$ defined in Section 6.2 of \cite{mod}. The space ${\mathfrak{V}}^k$ was itself realized as a $\BZ/2\BZ-$fixed locus of a certain moduli space of sheaves, and this fact was used in Section 6.10 of \loccit to show that ${\mathfrak{V}}^k$ is smooth. Refining this argument would allow one to show that $\fW_{\bd^+,\bd^-}$ is smooth, and we leave the details to the interested reader. \\ 

\noindent Instead of actually proving that $\fW_{\bd^+,\bd^-}$ is smooth, we will work with its virtual structure sheaf, as in Subsections \ref{sub:koszul} and \ref{sub:dimcount}. To do so, we introduce the quiver presentation of $\fW_{\bd^+,\bd^-}$. Take a quotient of $n-$tuples of vector spaces: 
\begin{equation}
\label{eqn:twostep}
\bV^+ \twoheadrightarrow \bV^-
\end{equation}
of dimension vectors $\bd^+$ and $\bd^-$, and consider the vector space of quadruples:
$$
W_{\bd^+,\bd^-} = \Big\{ (\bX,\bY,\bA,\bB) \Big\} \subset M_{\bd^+}
$$
which preserve the quotient \eqref{eqn:twostep}, and such that $\bX$ vanishes on the kernel of \eqref{eqn:twostep}. Consider the restriction of the moment map \eqref{eqn:moment} to the vector space $W_{\bd^+,\bd^-}$:
$$
\xymatrix{W_{\bd^+,\bd^-} \ar@{^{(}->}[d] \ar[r]^-{\upsilon} & \fa := \oplus_{i=1}^{n} \text{Hom}(V_{i-1}^-, V_i^+) \ar@{^{(}->}[d] \\
M_{\bd^+} \ar[r]^-\nu & \oplus_{i=1}^{n} \text{Hom}(V_{i-1}^+, V_i^+)}
$$
The reason why the moment map $\nu|_W$ factors through the space on the top right is the fact that the maps $X_i$ vanish on $L_i := \text{Ker}(V_{i}^+ \twoheadrightarrow V_{i}^-)$. Therefore:
\begin{equation}
\label{eqn:quiverw}
\fW_{\bd^+,\bd^-} = \upsilon^{-1}(0)^s/P
\end{equation}
where $P \subset GL_{\bd^+}$ is the subgroup of automorphisms which preserve the quotient \eqref{eqn:twostep}. As in Definition \ref{def:koszul}, the presentation \eqref{eqn:quiverw} gives us a virtual structure sheaf $[\fW_{\bd^+,\bd^-}]$ as an element in the $K$--theory group of $\fW_{\bd^+,\bd^-}$. Moreover, the discussion in Subsection \ref{sub:dimcount} implies that the virtual tangent bundle is given by:
\begin{equation}
\label{eqn:vir smooth}
[T\fW_{\bd^+,\bd^-}] = [T W_{\bd^+,\bd^-}] - [\fa] - [\fp]
\end{equation} 
(the fact that $\fW_{\bd^+,\bd^-}$ is smooth of expected dimension is equivalent with \eqref{eqn:vir smooth} coinciding with the actual tangent bundle). To keep our formulas in line with our conventions on shuffle algebras, we slightly adjust the virtual structure sheaves:
\begin{align}
&[\fW_{\bd^+,\bd^-}^+] = [\fW_{\bd^+,\bd^-}] \cdot \frac {u_{(\bd^+ - \bd^-)+1}}{q^{|\bd^+| - |\bd^-| + \langle \bd^+ - \bd^-, \bd^+ \rangle}}  \label{eqn:adjust5} \\ 
&[\fW_{\bd^+,\bd^-}^-] = [\fW_{\bd^+,\bd^-}] \cdot \frac { (-1)^{|\bd^+| - |\bd^-|} u_{\bd^+-\bd^-}}{q^{\langle \bd^-, \bd^- - \bd^+ \rangle}\prod_{i=1}^n \det \CL_i}  \label{eqn:adjust6}
\end{align}
where $\bd^\pm$ are the degrees of the moduli spaces in \eqref{eqn:graham}, and we write:
\begin{align}
&u_{\bk} = \prod_{i=1}^n u_i^{k_i} \label{eqn:not 1} \\
&u_{\bk+1} = \prod_{i=1}^n u_{i+1}^{k_i} \label{eqn:not 2} \\
&\langle \bk,\bk' \rangle = \sum_{i=1}^n k_i k_i' - k_ik'_{i+1} \label{eqn:not 3}
\end{align}
for all $\bk,\bk' \in \nn$, and $\CL_i = \text{Ker}(\CV_i^+ \rightarrow \CV_i^-)$ denotes the rank $d^+_i - d^-_i$ tautological quotient bundle on $\fW_{\bd^+, \bd^-}$, which parametrizes the vector spaces $L_i$. Let us abbreviate:
\begin{equation}
\label{eqn:rum}
\fW_\bk = \bigsqcup_{\bd^+ - \bd^- = \bk} \fW_{\bd^+, \bd^-}, \qquad \fW = \bigsqcup_{\bk \in \nn} \fW_{\bk}
\end{equation}
Take the projections $\fW_{\bk} \stackrel{\pi^\pm}\longrightarrow \CM_{\bd^\pm}$ induced by \eqref{eqn:graham}, and construct the operators:
\begin{equation}
\label{eqn:opop}
g_{\pm \bk} : K \rightarrow K, \qquad \alpha \mapsto \pi^{\pm}_*\Big([\fW_{\bk}^\pm] \cdot \pi^{\mp *}(\alpha)\Big) 
\end{equation}
To work with the operators \eqref{eqn:opop}, we need to compute the tangent bundle to $\fW$. \\

\begin{proposition}
\label{prop:tangentw}

The virtual tangent bundle to $\fW$ is given by:
\begin{multline}
[T \fW] - [T\CM_{\bd^+}] = \sum_{i=1}^{n} \left[ \left(1 - \frac 1{q^2}\right) \left( \frac {\CV_i^+}{\CL_i} - \frac {\CV_{i}^+}{\CL_{i-1}} \right) + \frac {\CL_i}{\CL_{i-1}} - \frac {\CL_i}{\CL_i} - \frac{u_{i+1}^2}{\CL_iq^2} \right] \label{eqn:form5}
\end{multline}
or equivalently:
\begin{multline}
[T \fW] - [T\CM_{\bd^-}] = \sum_{i=1}^{n} \left[ \left(1 - \frac 1{q^2}\right) \left( \frac {\CL_{i+1}}{\CV_{i}^-} - \frac {\CL_i}{\CV_i^-} \right) + \frac {\CL_i}{\CL_{i-1}} - \frac {\CL_i}{\CL_i} + \frac {\CL_i}{u_i^2} \right] \label{eqn:form6}
\end{multline}

\end{proposition}

\subsection{}\label{sub:syt}

Let us now compute the operators \eqref{eqn:opop1}, \eqref{eqn:opop2} and \eqref{eqn:opop} in the basis of fixed points, as well as in the integral notation of Theorem \ref{thm:act}. Recall from Subsection \ref{sub:fixedpoints} that fixed points of $\CM_\bd$ are indexed by $n-$tuples of partitions $\bla$ as in \eqref{eqn:tuple}. As explained in \cite{FFNR}, fixed points of the simple correspondence $\fZ_i = \fZ_{[i;i+1)} = \ofZ_{[i;i+1)}$ are pairs of such $n$--tuples:
$$
\bla \geq_i \bmu \qquad \text{by which we mean that} \qquad \bla \geq \bmu \text{ and } |\blamu| = \bs^i 
$$
In other words, $\bla$ is obtained from $\bmu$ by adding a box of color $i$. Iterating this logic shows us that fixed points of the fine correspondence $\fZ_{[i;j)}$ are given by collections:
\begin{equation}
\label{eqn:sytflag}
\bla = \bnu^{j} \geq_{j-1} \bnu^{j-1} \geq_{j-2} ... \geq_{i+1} \bnu^{i+1} \geq_i \bnu^i = \bmu
\end{equation}
where $\bnu^{j-1},...,\bnu^{i+1}$ are all $n-$tuples of Young diagrams. \\

\begin{definition}
\label{def:syt}

A {\bf standard Young tableau} (abbreviated $\syt$) of shape $\blamu$ is a way to label the boxes of $\blamu$ with the integers $i,...,j-1$:
\begin{equation}
\label{eqn:syt}
\blamu = \Big\{ \sq_i, ..., \sq_{j-1} \Big\} 
\end{equation}
such that the labels match the colors of the boxes modulo $n$, and within each of the constituent skew partitions of $\blamu$, the labels increase as one goes up or to the right. \\

\end{definition}

\noindent It is easy to see that a SYT is the same datum as the flag \eqref{eqn:sytflag}, so they also parametrize fixed points of the fine correspondences $\fZ_{[i;j)}$. Then the restriction of the tautological line bundles of Definition \ref{def:line} to such a fixed point are given by:
\begin{equation}
\label{eqn:tautline}
\CL_k|_{\text{SYT of shape }\blamu} = \chi_k := \chi_{\sq_k}
\end{equation}
where the box $\sq_k$ is assumed to have color $k$ and weight given by formula \eqref{eqn:weight}. We will refer to $\bla$ (resp. $\bmu$) as the upper (resp. lower) shape of the given SYT. \\

\subsection{}\label{sub:asyt} 

As for the fixed points of the eccentric correspondences $\ofZ_{[i;j)}$ of \eqref{eqn:eccentric}, there is a description similar to the previous Subsection. One must replace the flag \eqref{eqn:funk} by the flag \eqref{eqn:zeit} and the two bullets in Subsection \ref{sub:fine} by the two bullets of Subsection \ref{sub:eccentric}. This leads to the following definition: \\

\begin{definition}
\label{def:asyt}
	
An {\bf almost standard Young tableau} (abbreviated $\asyt$) of shape $\blamu$ is a way to label the boxes of $\blamu$ with the integers $i,...,j-1$:
\begin{equation}
\label{eqn:asyt}
\blamu = \Big\{ \sq_i, ..., \sq_{j-1} \Big\} 
\end{equation}
such that the labels match the colors of the boxes modulo $n$, and within each of the constituent skew partitions of $\blamu$, the labels decrease as one goes up or to the right, with the only possible exception that:
\begin{equation}
\label{eqn:exception}
\sq_{a} \text{ is allowed to be directly above } \sq_{a-1}
\end{equation}
for all $a\in \{i+1,...,j-1\}$. \\
	
\end{definition}

\noindent The exception \eqref{eqn:exception} arises from the fact that the $\bY$ maps ``almost" preserve the flag \eqref{eqn:zeit}, as in \eqref{eqn:almost}. This is the reason why one cannot build a full flag of intermediate partitions analogous to \eqref{eqn:sytflag} from the datum of an ASYT. However, it does make sense to construct a partial flag of intermediate partitions:
\begin{equation}
\label{eqn:asytflag}
\bla = \bnu^0 \geq \bnu^{1} \geq ... \geq \bnu^{t-1} \geq \bnu^t = \bmu
\end{equation}
where $\bnu^{s-1} \backslash \bnu^{s}$ is a vertical strip of boxes of colors $k_{s-1},k_{s-1}+1,...,k_s-1$ for some: 
\begin{equation}
\label{eqn:asytflag2}
i = k_0 < k_{1} < ... < k_{t-1} < k_t = j
\end{equation}
The data \eqref{eqn:asytflag} and \eqref{eqn:asytflag2} determine an ASYT completely. \\

\noindent Finally, we must describe fixed points of the smooth correspondences $\fW$ of Subsection \ref{sub:new}. As in \eqref{eqn:graham}, such a fixed point consists of two $n-$tuples of partitions $\bla$ and $\bmu$. Since the corresponding parabolic sheaves must be contained inside each other, we must have $\bla \geq \bmu$. The condition that the $\bX$ maps vanish on the kernel of \eqref{eqn:twostep} is equivalent to no box of $\blamu$ being immediately to the right of any other box. This means that:
\begin{equation}
\label{eqn:strips}
\blamu = S_1 \sqcup S_2 \sqcup ... \sqcup S_t
\end{equation}
where every $S_i$ is a vertical strip of boxes, and no two strips have a common edge. The strips in \eqref{eqn:strips} are unordered, as opposed from the setup of (almost) standard Young tableaux. To summarize,
\begin{equation}
\label{eqn:strip}
\fW^T = \Big \{(\bla \geq \bmu) \text{ s.t. } \blamu \text{ is a union of strips \eqref{eqn:strips}} \Big \}
\end{equation}

\subsection{}\label{sub:push}

We will use the description of the fixed points of $\fZ_{[i;j)}$, $\ofZ_{[i;j)}$ and $\fW$ to prove formulas for the push-forward of classes under the projection maps $p^\pm$, $\op^\pm$ and $\pi^\pm$ from each of these varieties to $\CM_{\bd^\pm}$. The subsequent Propositions will be key to computing the matrix coefficients of the operators \eqref{eqn:opop1}, \eqref{eqn:opop2} and \eqref{eqn:opop}. Recall the normal-ordered integral of Definition \ref{def:normal}. \\

\begin{proposition}
\label{prop:pushsyt}

For any Laurent polynomial $M(z_i,...,z_{j-1})$ with coefficients pulled-back from $K_{T}(\CM_{\bd^\mp})$, we have the following push-forward formula:
\begin{multline}
p^\pm_*\Big( M \left(\CL_i,...,\CL_{j-1}\right) [\fZ_{[i;j)}^\pm] \Big) = \\ = \int^\pm \frac {M(z_i,...,z_{j-1}) \prod_{a=i}^{j-1} \left[ \overline{\zeta\left(\frac {z_a^{\pm 1}}{X^{\pm 1}} \right)} \tau_\pm(z_a)  \right]^{\pm 1} Dz_a}{q^{(j-i)\delta_\pm^-} \prod_{a=i+1}^{j-1} \left(1 - \frac {z_a}{z_{a-1}q^2}\right)\prod_{i \leq a < b < j} \zeta \left( \frac {z_a}{z_b} \right)} \label{eqn:pushsyt}
\end{multline}
where the adjusted virtual class $[\fZ^\pm_{[i;j)}]$ was defined in equations \eqref{eqn:adjust1}--\eqref{eqn:adjust2}. \\

\end{proposition}

\begin{proposition}
\label{prop:pushasyt}

For any Laurent polynomial $M(z_i,...,z_{j-1})$ with coefficients pulled-back from $K_{T}(\CM_{\bd^\mp})$, we have the following push-forward formula:
\begin{multline}
\op^\pm_*\Big( M \left(\CL_i,...,\CL_{j-1}\right) [\ofZ_{[i;j)}^\pm] \Big) = \\ = \int^\pm \frac {M(z_i,...,z_{j-1}) \prod_{a=i}^{j-1} \left[ \overline{\zeta\left(\frac {z_a^{\pm 1}}{X^{\pm 1}} \right)} \tau_\pm(z_a) \right]^{\pm 1} Dz_a }{q^{(j-i)\delta_\pm^-} \prod_{a=i+1}^{j-1} \left(1 - \frac {z_{a-1}}{z_a} \right)\prod_{i \leq a < b < j} \zeta \left( \frac {z_b}{z_a} \right)} \label{eqn:pushasyt}
\end{multline}
where the adjusted virtual class $[\ofZ^\pm_{[i;j)}]$ was defined in equations \eqref{eqn:adjust3}--\eqref{eqn:adjust4}. \\

\end{proposition}

\noindent In the subsequent Propositions, the $...$ in the left and right-hand sides will refer to the $K$--theory class and the integrand (respectively) in the left and right-hand sides (respectively) of either formula \eqref{eqn:pushsyt} or \eqref{eqn:pushasyt}. \\

\begin{proposition}
\label{prop:pushsyt change}

Formula \eqref{eqn:pushsyt} implies the following:
\begin{equation}
\label{eqn:contour 1}  
p_*^+ \left(... \ [\fZ_{[i;j)}^+] \right) = \int_{z_{j-1} \prec ... \prec z_i \prec \{0,\infty\}} ...
\end{equation}
\begin{equation}
\label{eqn:contour 2}
p_*^- \left(... \ [\fZ_{[i;j)}^-] \right) = \int_{z_{i} \prec ... \prec z_{j-1} \prec \{0,\infty\}} ...
\end{equation}
where the notation $z_{j-1} \prec ... \prec z_i \prec \{0,\infty\}$ means that the variables $z_{j-1}$,...,$z_i$ are integrated over successive concentric contours, with $z_{j-1}$ farthest from and $z_i$ closest to $0$ and $\infty$. The distance between any two contours, as well as between the contours and $0$,$\infty$, must be $\gg$ the sizes of all the equivariant parameters. \\

\end{proposition}

\begin{proposition}
\label{prop:pushasyt change}

Formula \eqref{eqn:pushasyt} implies the following:
\begin{multline}
\op_*^+ \left(... \ [\ofZ_{[i;j)}^+] \right) = \sum^{t\geq 1}_{i = k_0 < k_1 < ... < k_t = j} \\ \int_{z_i = ... = z_{k_1-1} \prec ... \prec z_{k_{t-1}} = ... = z_{j-1} \prec \{0,\infty\}} (1-1)^{j-i-t} ... \label{eqn:contour 3} 
\end{multline}
\begin{multline}
\op_*^- \left(... \ [\ofZ_{[i;j)}^-] \right) = \sum^{t\geq 1}_{i = k_0 < k_1 < ... < k_t = j} \\ \int_{z_{j-1} = ... = z_{k_{t-1}}  \prec ... \prec z_{k_1-1} = ... = z_{i} \prec \{0,\infty\}} (1-1)^{j-i-t} ... \label{eqn:contour 4} 
\end{multline}
In the right-hand sides of \eqref{eqn:contour 3} and \eqref{eqn:contour 4}, we specialize the variables $z_i,...,z_{k_1-1}$ to a common value $y_1$, the variables $z_{k_1},...,z_{k_2-1}$ to another common value $y_2$, ..., and the variables $z_{k_{t-1}},...,z_{j-1}$ to a common value $y_t$. As one performs this specialization, the integrand in the right-hand side of \eqref{eqn:pushasyt} would gain $j-i-t$ factors of $(1-1)$ in the denominator, stemming from $1-\frac {z_{a-1}}{z_a}$ when both $z_{a-1}$ and $z_a$ are specialized to the same $y_i$. These factors must be eliminated in order for the right-hand sides of \eqref{eqn:contour 3} and \eqref{eqn:contour 4} to make sense, which is accomplished by the factor $(1-1)^{j-i-t}$. Once one takes this into account, the second rows of formulas \eqref{eqn:contour 3}--\eqref{eqn:contour 4} are contour integrals in the variables $y_1,...,y_t$, with the contours being very far away from each other and from $0,\infty$. \\

\end{proposition}

\subsection{} As for the correspondence $\fW_\bk$ of Subsection \ref{sub:new}, one has the tautological rank $k_i$ vector bundle $\CL_i$, $\forall i \in \{1,...,n\}$. Let us formally write $\CL_i = x_{i1}+...+x_{ik_i}$. For any color-symmetric Laurent polynomial $M$ in $\bk = (k_1,...,k_n)$ variables, define:
\begin{equation}
\label{eqn:M}
M(...,\CL_i,...) := M(..., x_{i1},...,x_{ik_i},...) \in K_{T}(\fW_\bk)
\end{equation}
Then the analogoue of Proposition \ref{prop:pushsyt} is the following result. \\

\begin{proposition}
\label{prop:pushsmooth}	

For any color-symmetric Laurent polynomial $M$ as in \eqref{eqn:M} with coefficients pulled back from $K_{T}(\CM_{\bd^\mp})$, we have:
\begin{multline}
\pi^\pm_*\Big( M(...,\CL_i,...) [\fW^\pm_\bk] \Big) = \frac 1{\bk!} \int^\pm M(...,z_{ia},...) \\ \frac {\prod^{1 \leq i \leq n}_{1 \leq a \neq b \leq k} \left( 1 - \frac {z_{ib}}{z_{ia}} \right)}{\prod^{1 \leq i \leq n}_{1 \leq a,b \leq k} \left( 1 - \frac {z_{i-1,b}}{z_{ia}} \right)} \prod^{1 \leq i \leq n}_{1 \leq a \leq k_i} \left[ \overline{\zeta \left( \frac {z^{\pm 1}_{ia}}{X^{\pm 1}} \right)} \tau_\pm(z_{ia}) \right] ^{\pm 1} Dz_{ia} \label{eqn:pushsmooth}
\end{multline}
where the adjusted virtual class $[\fW^\pm_\bk]$  was defined in equations \eqref{eqn:adjust5}--\eqref{eqn:adjust6}. \\

\end{proposition}

\begin{proposition}
\label{prop:pushsmooth change}	

Formula \eqref{eqn:pushsmooth} implies the following:
\begin{multline}
\pi^\pm_*\Big( M(...,\CL_i,...) [\fW^\pm_\bk] \Big) = \sum^{\text{decompositions}}_{\text{as in \eqref{eqn:turner}}} \frac 1{\#_{B_1,...,B_t}} \int_{y_1 \prec ... \prec y_t \prec \{0,\infty\}} M(...,z,...) (\pm 1)^{|\bk|-t} \\   \prod_{s,s'=1}^t \prod_{z' \in B_{s'}}^{z\in B_s} \left(1 - \frac {z}{z'} \right)^{\delta^{\col z}_{\col z'} - \delta^{\col z}_{\col z'-1}} \prod_{1\leq s \leq t}^{z\in B_s} \left[ \overline{\zeta \left( \frac {z^{\pm 1}}{X^{\pm 1}} \right)} \tau_\pm(z) \right]^{\pm 1} \Big|^{z \mapsto y_s}_{\text{if }z \in B_s} \prod_{s=1}^t Dy_s \qquad \label{eqn:contours smooth}
\end{multline}
where the number $\#_{B_1,...,B_t}$ is defined in \eqref{eqn:diez}. Each summand in \eqref{eqn:contours smooth} corresponds to a way to partition the set of variables $\{z_{ia}\}^{1\leq i \leq n}_{1\leq a \leq k_i}$ into sets $B_s$ of variables of consecutive colors, then specialize all the $z_{ia}$ from a given set $B_s$ to the same value $y_s$, and finally evaluate the corresponding integral when $y_1,...,y_t$ are far apart. \\
	
\end{proposition}

\begin{proof} {\bf of Theorem \ref{thm:geom}:} We will prove the required statement for the shuffle element $S_m^\pm$, since the other cases are analogous. By \eqref{eqn:opop1}, we have:
\begin{equation}
\label{eqn:str}
s_m^\pm \cdot \of \ = \ p^\pm_*\Big( m(\CL_i,...,\CL_{j-1}) \cdot [\fZ^\pm_{[i;j)}] \cdot f(X_\mp) \Big)
\end{equation}
where we use the symbols $X_+$ and $X_-$ as placeholders for Chern classes of tautological bundles, pulled-back to $\fZ_{[i;j)}$ from $\CM_{\bd^+}$ and $\CM_{\bd^-}$, respectively. We have:
$$
X_+|_{\fZ_{[i;j)}} = X_-|_{\fZ_{[i;j)}} + \CL_i + ... + \CL_{j-1}
$$
in the sense of the plethysm \eqref{eqn:pleth1}. Therefore, \eqref{eqn:str} becomes:
$$
s^\pm_m \cdot \of = p^\pm_*\Big( m(\CL_i,...,\CL_{j-1}) \cdot [\fZ^\pm_{[i;j)}] \cdot f(X_\pm \mp(\CL_i + ... + \CL_{j-1})) \Big)
$$
Because the variables $X_\pm$ are pulled-back under $p^\pm$, they pass unhindered through the push-forward $p^\pm_*$. Then we may use Proposition \ref{prop:pushsyt} to obtain:
$$
s^\pm_m \cdot \of  = \int^\pm \frac {m(z_i,...,z_{j-1}) \cdot \overline{f(X \mp z_i \mp ... \mp z_{j-1}) \prod_{a=i}^{j-1} \Big[ \zeta \left( \frac {z_a^{\pm 1}}{X^{\pm 1}} \right)} \tau_\pm(z_a)  \Big]^{\pm 1} Dz_a }{q^{(j-i)\delta_\pm^-} \prod_{a=i+1}^{j-1} \left(1 - \frac {z_a}{z_{a-1}q^2}\right)\prod_{i \leq a < b < j} \zeta \left( \frac {z_a}{z_b} \right)} 
$$
for all $f \in \Lambda$. Comparing the above formula with the shuffle elements \eqref{eqn:s} and Theorem \ref{thm:act}, we obtain the following equality of operators on $K$:
\begin{equation}
\label{eqn:S}
s_m^\pm = S_{m}^\pm
\end{equation}
The analogous computations for $\ofZ_{[i;j)}$ and $\fW_\bk$ gives us $t_m^\pm = T_m^\pm$ and $g_{\pm \bk} = G_{\pm \bk}$. \\
\end{proof}

\subsection{}\label{sub:ext}

In the remainder of this Section, we will study the analogue of the Carlsson-Okounkov vector bundle (\cite{CO}). In the setup of affine Laumon spaces, it was defined in \cite{FFNR} as the following rank $|\bd|+|\bd'|$ vector bundle $\CE$ on $\CM_{\bd} \times \CM_{\bd'}$, for any pair of degree vectors $\bd, \bd' \in \nn$:
\begin{equation}
\label{eqn:defe}
\CE|_{\CF_\bullet, \CF'_\bullet} = \textrm{Ext}^1(\CF'_\bullet,\CF_\bullet(-\infty))
\end{equation}
The notion of Ext space of flags of sheaves is defined in \cite{FFNR}, where it is explained why the twist by $\infty$ forces the corresponding $\text{Ext}^0$ and $\text{Ext}^2$ spaces to vanish. The following result is proved in \loccitt: \\

\begin{proposition} 
\label{prop:section}
	
The vector bundle $\CE$ has a section $s$ which vanishes on:
\begin{equation}
\label{eqn:locus}
\fC = \Big\{ (\CF_\bullet \supset \CF'_\bullet) \Big\} \subset \CM_\bd \times \CM_{\bd'}
\end{equation}
set-theoretically, but not necessarily scheme-theoretically. \\
\end{proposition}

\begin{remark}
\label{rem:section}
	
When $\bd' = \bd$, the locus $\fC$ is the diagonal. \\
	
\end{remark}

\begin{remark}
\label{rem:hecke}
	
When $\bd' =\bd+ \bs^i$ for some $1 \leq i \leq n$, the locus $\fC$ coincides with $\fZ_i$. We conclude that $\fC$ is smooth and middle-dimensional inside $\CM_\bd \times \CM_{\bd +\bs^i}$. \\
	
\end{remark}

\begin{remark}
\label{rem:vanishing}
	
If $\bd' - \bd \notin \nn$, the locus $\fC$ is empty, and hence the vector bundle $\CE$ has a nowhere vanishing section. This implies the exterior power of $\CE$ has class 0:
\begin{equation}
\label{eqn:vanishing}
[\wedge^\bullet (\CE^\vee)] := \sum_{k=0}^{\text{rank }\CE} (-1)^k [\wedge^ k \CE^\vee] = 0 \in K_\bd \otimes K_{\bd'}
\end{equation}
	
\end{remark}

\noindent In general, we need to compute the $K-$theory class of the vector bundle $\CE$ inside the ring $K_\bd \otimes K_{\bd'}$. We will do so in terms of the tautological vector bundles $\CV_k$ and $\CV_k'$ pulled-back from the two factors, and the formula we obtain is:
\begin{equation}
\label{eqn:classe}
[\CE] = \left(1 - \frac 1{q^2}\right) \sum_{k=1}^{n} \left(\frac {\CV_k}{\CV'_{k-1}} - \frac {\CV_k}{\CV_k'}\right) + \sum_{k=1}^{n} \left( \frac {\CV_k}{u_k^2} + \frac {u_{k+1}^2}{\CV'_kq^2} \right)
\end{equation}
Together with \eqref{eqn:tangentspace}, the above formula may be rewritten as:
\begin{equation}
\label{eqn:aret}
[T\CM_{\bd'}] - [\CE] = \left(1 - \frac 1{q^2} \right) \sum_{k=1}^{n} \left(\frac 1{\CV_{k-1}'} - \frac 1{\CV_k'} \right) (\CV'_k - \CV_k) + \sum_{k=1}^{n} \frac {\CV'_k - \CV_k}{u_k^2} \qquad
\end{equation}
If we restrict the formula above to the diagonal (i.e. $\CV_k = \CV_k'$), then we obtain 0, as one would expect from Remark \ref{rem:section}. If we restrict it to the simple correspondence $\fZ_i \subset \CM_{\bd} \times \CM_{\bd'}$ (where we have $\CV_k' = \CV_k + \delta_i^k \cdot \CL_i$), then the right-hand side of \eqref{eqn:aret} matches \eqref{eqn:form1}, as one would expect from Remark \ref{rem:hecke}. \\

\subsection{} \label{sub:compute}

As a consequence of \eqref{eqn:classe}, we infer that the exterior class of $\CE^\vee$ is given by:
$$
\left [ \wedge^\bullet (\CE^\vee) \right] = \prod_{k=1}^{n} \frac {\left(1-\frac {\CV_k' q^2}{\CV_k}\right)\left(1-\frac {\CV'_{k-1}}{\CV_k}\right)}{\left(1-\frac {\CV_k'}{\CV_k}\right)\left(1-\frac {\CV'_{k-1}q^2}{\CV_k}\right)} \left(1 -  \frac {u_k^2}{\CV_k} \right) \left(1 - \frac {\CV'_kq^2}{u_{k+1}^2} \right) 
$$
We will work with the following adjustment of the above exterior class:
\begin{equation}
\label{eqn:adjust}
[\widetilde{\wedge}^\bullet (\CE^\vee)] := [\wedge^\bullet (\CE^\vee)] \cdot (-1)^{|\bd'| - |\bd|} u_{\bd' - \bd} q^{\langle \bd', \bd' - \bd \rangle} \prod_{i=1}^n \frac {\det \CV_i}{\det \CV_i'}
\end{equation}
on $\CM_\bd \times \CM_{\bd'}$. Let $p_1$ and $p_2$ be the projections from $\CM_\bd \times \CM_{\bd'}$ to the two factors, and define the operators:
\begin{equation}
\label{eqn:op}
a_{\bd, \bd'} : K_{\bd'} \longrightarrow K_{\bd}, \qquad \alpha \mapsto p_{2*} \Big( [\widetilde{\wedge}^\bullet (\CE^\vee)] \cdot p_1^*(\alpha) \Big)
\end{equation}
According to \eqref{eqn:vanishing}, $a_{\bd, \bd'} = 0$ unless $\bd' - \bd \in \nn$. Therefore, we will write: 
$$
a_{\bk} = \bigoplus_{\bd \in \nn} a_{\bd, \bd + \bk}, \qquad a = \bigoplus_{\bk \in \nn} a_{\bk}
$$ 
as an endomorphism of $K$, and prove the following: \\

\begin{proposition}
\label{prop:a}

The operator $K \stackrel{a_{\bk}}\longrightarrow K$ coincides with the action of the element:
\begin{equation}
\label{eqn:one}
(1 - q^{-2})^{-|\bk|} \in \CA_{-\bk} \curvearrowright K
\end{equation}
Above, $(1 - q^{-2})^{-|\bk|}$ is a constant rational function in $(k_1,...,k_n)$ variables, which is an element of $\CA^-$ as in Definition \ref{def:shuf}, and the action $\curvearrowright$ is the one of Theorem \ref{thm:act}. \\

\end{proposition}

\begin{proof} If $\bla \not \geq \bmu$, then Remark \eqref{eqn:vanishing} implies that $\langle \bmu | a | \bla \rangle = 0$. Otherwise:
$$
\langle \bmu | a | \bla \rangle = \frac {\left [ \widetilde{\wedge}^\bullet (\CE^\vee_{\bmu,\bla}) \right]}{[\wedge^\bullet(T^\vee_\bla\CM_{\bd'})]} = \frac {(-1)^{|\bd'| - |\bd|} u_{\bd' - \bd} q^{\langle \bd', \bd' - \bd \rangle} \prod_{i=1}^n \frac {\det \CV_i}{\det \CV_i'}}{\wedge^\bullet(T^\vee_{\bla} \CM_{\bd'}  - \CE^\vee_{\bmu,\bla})} 
$$
Using formula \eqref{eqn:aret} to compute the denominator, we conclude that:
$$
\langle \bmu | a | \bla \rangle = \prod_{\bsq \in \blamu} \left[ \zeta \left( \frac {\chi_\bla}{\chi_\bsq} \right) \tau_-(\chi_\bsq) \right]^{-1}
$$
Comparing the above with \eqref{eqn:coeff-} gives us the required equality of operators. 

\end{proof}

\section{Appendix}
\label{sec:app}

\begin{proof} {\bf of Proposition \ref{prop:tangent}:} The proof is mostly based on the following claim: \\
	
\begin{claim}
\label{claim:croissant}

Consider vector spaces $V^+_i \twoheadrightarrow V^-_i$ equipped with a flag of subspaces:
\begin{equation}
\label{eqn:paratu}
0 = U^0_i \subset U^1_i \subset ... \subset U^k_i = \text{Ker} \left( V^+_i \twoheadrightarrow V^-_i \right)
\end{equation}
for $i \in \{1,2\}$. The vector space $W_k$ of homomorphisms $V^+_1 \rightarrow V^+_2$ which preserve the flag \eqref{eqn:paratu} can be described as:
\begin{equation}
\label{eqn:plus}
[W_k] \ = \ \frac {V_2^+}{V_1^+} - \sum_{a=1}^k \frac {V_2^+}{L_1^a} + \sum_{1\leq a \leq b \leq k} \frac {L_2^a}{L_1^b} 
\end{equation}
where $L^a_i = U_i^a/U_i^{a-1}$. Formula \eqref{eqn:plus} holds in the Grothendieck group, i.e. when we identify $V$ with $V_1+V_2$ for any short exact sequence $0 \rightarrow V_1 \rightarrow V \rightarrow V_2 \rightarrow 0$. \\

\end{claim}

\begin{proof} The flag \eqref{eqn:paratu} is equivalent to a flag of quotients:
$$
V_i^+ = V_i^0 \twoheadrightarrow V_i^{1} \twoheadrightarrow ... \twoheadrightarrow V_i^{k-1} \twoheadrightarrow V_i^k = V_i^-
$$
where $L_i^a = \text{Ker}(V_i^{a-1} \twoheadrightarrow V_i^{a})$. Consider the vector space $W_a$ of homomorphisms $V_1^+ \rightarrow V_2^+$ which preserve the truncated flags $V_i^+ = V_i^0 \twoheadrightarrow ... \twoheadrightarrow V_i^a$. We have:
\begin{equation}
\label{eqn:endo}
W_0 := \Hom \left( V_1^+,V_2^+ \right) = (V_1^+)^\vee \otimes V_2^+ = \frac {V_2^+}{V_1^+}
\end{equation}
where the last equality is simply a matter of notation for us. In general, we have:
$$
W_{a} = \text{Ker}\Big(W_{a-1} \rightarrow \Hom \left(L_1^a,V_2^{a} \right) \Big)
$$
and so we have the following equality in the Grothendieck group:
$$
W_{a} = W_{a-1} - \frac {V_2^a}{L_1^a} = W_{a-1} - \frac {V_2^+}{L_1^a} + \sum_{a' \leq a} \frac {L_2^{a'}}{L_1^a}
$$
Iterating the preceding relation for $a \in \{1,...,k\}$ gives us formula \eqref{eqn:plus}. \\
\end{proof}

\noindent Recall that the affine space $Z_{[i;j)}$ parametrizes quadruples of linear maps $\bX,\bY,\bA,\bB$ involving the $n-$tuple of vector spaces $\bV^+ = (V_1^+,...,V_{n}^+)$ and $\oplus_{k=1}^n \BC w_k$, which interact with the flag \eqref{eqn:funk} as in the two bullets of Subsection \ref{sub:fine}. Applying Claim \ref{claim:croissant}, we see that:
\begin{equation}
\label{eqn:buddha1}
\left[Z_{[i;j)}\right] = \left(\sum_{k=1}^{n} \frac {\CV_k^+}{\CV_k^+ q^2} - \sum_{a=i}^{j-1} \frac {\CV_a^+}{\CL_a q^2} + \sum_{i\leq a < b < j}^{a \equiv b} \frac {\CL_b}{\CL_a q^2} \right) + 
\end{equation}
$$
+ \left( \sum_{k=1}^{n} \frac {\CV_k^+}{\CV^+_{k-1}} - \sum_{a=i}^{j-1} \frac {\CV_{a+1}^+}{\CL_a} + \sum_{i\leq a < b < j}^{a+1 \equiv b} \frac {\CL_b}{\CL_a} \right) + \sum_{k=1}^{n} \frac {\CV_k^+}{u_k^2} + \left( \sum_{k=1}^{n} \frac {u_{k+1}^2}{\CV_k^+ q^2} - \sum_{a=i}^{j-1} \frac {u_{a+1}^2}{\CL_a q^2} \right) 
$$
where the 4 parentheses keep track of the contributions of the maps $\bX$, $\bY$, $\bA$, $\bB$, respectively. The codomain of the map $\eta$ and the Lie algebra of $P_{[i;j)}$ contribute:
\begin{align}
&[\fa_{[i;j)}] = \sum_{k=1}^{n} \frac {\CV_k^+}{\CV_{k-1}^+q^2} - \sum_{a=i}^{j-1} \frac {\CV_{a+1}^+}{\CL_a q^2} + \sum_{i\leq a < b - 1 < j - 1}^{b\equiv a+1} \frac {\CL_b}{\CL_a q^2}  \label{eqn:buddha2} \\
&[\fp_{[i;j)}] = \sum_{k=1}^{n} \frac {\CV_k^+}{\CV_k^+} - \sum_{a=i}^{j-1} \frac {\CV_a^+}{\CL_a} + \sum_{i\leq a \leq b < j}^{b\equiv a} \frac {\CL_b}{\CL_a} \label{eqn:buddha3}
\end{align}
By formula \eqref{eqn:vir fine}, the class of the (virtual) tangent space to $\fZ_{[i;j)}$ is:
$$
\Big(\text{RHS of \eqref{eqn:buddha1}}\Big) - \Big(\text{RHS of \eqref{eqn:buddha2}}\Big) - \Big(\text{RHS of \eqref{eqn:buddha3}}\Big)
$$
In the right-hand sides of \eqref{eqn:buddha1}, \eqref{eqn:buddha2}, \eqref{eqn:buddha3}, the summands which only involve $\CV_k^+$ contribute precisely the class of the tangent space to $\CM_{\bd^+}$, by \eqref{eqn:tangentspace}. The remaining summands precisely constitute the RHS of \eqref{eqn:form1}, as we needed to prove. \\

\noindent Recall that the affine space $\oZ_{[i;j)}$ parametrizes quadruples of linear maps $\bX,\bY,\bA,\bB$ involving the $n-$tuple of vector spaces $\bV^+ = (V_1^+,...,V_{n}^+)$ and $\oplus_{k=1}^n \BC w_k$, which interact with the flag \eqref{eqn:zeit} as in the two bullets of Subsection \ref{sub:eccentric}. Applying Claim \ref{claim:croissant}, we see that:
\begin{equation}
\label{eqn:bar1}
\left[\oZ_{[i;j)}\right] = \left(
\sum_{k=1}^{n} \frac {\CV_k^+}{\CV_k^+ q^2} - \sum_{a=i}^{j-1} \frac {\CV_a^+}{\CL_a q^2} + \sum_{i\leq a < b < j}^{a \equiv b} \frac {\CL_a}{\CL_b q^2} \right) + 
\end{equation}
$$
+ \left( \sum_{k=1}^{n} \frac {\CV_k^+}{\CV^+_{k-1}} - \sum_{a=i}^{j-1} \frac {\CV_{a+1}^+}{\CL_a} + \sum_{i\leq a \leq b + 1 < j + 1}^{a-1 \equiv b} \frac {\CL_a}{\CL_b} \right) + \sum_{k=1}^{n} \frac {\CV_k^+}{u_k^2} + \left( \sum_{k=1}^{n} \frac {u_{k+1}^2}{\CV_k^+ q^2} - \sum_{a=i}^{j-1} \frac {u_{a+1}^2}{\CL_a q^2} \right)
$$
where the 4 parentheses keep track of the contributions of the maps $\bX$, $\bY$, $\bA$, $\bB$, respectively. The codomain of the map $\oeta$ and the Lie algebra of $\oP_{[i;j)}$ contribute:
\begin{align}
&[\overline{\fa}_{[i;j)}] = \sum_{k=1}^{n} \frac {\CV_k^+}{\CV_{k-1}^+q^2} - \sum_{a=i}^{j-1} \frac {\CV_{a+1}^+}{\CL_a q^2} + \sum_{i\leq a < b < j}^{a-1 \equiv b} \frac {\CL_a}{\CL_b q^2}  \label{eqn:bar2} \\
&[\overline{\fp}_{[i;j)}] = \sum_{k=1}^{n} \frac {\CV_k^+}{\CV_k^+} - \sum_{a=i}^{j-1} \frac {\CV_a^+}{\CL_a} + \sum_{i\leq a \leq b < j}^{a \equiv b} \frac {\CL_a}{\CL_b} \label{eqn:bar3}
\end{align}
By formula \eqref{eqn:vir eccentric}, the class of the (virtual) tangent space to $\ofZ_{[i;j)}$ is:
$$
\Big(\text{RHS of \eqref{eqn:bar1}}\Big) - \Big(\text{RHS of \eqref{eqn:bar2}}\Big) - \Big(\text{RHS of \eqref{eqn:bar3}}\Big)
$$
In the right-hand sides of \eqref{eqn:bar1}, \eqref{eqn:bar2}, \eqref{eqn:bar3}, the summands which only involve $\CV_k^+$ contribute precisely the class of the tangent space to $\CM_{\bd^+}$, by \eqref{eqn:tangentspace}. The remaining summands precisely constitute the RHS of \eqref{eqn:form3}, as we needed to prove. \\
	
\end{proof}

\begin{proof} {\bf of Proposition \ref{prop:tangentw}:} We will use formula \eqref{eqn:vir smooth}, and to do so, we will need to compute the class of the vector space $W_{\bd^+,\bd^-}$ in the Grothendieck group. Claim \ref{claim:croissant} gives us:
$$
[W_{\bd^+,\bd^-}] = \sum_{i=1}^{n} \left( \frac {\CV^+_i}{\CV^+_iq^2} - \frac {\CV^+_i}{\CL_iq^2}  \right) + \sum_{i=1}^{n} \left( \frac {\CV^+_i}{\CV^+_{i-1}} -  \frac {\CV^+_i}{\CL_{i-1}} + \frac {\CL_i}{\CL_{i-1}} \right) + 
$$
\begin{equation}
\label{eqn:capra1}
+ \sum_{i=1}^{n}  \frac {\CV^+_i}{u_i^2}  + \sum_{i=1}^{n} \left( \frac {u_{i+1}^2}{\CV_i^+ q^2} - \frac {u_{i+1}^2}{\CL_i q^2} \right)
\end{equation}
The four sums in \eqref{eqn:capra1} correspond to the $\bX,\bY,\bA,\bB$ maps, respectively. Moreover:
\begin{align}
&[\fa] = \sum_{i=1}^{n} \left( \frac {\CV^+_i}{\CV^+_{i-1}q^2} -  \frac {\CV^+_i}{\CL_{i-1}q^2} \right) \label{eqn:capra2} \\
&[\fp] = \sum_{i=1}^{n} \left( \frac {\CV^+_i}{\CV^+_{i}} -  \frac {\CV^+_i}{\CL_{i}} + \frac {\CL_i}{\CL_{i}} \right) \label{eqn:capra3}
\end{align}
According to \eqref{eqn:vir smooth}, the difference:
$$
\Big(\text{RHS of \eqref{eqn:capra1}}\Big) - \Big(\text{RHS of \eqref{eqn:capra2}}\Big) - \Big(\text{RHS of \eqref{eqn:capra3}}\Big)
$$
gives us a formula for $[T\fW_{\bd^+,\bd^-}]$. In the right-hand sides of \eqref{eqn:capra1}, \eqref{eqn:capra2}, \eqref{eqn:capra3}, the terms that do not involve $\CL_i$ produce $[T\CM_{\bd^+}]$, according to formula \eqref{eqn:tangentspace}. The remaining terms are precisely the right-hand side of \eqref{eqn:form5}. Formula \eqref{eqn:form6} is proved similarly, so we leave it as an exercise to the interested reader. \\
\end{proof}

\begin{proof} {\bf of Propositions \ref{prop:pushsyt} and \ref{prop:pushsyt change}:} Let us first prove \eqref{eqn:pushsyt} when the sign is $\pm = +$. The virtual equivariant localization formula \eqref{eqn:eqloc vir} reads:
$$
\text{LHS of \eqref{eqn:pushsyt}}  = \sum_{\bla \vdash \bd^+} | \bla \rangle \sum^{\syt \text{ of upper}}_{\text{shape }\bla} \frac {M(\CL_i,...,\CL_{j-1})\cdot [\fZ_{[i;j)}^+]}{\wedge^\bullet(T^\vee \fZ_{[i;j)} - T^\vee \CM_{\bd^+})} \ \Big|_{\syt}
$$	
Recall from \eqref{eqn:tautline} that the restriction of $\CL_a$ to the fixed point corresponding to a SYT is just the weight $\chi_a$ of the box labeled $a$ in the SYT. We will use \eqref{eqn:adjust1} and \eqref{eqn:form1} to compute the adjusted virtual class in the numerator, respectively the exterior class of the normal bundle in the denominator. Therefore, $\text{LHS of \eqref{eqn:pushsyt}} |_\bla = $
$$
\sum^{\syt \text{ of upper}}_{\text{shape }\bla} M(\chi_i,...,\chi_{j-1}) u_{i+1}...u_j \frac {\chi_i}{\chi_{j-1}} (-1)^{j-i-1} \frac {q^{\left \lceil \frac {j-i}n \right \rceil - 2}}{q^{d_i - d_j}} \frac {(1-1)^{j-i}}{\prod_{a=i+1}^{j-1} \left(1- \frac {\chi_{a-1}q^2}{\chi_a}\right)}
$$
$$
\prod_{a < b}^{a \equiv b} \frac {1 - \frac {\chi_{a}}{\chi_{b}}}{1 - \frac {\chi_{a}q^2}{\chi_{b}}} \prod_{a < b}^{a \equiv b - 1} \frac {1 - \frac {\chi_{a}q^2}{\chi_{b}}}{1 - \frac {\chi_{a}}{\chi_{b}}} \prod_{a=i}^{j-1} \left[\left(1- \frac {\chi_aq^2}{u^2_{a+1}}\right) \prod^{\col \sq = a}_{\sq \in \bla} \frac {1 - \frac {\chi_aq^2}{\chi_\sq}}{1 - \frac {\chi_a}{\chi_\sq}} \prod^{\col \sq = a+1}_{\sq \in \bla} \frac {1 - \frac {\chi_a}{\chi_\sq}}{1 - \frac {\chi_a q^2}{\chi_\sq}} \right] 
$$
where $\bd = |\bla|$. The factor $(1-1)^{j-i}$ in the numerator of the first row is meant to cancel $j-i$ zeroes in the denominator of the second row. Explicitly, if we write a SYT in the form \eqref{eqn:sytflag}, then we have $\chi_\bla = \chi_{\bnu^a} + \chi_a + ... + \chi_{j-1}$, and so the formula above implies that $\text{LHS of \eqref{eqn:pushsyt}} |_\bla = $ 
$$
\sum^{\syt \text{ as in}}_{\eqref{eqn:sytflag}} M(\chi_i,...,\chi_{j-1}) u_{i+1}...u_j \frac {\chi_i}{\chi_{j-1}} (-1)^{j-i-1} \frac {q^{\left \lceil \frac {j-i}n \right \rceil - 2}}{q^{d_i - d_j}} \frac {\left(1 - q^2 \right)^{j-i}}{\prod_{a=i+1}^{j-1} \left(1- \frac {\chi_{a-1}q^2}{\chi_a}\right)}
$$
$$
\prod_{a=i}^{j-1} \left[\left(1- \frac {\chi_aq^2}{u^2_{a+1}}\right) \prod^{\col \sq = a}_{\sq \in \bnu^a} \frac {1 - \frac {\chi_aq^2}{\chi_\sq}}{1 - \frac {\chi_a}{\chi_\sq}} \prod^{\col \sq = a+1}_{\sq \in \bnu^a} \frac {1 - \frac {\chi_a}{\chi_\sq}}{1 - \frac {\chi_a q^2}{\chi_\sq}} \right] 
$$
Absorbing all the $u_a$'s, $\chi_a$'s and $q$'s into the second row, we obtain the formula:
\begin{equation}
\label{eqn:donald}
\text{LHS of \eqref{eqn:pushsyt}} \Big|_\bla = 
\end{equation}
$$
= \sum^{\syt \text{ as in}}_{\eqref{eqn:sytflag}} \frac {M(\chi_i,...,\chi_{j-1}) (q^{-1} - q)^{j-i} \prod_{a=i}^{j-1} \zeta\left(\frac {\chi_a}{\chi_{\bnu^a}} \right) \tau_+(\chi_a)}{\prod_{a=i+1}^{j-1} \left(1- \frac {\chi_a}{\chi_{a-1}q^2}\right)} 
$$
Recalling the notion $\int^+$ from Definition \ref{def:normal}, we note that the $\text{RHS of \eqref{eqn:pushsyt}} |_\bla = $
$$
\sum^{\text{functions}}_{\sigma : \{i,...,j-1\} \rightarrow \{1,-1\}}  \int^{|q|^{\pm 1}, |\oq|^{\pm 1} < 1}_{|z_{a}| = \gamma^{\sigma(a)}  \oq^{-\frac {2a}n}} \frac {M(z_i,...,z_{j-1}) \prod_{a=i}^{j-1} \zeta \left( \frac {z_a}{\chi_\bla} \right) \tau_+(z_a) \sigma(a) Dz_a}{\prod_{a=i+1}^{j-1} \left(1 - \frac {z_a}{z_{a-1}q^2}\right)\prod_{i \leq a < b < j} \zeta \left( \frac {z_a}{z_b} \right)}
$$
Each variable $z_a$ is integrated over the difference of two circles, one very small and one very large, which separate the set $S = \{0,\infty\}$ from the finite poles of the integral, by which we mean all the poles that arise from the functions $\zeta$ and $\tau_+$. The only poles that involve two variables $z_a$ and $z_b$ with $a<b$ are of the form:
\begin{align*}
&z_a - z_b \oq^{\frac {2(b-a-1)}n} \quad \text{ for } a \equiv b-1 \\ 
&z_a - z_bq^{-2} \oq^{\frac {2(b-a)}n} \ \text{ for } a \equiv b
\end{align*}
Because of the assumptions on the sizes of $q,\oq$, none of these poles hinder us from moving the contours of $z_i,...,z_{j-1}$ very far away from each other. Therefore:
$$
\text{RHS of \eqref{eqn:pushsyt}} \Big|_\bla = \int_{z_{j-1} \prec ... \prec z_i \prec \{0,\infty\}} \frac {M(z_i,...,z_{j-1}) \prod_{a=i}^{j-1} \zeta \left( \frac {z_a}{\chi_\bla} \right) \tau_+(z_a) Dz_a}{\prod_{a=i+1}^{j-1} \left(1 - \frac {z_a}{z_{a-1}q^2}\right)\prod_{i \leq a < b < j} \zeta \left( \frac {z_a}{z_b} \right)}
$$
as prescribed by \eqref{eqn:contour 1}. We will compute the-right hand side by integrating over the finite poles, i.e. those other than $0$ and $\infty$. The first variable to be integrated is $z_{j-1}$, and as shown in \eqref{eqn:Gamma+}, the finite poles are of the form $z_{j-1} = \chi_{j-1} := $ weight of an outer corner $\sq_{j-1} \in \bla$ of color $j-1$. If we write $\bnu^{j-1} = \bla \backslash \sq_{j-1}$, then:
$$
\text{RHS of \eqref{eqn:pushsyt}} \Big|_\bla = \sum_{\bla \geq_{j-1} \bnu^{j-1}} \int_{z_{j-2} \prec ... \prec z_i \prec \{0,\infty\}} (q^{-1}-q) \cdot
$$
$$
\frac {M(z_i,...,z_{j-1}) \zeta \left( \frac {z_{j-1}}{\chi_{\bnu^{j-1}}} \right) \tau_+(z_{j-1}) \prod_{a=i}^{j-2} \zeta \left( \frac {z_a}{\chi_{\bnu^{j-1}}} \right) \tau_+(z_a) Dz_a}{\prod_{a=i+1}^{j-1} \left(1 - \frac {z_a}{z_{a-1}q^2}\right)\prod_{i \leq a < b < j-1} \zeta \left( \frac {z_a}{z_b} \right)} \Big|_{z_{j-1} \mapsto \chi_{j-1}}
$$
(see the proof of Proposition \ref{prop:coeffs} for the way the factor $q^{-1}-q$ on the first row arises when computing the integral of $z_{j-1}$ by residues). One needs now to repeat the argument by integrating $z_{j-2}$ over the finite poles, which are now of the form $z_{j-2} = \chi_{j-2} := $ weight of an outer corner $\sq_{j-2}$ of $\bnu^{j-1}$ of color $j-2$. Iterating this argument gives rise to a flag of partitions \eqref{eqn:sytflag}, and the RHS of \eqref{eqn:pushsyt}$|_\bla$ matches \eqref{eqn:donald}. This completes the proof of \eqref{eqn:pushsyt} in the case $\pm = +$. \\

\noindent Let us now prove the case $\pm = -$ of \eqref{eqn:pushsyt}. Formula \eqref{eqn:eqloc vir} reads:
$$
\text{LHS of \eqref{eqn:pushsyt}}  = \sum_{\bmu \vdash \bd^-} | \bmu \rangle \sum^{\syt \text{ of lower}}_{\text{shape }\bmu} \frac {M(\CL_i,...,\CL_{j-1})\cdot [\fZ_{[i;j)}^-]}{\wedge^\bullet(T^\vee \fZ_{[i;j)} - T^\vee \CM_{\bd^-})} \ \Big|_{\syt}
$$	
Recall from \eqref{eqn:tautline} that the restriction of $\CL_a$ to the fixed point corresponding to a SYT is just the weight $\chi_a$ of the box labeled $a$ in the SYT. We will use \eqref{eqn:adjust2} and \eqref{eqn:form2} to compute the adjusted virtual class in the numerator, respectively the exterior class of the normal bundle in the denominator. Therefore, $\text{LHS of \eqref{eqn:pushsyt}} |_\bmu = $
$$
\sum^{\syt \text{ of lower}}_{\text{shape }\bmu} M(\chi_i,...,\chi_{j-1}) \frac {u_i...u_{j-1}}{\chi_i... \chi_{j-1}} \frac {\chi_i}{\chi_{j-1}} (-1)^{j-i-1} \frac {q^{\left \lceil \frac {j-i}n \right \rceil-2}}{q^{d_{i-1} - d_{j-1}}} \frac {(1-1)^{j-i}}{\prod_{a=i+1}^{j-1} \left(1- \frac {\chi_{a-1}q^2}{\chi_a}\right)}
$$
$$
\prod_{a < b}^{a \equiv b} \frac {1 - \frac {\chi_{a}}{\chi_{b}}}{1 - \frac {\chi_{a}q^2}{\chi_{b}}} \prod_{a < b}^{a \equiv b - 1} \frac {1 - \frac {\chi_{a}q^2}{\chi_{b}}}{1 - \frac {\chi_{a}}{\chi_{b}}} \prod_{a=i}^{j-1} \left[ \frac 1{1- \frac {u^2_a}{\chi_a}}\prod^{\col \sq = a}_{\sq \in \bmu} \frac {1 - \frac {\chi_\sq}{\chi_a}}{1 - \frac {\chi_\sq q^2}{\chi_a}} \prod^{\col \sq = a-1}_{\sq \in \bmu} \frac {1 - \frac {\chi_\sq q^2}{\chi_a}}{1 - \frac {\chi_\sq}{\chi_a}} \right] 
$$
where $\bd = |\bmu|$. The factor $(1-1)^{j-i}$ in the numerator of the first row is meant to cancel $j-i$ zeroes in the denominator of the second row. Explicitly, if we write a SYT in the form \eqref{eqn:sytflag}, then we have $\chi_\bmu = \chi_{\bnu^{a+1}} - \chi_{a} - ... - \chi_i$, and so the formula above implies that LHS of \eqref{eqn:pushsyt}$|_\bmu =$
$$
\sum^{\syt \text{ as in}}_{\eqref{eqn:sytflag}} M(\chi_i,...,\chi_{j-1}) \frac {u_i...u_{j-1}}{\chi_i... \chi_{j-1}} \frac {\chi_i}{\chi_{j-1}} (-1)^{j-i-1} \frac {q^{\left \lceil \frac {j-i}n \right \rceil-2}}{q^{d_{i-1} - d_{j-1}}} \frac {(1-q^2)^{j-i}}{\prod_{a=i+1}^{j-1} \left(1- \frac {\chi_{a-1}q^2}{\chi_a}\right)}
$$
$$
\prod_{a=i}^{j-1} \left[ \frac 1{1- \frac {u^2_a}{\chi_a}}\prod^{\col \sq = a}_{\sq \in \bnu^{a+1}} \frac {1 - \frac {\chi_\sq}{\chi_a}}{1 - \frac {\chi_\sq q^2}{\chi_a}} \prod^{\col \sq = a-1}_{\sq \in \bnu^{a+1}} \frac {1 - \frac {\chi_\sq q^2}{\chi_a}}{1 - \frac {\chi_\sq}{\chi_a}} \right] 
$$
Absorbing all the $u_a$'s, $\chi_a$'s and $q$'s into the second row, we obtain the formula:
\begin{equation}
\label{eqn:donald 2}
\text{LHS of \eqref{eqn:pushsyt}} \Big|_\bmu = 
\end{equation}
$$
= \sum^{\syt \text{ as in}}_{\eqref{eqn:sytflag}} \frac {M(\chi_i,...,\chi_{j-1}) (1-q^{-2})^{j-i} \prod_{a=i}^{j-1} \left[\zeta\left(\frac {\chi_{\bnu^{a+1}}}{\chi_a} \right) \tau_-(\chi_a) \right]^{-1} }{\prod_{a=i+1}^{j-1} \left(1- \frac {\chi_a}{\chi_{a-1}q^2}\right)} 
$$
Recalling the notion $\int^-$ from Definition \ref{def:normal}, we note that the $\text{RHS of \eqref{eqn:pushsyt}} |_\bmu = $
$$
\sum^{\text{functions}}_{\sigma : \{i,...,j-1\} \rightarrow \{1,-1\}}  \int^{|q|^{\pm 1}, |\oq|^{\pm 1} > 1}_{|z_{a}| = \gamma^{\sigma(a)}  \oq^{-\frac {2a}n}} \frac {M(z_i,...,z_{j-1}) \prod_{a=i}^{j-1} \left[\zeta \left( \frac {\chi_\bmu}{z_a} \right) \tau_-(z_a) \right]^{-1} \sigma(a) Dz_a}{q^{j-i} \prod_{a=i+1}^{j-1} \left(1 - \frac {z_a}{z_{a-1}q^2}\right)\prod_{i \leq a < b < j} \zeta \left( \frac {z_a}{z_b} \right)}
$$
Each variable $z_a$ is integrated over the difference of two circles, one very small and one very large, which separate the set $S = \{0,\infty\}$ from the finite poles of the integral, by which we mean all the poles that arise from the functions $\zeta$ and $\tau_-$. The only poles that involve two variables $z_a$ and $z_b$ with $a<b$ are of the form:
\begin{align*}
&z_a - z_b \oq^{\frac {2(b-a-1)}n} \quad \text{ for } a \equiv b-1 \\ 
&z_a - z_bq^{-2} \oq^{\frac {2(b-a)}n} \ \text{ for } a \equiv b
\end{align*}
Because of the assumptions on the sizes of $q,\oq$, none of these poles hinder us from moving the contours of $z_i,...,z_{j-1}$ very far away from each other. Therefore:
$$
\text{RHS of \eqref{eqn:pushsyt}} \Big|_\bmu = \int_{z_i \prec ... \prec z_{j-1} \prec \{0,\infty\}} \frac {M(z_i,...,z_{j-1}) \prod_{a=i}^{j-1} \Big[ \zeta \left( \frac {\chi_\bmu}{z_a} \right) \tau_-(z_a) \Big]^{-1} Dz_a}{q^{j-i} \prod_{a=i+1}^{j-1} \left(1 - \frac {z_a}{z_{a-1}q^2}\right)\prod_{i \leq a < b < j} \zeta \left( \frac {z_a}{z_b} \right)}
$$
as prescribed by \eqref{eqn:contour 2}. We will compute the-right hand side by integrating over the finite poles, i.e. those other than $0$ and $\infty$. The first variable to be integrated is $z_i$, and as shown in \eqref{eqn:Gamma-}, the finite poles are of the form $z_i = \chi_i := $ weight of an inner corner $\sq_i \in \bmu$ of color $i$. If we write $\bnu^{i+1} = \bmu \sqcup \sq_i$, then:
$$
\text{RHS of \eqref{eqn:pushsyt}} \Big|_\bmu = \sum_{\bnu^{i+1} \geq_i \bmu} \int_{z_{i+1} \prec ... \prec z_{j-1} \prec \{0,\infty\}} (1-q^{-2}) \cdot
$$
$$
\frac {M(z_i,...,z_{j-1}) \Big[ \zeta \left( \frac {\chi_{\bnu^{i+1}}}{z_i} \right) \tau_-(z_i) \Big]^{-1} \prod_{a=i+1}^{j-1} \Big[ \zeta \left( \frac {\chi_{\bnu^{i+1}}}{z_a} \right) \tau_-(z_a) \Big]^{-1} Dz_a }{q^{j-i-1} \prod_{a=i+1}^{j-1} \left(1 - \frac {z_a}{z_{a-1}q^2}\right)\prod_{i+1 \leq a < b < j} \zeta \left( \frac {z_a}{z_b} \right)} \Big|_{z_i \mapsto \chi_i}
$$
(see the proof of Proposition \ref{prop:coeffs} for the way the factor $1-q^{-2}$ on the first row arises when computing the integral of $z_i$ by residues). One needs now to repeat the argument by integrating $z_{i+1}$ over the finite poles, which are of the form $z_{i+1} = \chi_{i+1} := $ weight of an inner corner $\sq_{i+1}$ of $\bnu^{i+1}$ of color $i+1$. Iterating this argument gives rise to a flag of partitions \eqref{eqn:sytflag}, and we conclude that the RHS of \eqref{eqn:pushsyt}$|_\bmu$ equals \eqref{eqn:donald 2}. This completes the proof of \eqref{eqn:pushsyt} when $\pm = -$.

\end{proof}

\begin{proof} {\bf of Propositions \ref{prop:pushasyt} and \ref{prop:pushasyt change}:} Let us first prove \eqref{eqn:pushasyt} when the sign is $\pm = +$. The virtual equivariant localization formula \eqref{eqn:eqloc vir} reads:
$$
\text{LHS of \eqref{eqn:pushasyt}}  = \sum_{\bla \vdash \bd^+} | \bla \rangle \sum^{\asyt \text{ of upper}}_{\text{shape }\bla} \frac {M(\CL_i,...,\CL_{j-1}) \cdot [\ofZ_{[i;j)}^+]}{\wedge^\bullet(T^\vee \ofZ_{[i;j)} - T^\vee \CM_{\bd^+})} \ \Big |_{\asyt}
$$	
Use \eqref{eqn:adjust3} to compute the adjusted virtual class in the numerator, and \eqref{eqn:form3} to compute the exterior class in the denominator. Therefore, $\text{LHS of \eqref{eqn:pushasyt}} |_\bla = $
$$
\sum^{\asyt \text{ of upper}}_{\text{shape }\bla} M(\chi_i,...,\chi_{j-1}) \cdot u_{i+1}...u_j \frac {q^{- j + i - \left \lfloor \frac {j-i}n \right \rfloor}}{q^{d_i - d_j}} \cdot \frac {(1-1)^{j-i}}{\prod_{a=i+1}^{j-1} \left(1- \frac {\chi_{a-1}}{\chi_a}\right)}
$$
$$
\prod_{a < b}^{a \equiv b} \frac {1 - \frac {\chi_{b}}{\chi_{a}}}{1 - \frac {\chi_{b}q^2}{\chi_{a}}} \prod_{a < b}^{a \equiv b + 1} \frac {1 - \frac {\chi_{b}q^2}{\chi_{a}}}{1 - \frac {\chi_{b}}{\chi_{a}}} \prod_{a=i}^{j-1} \left[\left(1- \frac {\chi_aq^2}{u^2_{a+1}}\right) \prod^{\col \sq = a}_{\sq \in \bla} \frac {1 - \frac {\chi_aq^2}{\chi_\sq}}{1 - \frac {\chi_a}{\chi_\sq}} \prod^{\col \sq = a+1}_{\sq \in \bla} \frac {1 - \frac {\chi_a}{\chi_\sq}}{1 - \frac {\chi_a q^2}{\chi_\sq}} \right] 
$$
As before, the factor $(1-1)^{j-i}$ in the first row is meant to cancel $j-i$ zeroes in the denominator of the rational function on the second row. Explicitly, we may write an ASYT in the form \eqref{eqn:asytflag}, where $\bla = \bnu^s \sqcup B_1 \sqcup ... \sqcup B_s$. Here, each $B_s$ is a vertical strip of boxes of colors $k_{s-1},...,k_s-1$, all of which have the same weight $\chi_s$ if we apply rule\footnote{By this we mean that a box of color $k$ is said to have weight $\chi$ if its weight \eqref{eqn:weight} as a box of color $\bark$ is $\chi \cdot \oq^{2\left \lfloor \frac {k-1}n \right \rfloor}$, where $\bark$ is the residue class of $k$ in the set $\{1,...,n\}$} \eqref{eqn:color change}. Therefore, the formula above implies $\text{LHS of \eqref{eqn:pushasyt}} |_\bla = $
$$
\sum^{\asyt \text{ as in}}_{\eqref{eqn:asytflag}} M(\chi_i,...,\chi_{j-1}) \cdot u_{i+1}...u_j \frac {q^{- j + i - \left \lfloor \frac {j-i}n \right \rfloor}}{q^{d_i - d_j}} \cdot \frac {(1-q^2)^t}{\prod_{s=1}^{t-1}  \left( 1- \frac {\chi_{k_s-1}}{\chi_{k_s}} \right)} \prod_{a=i}^{j-1} \left(1- \frac {\chi_aq^2}{u^2_{a+1}}\right)
$$
$$
\prod_{s=1}^t \left( \prod_{k_{s-1} \leq a < b-1 < k_s-1} \left[ \frac {1-\frac {\chi_a q^2}{\chi_b}}{1-\frac {\chi_a}{\chi_b}} \right]^{\delta_{\bar{a}}^{\bar{b}} - \delta_{\bar{a}+1}^{\bar{b}}}  \prod_{a=k_{s-1}}^{k_s-1} \left[ \prod^{\col \sq = a}_{\sq \in \bnu^s} \frac {1 - \frac {\chi_aq^2}{\chi_\sq}}{1 - \frac {\chi_a}{\chi_\sq}} \prod^{\col \sq = a+1}_{\sq \in \bnu^s} \frac {1 - \frac {\chi_a}{\chi_\sq}}{1 - \frac {\chi_a q^2}{\chi_\sq}} \right] \right)
$$
Absorbing all the $u_a$'s and $q$'s into the second row, we obtain the formula:
\begin{equation}
\label{eqn:donald 3}
\text{LHS of \eqref{eqn:pushasyt}} \Big|_\bla = 
\end{equation}
$$
= \sum^{\asyt \text{ as in}}_{\eqref{eqn:asytflag}} \frac {M(\chi_i,...,\chi_{j-1}) (q^{-1}-q)^t \prod_{s=1}^t \left[ \beta_{k_s - k_{s-1}} \prod_{a=k_{s-1}}^{k_s-1} \zeta \left( \frac {\chi_a}{\chi_{\bnu^s}} \right) \tau_+(\chi_a)  \right]}{\prod_{s=1}^{t-1} \left(1- \frac {\chi_{k_s-1}}{\chi_{k_s}}\right)}
$$
where:
\begin{equation}
\label{eqn:def beta}
\beta_k = \prod_{0 \leq a < b-1 < k - 1} \zeta \left( \frac {1_a}{1_b} \right) = \prod_{i=1}^{\left \lfloor \frac {k-1}n \right \rfloor} \frac {q^{-1} - q \cdot \oq^{-2i}}{1-\oq^{-2i}}
\end{equation}
where $1_{a}$ refers the value 1 as a variable of color $a$. \\

\noindent Recalling the notion $\int^+$ from Definition \ref{def:normal}, we note that the $\text{RHS of \eqref{eqn:pushasyt}} |_\bla = $
$$
= \sum^{\text{functions}}_{\sigma : \{i,...,j-1\} \rightarrow \{1,-1\}}  \int^{|q|^{\pm 1}, |\oq|^{\pm 1} < 1}_{|z_{a}| = \gamma^{\sigma(a)}  \oq^{-\frac {2a}n}} \frac {M(z_i,...,z_{j-1}) \prod_{a=i}^{j-1} \zeta \left( \frac {z_a}{\chi_\bla} \right) \tau_+(z_a) \sigma(a) Dz_a}{\prod_{a=i+1}^{j-1} \left(1 - \frac {z_{a-1}}{z_a}\right)\prod_{i \leq a < b < j} \zeta \left( \frac {z_b}{z_a} \right)} 
$$
Each variable $z_a$ is integrated over the difference of two circles, one very small and one very large, which separate the set $S = \{0,\infty\}$ from the finite poles of the integral, by which we mean all the poles that arise from the functions $\zeta$ and $\tau_+$. The only poles that involve two variables $z_a$ and $z_b$ with $a<b$ are of the form:
\begin{align}
&z_a - z_b \oq^{\frac {2(b-a+1)}n} \ \text{ for } a \equiv b + 1 \label{eqn:poles1} \\ 
&z_a - z_b q^2 \oq^{\frac {2(b-a)}n} \ \text{ for } a \equiv b \label{eqn:poles2}
\end{align}
as well as $z_{a} = z_{a-1}$. Because of the assumptions on the sizes of $q,\oq$, the poles \eqref{eqn:poles1}--\eqref{eqn:poles2} do not hinder us from moving the contours of $z_i,...,z_{j-1}$ very far away from each other (with $z_b$ closer to $0$ and $\infty$ than $z_a$, if $a<b$). However, the poles $z_a = z_{a-1}$ do hinder us, and we have to consider the corresponding residues:
$$
\text{RHS of \eqref{eqn:pushasyt}}\Big |_\bla = \sum^{t\geq 1}_{i = k_0 < k_1 < ... < k_t = j} \int_{z_i = ... = z_{k_1-1} \prec ... \prec z_{k_{t-1}} = ... = z_{j-1} \prec \{0,\infty\}}
$$
\begin{equation}
\label{eqn:chain}
\frac {M(z_i,...,z_{j-1}) \prod_{a=i}^{j-1} \zeta \left( \frac {z_a}{\chi_\bla} \right) \tau_+(z_a)}{\prod_{s=1}^{t-1} \left(1 - \frac {z_{k_s-1}}{z_{k_s}}\right)\prod_{i \leq a < b < j} \zeta \left( \frac {z_b}{z_a} \right)} \prod_{s=1}^t Dz_{k_s-1} 
\end{equation}
as prescribed by \eqref{eqn:contour 3}. Each summand in \eqref{eqn:chain} corresponds to a {\bf chain}: 
$$
i = k_0 < ... < k_t = j
$$
Such a chain can be thought of as consisting of {\bf links}, namely collections of variables: 
\begin{equation}
\label{eqn:equal}
z_{k_s-1} = ... = z_{k_{s-1}}
\end{equation}
which are set equal to each other in \eqref{eqn:chain}. We will think of such a link as a vertical strip of boxes of colors $k_s-1,...,k_{s-1}$ (indeed, the weights of the boxes in such a vertical strip are all equal, given the convention \eqref{eqn:tautline}). We will compute the integral \eqref{eqn:chain} as a sum of residues over the finite poles, namely those poles other than $0$ and $\infty$. The first variable which we have to integrate is $y_1 := z_{i} = ... = z_{k_1-1}$, and by \eqref{eqn:Gamma+} we have a pole in this variable whenever:
$$
y_1 = \chi_\sq \quad \text{ for }\sq \text{ an outer corner of color }c\text{ of }\bla
$$
for some $c \in \{i,...,k_1-1\}$. As explained above, this corresponds to taking a vertical strip consisting of boxes of colors $k_1-1,...,i$ and placing the box of color $c$ in this strip on top of the outer corner $\sq \in \bla$. Because of the linear factors $1 - \frac {z_{k_s-1}}{z_{k_s}}$ in \eqref{eqn:chain}, the same residue can be obtained in other ways, by considering the ``finer" chain where the link $z_i = ... = z_{k_1-1}$ is subdivided into any $r+1$ smaller links: 
$$
z_i = ... = z_{c} \qquad z_{c+1} = ... = z_{u_1 - 1} \qquad ... \qquad z_{u_{r-1}} = ... = z_{k_1-1}
$$ 
The sign with which the residue appears from the finer chain is $(-1)^{r}$. Since there are $2^{k_1-1-c}$ finer chains, the total contribution of the given residue is:
$$
\sum_{r=0}^{k_1-1-c} (-1)^{r} {k_1-1-c \choose r} = \delta_{k_1-1}^c
$$
We conclude that the only non-zero residue one obtains is when the box of color $k_1-1$ of the vertical strip is placed on top of the outer corner $\sq \in \bla$. In other words, this corresponds to removing a vertical strip from the partition $\bla$:
$$
\bla = \bnu^1 \sqcup \Big( \text{vertical strip } \sq_{k_1-1},...,\sq_i \Big)
$$
for some partition $\bnu^1$. Letting $\chi_a$ denote the weight of the box $\sq_a$, we see that:
$$
\text{RHS of \eqref{eqn:pushasyt}}\Big |_\bla =  \sum^{t\geq 1}_{i = k_0 < k_1 < ... < k_t = j} \int_{z_{k_1} = ... = z_{k_2-1} \prec ... \prec z_{k_{t-1}} = ... = z_{j-1} \prec \{0,\infty\}} M(z_i,...,z_{j-1})
$$
$$
\frac {(q^{-1}-q) \beta_{k_1-i} \prod_{a=i}^{k_1-1} \zeta \left( \frac {z_a}{\chi_{\bnu^1}} \right) \tau_+(z_a) \prod_{a=k_1}^{j-1} \zeta \left( \frac {z_a}{\chi_{\bnu^1}} \right) \tau_+(z_a)}{\prod_{s=1}^{t} \left(1 - \frac {z_{k_s-1}}{z_{k_s}}\right)\prod_{k_1 \leq a < b < j} \zeta \left( \frac {z_b}{z_a} \right)} \prod_{s=2}^t Dz_{k_s-1} \ \Big|_{z_{k_1-1} \mapsto \chi_{k_1-1},...,z_i \mapsto \chi_i}
$$
One needs now to integrate $z_{k_2 - 1} = ... = z_{k_1}$ over the finite poles, which by the same reasoning as above, corresponds to removing a whole vertical strip from $\bnu^1$. Iterating this argument gives rise to a flag of partitions \eqref{eqn:asytflag}, in other words an almost standard Young tableau of upper shape $\bla$, and the RHS of \eqref{eqn:pushasyt}$|_\bla$ equals \eqref{eqn:donald 3}. This completes the proof of \eqref{eqn:pushasyt} when $\pm = +$. \\

\noindent Let us now prove the case $\pm = -$ of \eqref{eqn:pushasyt}. Formula \eqref{eqn:eqloc vir} reads:
$$
\text{LHS of \eqref{eqn:pushasyt}}  = \sum_{\bmu \vdash \bd^-} | \bmu \rangle \sum^{\asyt \text{ of lower}}_{\text{shape }\bmu} \frac {M(\CL_i,...,\CL_{j-1}) \cdot [\ofZ_{[i;j)}^-]}{\wedge^\bullet(T^\vee \ofZ_{[i;j)} - T^\vee \CM_{\bd^-})} \ \Big |_{\asyt}
$$	
Use \eqref{eqn:adjust4} to compute the adjusted virtual class in the numerator, and \eqref{eqn:form4} to compute the exterior class in the denominator. Therefore, $\text{LHS of \eqref{eqn:pushasyt}} |_\bla = $
$$
\sum^{\asyt \text{ of lower}}_{\text{shape }\bmu} M(\chi_i,...,\chi_{j-1}) \cdot \frac {u_i...u_{j-1}}{\chi_i... \chi_{j-1}} \frac {q^{- j + i - \left \lfloor \frac {j-i}n \right \rfloor}}{q^{d_{i-1} - d_{j-1}}} \cdot \frac {(1-1)^{j-i}}{\prod_{a=i+1}^{j-1} \left(1- \frac {\chi_{a-1}}{\chi_a}\right)}
$$
$$
\prod_{a < b}^{a \equiv b} \frac {1 - \frac {\chi_{b}}{\chi_{a}}}{1 - \frac {\chi_{b}q^2}{\chi_{a}}} \prod_{a < b}^{a \equiv b + 1} \frac {1 - \frac {\chi_{b}q^2}{\chi_{a}}}{1 - \frac {\chi_{b}}{\chi_{a}}} \prod_{a=i}^{j-1} \left[ \frac 1{\left(1- \frac {u^2_a}{\chi_a}\right)} \prod^{\col \sq = a}_{\sq \in \bmu} \frac {1 - \frac {\chi_\sq}{\chi_a}}{1 - \frac {\chi_\sq q^2}{\chi_a}} \prod^{\col \sq = a-1}_{\sq \in \bmu} \frac {1 - \frac {\chi_\sq q^2}{\chi_a}}{1 - \frac {\chi_\sq}{\chi_a}} \right] 
$$
As before, the factor $(1-1)^{j-i}$ in the first row is meant to cancel $j-i$ zeroes in the denominator of the rational function on the second row. Explicitly, we may write an ASYT in the form \eqref{eqn:asytflag}, where $\bmu = \bnu^{s-1} \backslash (B_s \sqcup ... \sqcup B_t)$, where each $B_s$ is a vertical strip of boxes of colors $k_{s-1},...,k_s-1$. Therefore, the formula above implies $\text{LHS of \eqref{eqn:pushasyt}} |_\bla = $
$$
\sum^{\asyt \text{ as in}}_{\eqref{eqn:asytflag}} M(\chi_i,...,\chi_{j-1}) \cdot \frac {u_i...u_{j-1}}{\chi_i... \chi_{j-1}} \frac {q^{- j + i - \left \lfloor \frac {j-i}n \right \rfloor}}{q^{d_{i-1} - d_{j-1}}} \cdot \frac {(1-q^2)^t}{\prod_{s=1}^{t-1}  \left( 1- \frac {\chi_{k_s-1}}{\chi_{k_s}} \right)} \prod_{a=i}^{j-1} \frac 1{\left(1- \frac {u^2_a}{\chi_a}\right)}
$$
$$
\prod_{s=1}^t \left( \prod_{k_{s-1} \leq a < b-1 < k_s-1} \left[ \frac {1-\frac {\chi_a q^2}{\chi_b}}{1-\frac {\chi_a}{\chi_b}} \right]^{\delta_{\bar{a}}^{\bar{b}} - \delta_{\bar{a}+1}^{\bar{b}}}  \prod_{a=k_{s-1}}^{k_s-1} \left[\prod^{\col \sq = a}_{\sq \in \bnu^{s-1}} \frac {1 - \frac {\chi_\sq}{\chi_a}}{1 - \frac {\chi_\sq q^2}{\chi_a}} \prod^{\col \sq = a-1}_{\sq \in \bnu^{s-1}} \frac {1 - \frac {\chi_\sq q^2}{\chi_a}}{1 - \frac {\chi_\sq}{\chi_a}} \right] \right)
$$
Absorbing all the $u_a$'s, $\chi_a$'s and $q$'s into the second row, we obtain the formula:
\begin{equation}
\label{eqn:donald 4}
\text{LHS of \eqref{eqn:pushasyt}} \Big|_\bmu =  \sum^{\asyt \text{ as in}}_{\eqref{eqn:asytflag}} 
\end{equation}
$$
\frac {M(\chi_i,...,\chi_{j-1}) (-1)^{j-i} (q^{-1} - q)^t \prod_{s=1}^t \left[ \beta_{k_s - k_{s-1}} \prod_{a=k_{s-1}}^{k_s-1} \left[ \zeta \left( \frac {\chi_{\bnu^{s-1}}}{\chi_a} \right)  \tau_-(\chi_a) \right]^{-1}  \right]}{q^{j-i}\prod_{s=1}^{t-1} \left(1- \frac {\chi_{k_s-1}}{\chi_{k_s}}\right)}
$$
where $\beta_k$ is defined in \eqref{eqn:def beta}. Recalling the notion $\int^-$ from Definition \ref{def:normal}, we note that the $\text{RHS of \eqref{eqn:pushasyt}}|_\bmu = $
$$
= \sum^{\text{functions}}_{\sigma : \{i,...,j-1\} \rightarrow \{1,-1\}}  \int^{|q|^{\pm 1}, |\oq|^{\pm 1} > 1}_{|z_{a}| = \gamma^{\sigma(a)}  \oq^{-\frac {2a}n}} \frac {M(z_i,...,z_{j-1}) \prod_{a=i}^{j-1} \left[ \zeta \left( \frac {\chi_\bmu}{z_a} \right) \tau_-(z_a) \right]^{-1} \sigma(a) Dz_a}{q^{j-i}\prod_{a=i+1}^{j-1} \left(1 - \frac {z_{a-1}}{z_a}\right)\prod_{i \leq a < b < j} \zeta \left( \frac {z_b}{z_a} \right)} 
$$
Each variable $z_a$ is integrated over the difference of two circles, one very small and one very large, which separate the set $S = \{0,\infty\}$ from the finite poles of the integral, by which we mean all the poles that arise from the functions $\zeta$ and $\tau_+$. The only poles that involve two variables $z_a$ and $z_b$ with $a<b$ are of the form \eqref{eqn:poles1}--\eqref{eqn:poles2}, as well as $z_{a} = z_{a-1}$. Because of the assumptions on the sizes of $q,\oq$, the poles \eqref{eqn:poles1}--\eqref{eqn:poles2} do not hinder us from moving the contours of $z_i,...,z_{j-1}$ very far away from each other (with $z_a$ closer to $0$ and $\infty$ than $z_b$, if $a<b$). However, the poles $z_a = z_{a-1}$ do hinder us, and we have to consider the corresponding residues:
$$
\text{RHS of \eqref{eqn:pushasyt}}\Big |_\bmu = \sum^{t\geq 1}_{i = k_0 < k_1 < ... < k_t = j} \int_{z_{k_{t-1}} = ... = z_{j-1} \prec ... \prec  z_i = ... = z_{k_1-1} \prec \{0,\infty\}} (-1)^{j-i-t}
$$
\begin{equation}
\label{eqn:chain 2}
\frac {M(z_i,...,z_{j-1}) \prod_{a=i}^{j-1} \left[\zeta \left( \frac {\chi_\bmu}{z_a}\right) \tau_-(z_a) \right]^{-1}}{\prod_{s=1}^{t-1} \left(1 - \frac {z_{k_s-1}}{z_{k_s}}\right)\prod_{i \leq a < b < j} \zeta \left( \frac {z_b}{z_a} \right)} \prod_{s=1}^t Dz_{k_s-1} 
\end{equation}
as prescribed by \eqref{eqn:contour 4}. The factor $(-1)^{j-i-t}$ arises from taking the residue at $z_{a-1}$ in the variable $z_{a}$:
$$
\forall a \in \{i+1,...,j-1\} \backslash \{k_1,...,k_{t-1}\}
$$ 
As in the case $\pm = +$ treated above, each summand in \eqref{eqn:chain 2} corresponds to a chain $i = k_0 < ... < k_t = j$ consisting of links of equal variables \eqref{eqn:equal}. We will compute the integral \eqref{eqn:chain 2} as a sum of residues over the finite poles, namely those poles other than $0$ and $\infty$. The first variable which we have to integrate is $y_t := z_{k_{t-1}-1} = ... = z_{j-1}$, and by \eqref{eqn:Gamma-} we have a pole in this variable whenever:
$$
y_t = \chi_\sq \quad \text{ for }\sq \text{ an inner corner of color }c\text{ of }\bmu
$$
for some $c \in \{k_{t-1},...,j-1\}$. As in the case $\pm = +$ discussed above, the only choice which produces a non-zero contribution to the integral \eqref{eqn:chain 2} is when $c = k_{t-1}$, which corresponds to adding a vertical strip to $\bmu$:
$$
\bmu \sqcup \Big( \text{vertical strip } \sq_{j-1},...,\sq_{k_{t-1}} \Big) = \bnu^{t-1}
$$
for some partition $\bnu^{t-1}$. Letting $\chi_a$ denote the weight of the box $\sq_a$, we see that:
$$
\text{RHS of \eqref{eqn:pushasyt}}\Big |_\bmu =  \sum^{t\geq 1}_{i = k_0 < k_1 < ... < k_t = j} \int_{z_{k_{t-2}} = ... = z_{k_{t-1}-1} \prec ... \prec z_i = ... = z_{k_1-1} \prec \{0,\infty\}} (-1)^{j-i-t} 
$$
$$
\frac { \prod_{a=k_{t-1}}^{j-1} \left[ \zeta \left( \frac {\chi_{\bnu^{t-1}}}{z_a} \right) \tau_-(z_a) \right]^{-1} \prod_{a=i}^{k_{t-1}-1} \left[\zeta \left( \frac {\chi_{\bnu^{t-1}}}{z_a} \right) \tau_-(z_a) \right]^{-1}}{q^{j-i} \prod_{s=1}^{t} \left(1 - \frac {z_{k_s-1}}{z_{k_s}}\right)\prod_{i \leq a < b < k_{t-1}} \zeta \left( \frac {z_b}{z_a} \right)} 
$$
$$
M(z_i,...,z_{j-1}) (q - q^{-1}) \beta_{j-k_{t-1}} \prod_{s=1}^{t-1} Dz_{k_s-1} \ \Big|_{z_{j-1} \mapsto \chi_{j-1},...,z_{k_{t-1}} \mapsto \chi_{k_{t-1}}}
$$
One needs now to integrate $z_{k_{t-1} - 1} = ... = z_{k_{t-2}}$ over the finite poles, which by the same reasoning corresponds to adding a whole vertical strip to $\bnu^{t-1}$. Iterating this argument gives rise to a flag of partitions \eqref{eqn:asytflag}, in other words an almost standard Young tableau of lower shape $\bmu$, and the RHS of \eqref{eqn:pushasyt}$|_\bmu$ equals \eqref{eqn:donald 4}. This completes the proof of \eqref{eqn:pushasyt} when $\pm = -$. 

\end{proof}

\begin{proof} {\bf of Proposition \ref{prop:pushsmooth} and \ref{prop:pushsmooth change}:} We will only prove \eqref{eqn:pushsmooth} and \eqref{eqn:contours smooth} when the sign is $+$, and leave the case of $-$ to the interested reader. The virtual equivariant localization formula \eqref{eqn:eqloc vir} allows to compute the LHS of \eqref{eqn:pushsmooth}:
$$
\pi^+_*\Big(M(...,\CL_i,...) \cdot [\fW_\bk^+] \Big) = \sum_{\bla} |\bla \rangle \sum^{\bmu \text{ such that}}_{\blamu \text{ is as in }\eqref{eqn:strip}} \frac {M(...,\CL_i,...) \cdot [\fW_\bk^+] }{\wedge^\bullet\left([T^\vee\fW_\bk] - [T^\vee\CM_{\bd^+}]\right)} \ \Big|_{(\bla, \bmu)}
$$
We will use $\CL_i|_{(\bla,\bmu)} = \sum_{\bsq \in \blamu}^{\col \bsq = i} \chi_\bsq$ and the definition of the adjusted virtual class \eqref{eqn:adjust5} to compute the numerator, and \eqref{eqn:form5} to compute the denominator:
$$
\text{LHS of \eqref{eqn:pushsmooth}} \Big|_\bla = \sum^{\bmu \text{ such that}}_{\blamu \text{ is as in }\eqref{eqn:strip}} M(\blamu) \cdot \frac {u_{\bk+1}}{q^{|\bk|+\langle \bk, \bd^+ \rangle}} \cdot \prod_{\bsq \in \blamu} \left(1 - \frac {\chi_\bsq q^2}{u_{\col \bsq + 1}^2} \right)
$$
$$
\prod_{\bsq,\sq \in \blamu} \left(1 - \frac {\chi_\bsq}{\chi_{\sq}} \right)^{\delta_{\col \bsq}^{\col \sq} - \delta_{\col \bsq +1}^{\col \sq}} \prod^{\bsq \in \blamu}_{\sq \in \bla} \left( \frac {1 - \frac {\chi_\bsq q^2}{\chi_\sq}}{1 - \frac {\chi_\bsq}{\chi_\sq}}\right)^{\delta^{\col \sq}_{\col \bsq} - \delta^{\col \sq}_{\col \bsq+1}}  
$$
If we convert the sum over $\sq \in \bla$ to a sum over $\sq \in \bmu$, the expression above equals:
$$
\text{LHS of \eqref{eqn:pushsmooth}} \Big|_\bla = \sum^{\bmu \text{ such that}}_{\blamu \text{ is as in }\eqref{eqn:strip}} M(\blamu) \cdot \frac {u_{\bk+1}}{q^{|\bk|+\langle \bk, \bd^+ \rangle}} \cdot \prod_{\bsq \in \blamu} \left(1 - \frac {\chi_\bsq q^2}{u_{\col \bsq + 1}^2} \right)
$$
$$
\prod_{\bsq,\sq \in \blamu} \left(1 - \frac {\chi_\bsq q^2}{\chi_{\sq}} \right)^{\delta_{\col \bsq}^{\col \sq} - \delta_{\col \bsq +1}^{\col \sq}} \prod^{\bsq \in \blamu}_{\sq \in \bmu} \left( \frac {1 - \frac {\chi_\bsq q^2}{\chi_\sq}}{1 - \frac {\chi_\bsq}{\chi_\sq}}\right)^{\delta^{\col \sq}_{\col \bsq} - \delta^{\col \sq}_{\col \bsq+1}}  
$$
and if we absorb the $u_a$'s and $q$'s from the first row into the second row, we obtain:
\begin{equation}
\label{eqn:mish}
\text{LHS of \eqref{eqn:pushsmooth}} \Big|_\bla = \sum^{\bmu \text{ such that}}_{\blamu \text{ is as in }\eqref{eqn:strip}} M(\blamu) \cdot 
\end{equation}
$$
\prod_{\bsq,\sq \in \blamu} \left(\frac 1q - \frac {\chi_\bsq q}{\chi_{\sq}} \right)^{\delta_{\col \bsq}^{\col \sq} - \delta_{\col \bsq +1}^{\col \sq}}  \prod_{\bsq \in \blamu} \left[ \zeta \left( \frac {\chi_\bsq}{\chi_\bmu} \right)\tau_+(\chi_\bsq) \right]
$$
We will compute the right hand side of \eqref{eqn:pushsmooth} by changing the contours in the integral, and showing that the resulting residue computation equals \eqref{eqn:mish}. To this end, recall that the normal ordered integral was defined in Definition \ref{def:normal} as:
$$
\text{RHS of \eqref{eqn:pushsmooth}} \Big |_\bla = \frac {1}{\bk!} \sum^{\text{functions}}_{\sigma : \{(i,a)\}^{1 \leq i \leq n}_{1\leq a \leq k_i} \rightarrow \{1,-1\}}  \int^{|q|^{\pm 1}, |\oq|^{\pm 1} < 1}_{|z_{ia}| = \gamma^{\sigma(i,a)}  \oq^{-\frac {2i}n}} M(...,z_{ia},...) 
$$
\begin{equation}
\label{eqn:tina}
\frac {\prod_{1 \leq i \leq n} \prod_{1 \leq a \neq b \leq k_i} \left(1 - \frac {z_{ib}}{z_{ia}} \right)}{\prod_{1 \leq i \leq n} \prod^{1 \leq a \leq k_i}_{1 \leq b \leq k_{i-1}} \left(1 - \frac {z_{i-1,b}}{z_{ia}} \right)} \prod_{1 \leq i \leq n}^{1 \leq a \leq k_i} \zeta \left( \frac {z_{ia}}{\chi_\bla} \right) \tau_+(z_{ia}) \sigma(i,a) Dz_{ia}
\end{equation}
The only poles that involve two of the $z$ variables are of the form $z_{ia} = z_{i-1,b}$. We will seek to move the contours of $\{z_{ia}\}$ far apart from each other by the following procedure: pick an arbitrary $z_{ia}$, $1 \leq a \leq k_i$, which we will henceforth call the \textbf{starting variable}, and seek to move its contour toward $\{0,\infty\}$. The only poles we can pick up along the way are of the form $z_{ia} = z_{i-1,b}$ for some $1 \leq b \leq k_{i-1}$. For the corresponding residue, try to move the contour of $z_{i-1,b}$ toward $\{0,\infty\}$. As there are zeroes at $z_{i-1,b} = z_{i-1,b'}$ in the integrand, the only poles we can pick up are of the form $z_{ia} = z_{i-1,b} = z_{i-2,c}$ for some $1 \leq c \leq i-2$. Iterating this argument shows that the only residues one obtains are given by partitions of the variable set:
\begin{equation}
\label{eqn:turner}
\{z_{ia}\}^{1\leq i \leq n}_{1\leq a \leq k_i} = B_1 \sqcup B_2 \sqcup ... \sqcup B_t
\end{equation}
where:
\begin{equation}
\label{eqn:strip proof}
B_s = \Big\{ z_{i_s,a_s}, ..., \underline{z_{j_s-1,b_s}}\Big\}
\end{equation}
for various $i_s < j_s$, and we underline the starting variable. The variables in each group are set equal to each other $z_{i_s,a_s} = ... = z_{j_s-1,b_s}  =: y_s$, so we interpret each $B_s$ as the set of weights of a vertical strip of boxes of colors $i_s,...,j_s-1$. With the above discussion in mind, \eqref{eqn:tina} yields:
$$
\text{RHS of \eqref{eqn:pushsmooth}} \Big |_\bla = \sum^{\text{decompositions}}_{\text{as in \eqref{eqn:turner}}}  \int_{y_1 \prec ... \prec y_t \prec \{0,\infty\}} \frac {M(...,z_{ia},...)}{\#_{B_1,...,B_t}}
$$
\begin{equation}
\label{eqn:oier}
\prod_{s,s'=1}^t \prod_{z' \in B_{s'}}^{z\in B_s} \left(1 - \frac {z}{z'} \right)^{\delta^{\col z}_{\col z'} - \delta^{\col z}_{\col z'-1}} \prod_{1\leq s \leq t}^{z\in B_s} \zeta \left( \frac z{\chi_\bla} \right)\tau_+(z) \Big|^{z \mapsto y_s}_{\text{if }z \in B_s} \prod_{s=1}^t Dy_s \qquad
\end{equation}
While we will not mention this explicitly in order to keep the notation legible, one must take care to remove all 0 factors from the numerators or denominators of our products, such as the situation when $z = z'$ in the first product on the second row of \eqref{eqn:oier}. The denominator that appears on the first row of \eqref{eqn:oier} is:
\begin{equation}
\label{eqn:diez}
\#_{B_1,...,B_t} =
\end{equation} 
$$
= \prod_{a=1}^t \Big( (j_a-1)\text{--th entry of the vector } \bk - [i_{a+1};j_{a+1}) - ... - [i_t;j_t) \in \nn \Big)
$$
The appearance of $\#_{B_1,...,B_t}$ in \eqref{eqn:oier} stems from the fact that the second indices of the variables $\{z_{jb}\}$ can be permuted at random, except for the second indices of the starting variables, which are fixed by our algorithm. This proves \eqref{eqn:contours smooth}. \\

\noindent To compute the integral \eqref{eqn:oier}, we will first integrate $y_1$ over the finite poles of the integral (namely those poles $\neq 0,\infty$). According to \eqref{eqn:Gamma+}, these are of the form $y_1 = \chi_1 = $ weight of an outer corner $\sq$ of the partition $\bla$. This corresponds to a vertical strip of boxes of colors $i_1,...,j_1-1$, such that the outer corner $\sq$ has color $c \in \{i_1,...,j_1-1\}$. As in the discussion following \eqref{eqn:chain} in the proof of Proposition \ref{prop:pushasyt}, the corresponding residue is canceled by an inclusion-exclusion argument, unless:
$$
c = j_1 - 1 \quad \Rightarrow \quad \bla = \bnu^1 \sqcup \Big( \text{vertical strip } B_1 \Big)
$$ 
With this in mind, we may rewrite \eqref{eqn:oier} by breaking up the products in terms of whether each variable $z$ lies in $B_2 \sqcup .. \sqcup B_t$ or in $B_1$. Therefore, $\text{RHS of \eqref{eqn:pushsmooth}} |_\bla =$
$$
= \sum^{\text{decompositions}}_{\text{as in \eqref{eqn:turner}}} \ \int_{y_2 \prec ... \prec y_t \prec \{0,\infty\}} \frac {M(...,z_{ia},...)}{\#_{B_1,...,B_t}} \prod_{2 \leq s \leq t}^{w\in B_1, z\in B_s} \frac {\left(1-\frac wz\right)^{\delta^{\col z}_{\col w}}\left(\frac 1q -\frac {zq}{w}\right)^{\delta^{\col z}_{\col w}}}{\left(1-\frac {w}{z}\right)^{\delta^{\col z}_{\col w+1}}\left(\frac 1q -\frac {zq}w \right)^{\delta^{\col z}_{\col w-1}}}
$$
\begin{equation}
\label{eqn:tempy}
\prod_{w,w' \in B_1} \left(1 - \frac {w}{w'} \right)^{\delta^{\col w}_{\col w'} - \delta^{\col w}_{\col w' - 1}} \prod_{s,s'=2}^{t} \prod_{z' \in B_{s'}}^{z\in B_s} \left(1 - \frac {z}{z'} \right)^{\delta^{\col z}_{\col z'} - \delta^{\col z}_{\col z' - 1}}
\end{equation}
$$
\prod_{w\in B_1} \zeta \left( \frac w{\chi_\bla} \right)\tau_+(w) \prod_{2 \leq s \leq t}^{z\in B_s} \zeta \left( \frac z{\chi_{\bnu^1}} \right)\tau_+(z) \Big|^{w \mapsto y_1, z \mapsto y_s}_{\text{if }w \in B_1, z \in B_s} \ \prod_{s=2}^{t} Dy_s \Big|_{y_1 \mapsto \chi_1}				
$$
We may now analyze the finite poles of the integrand in $y_2$, which corresponds to those variables in $B_2$. Note that the finite poles come from $\chi_{\bnu^1}$ as well as the $w$ variables. Relation \eqref{eqn:Gamma+} tells us that the third row of \eqref{eqn:tempy} has a pole when $y_2 = \chi_2 = $ weight of an outer corner of the partition $\bnu^1 = \bla \backslash B_1$, but also a zero when $y_2 q^2 = $ weight of an inner corner of the partition $\bnu^1$. The latter zero precisely cancels out the pole produced by the first row of \eqref{eqn:tempy}, so we can repeat our argument to conclude that the variables in $B_{t-1}$ correspond to removing a strip:
$$
\bnu^1 = \bnu^2 \sqcup \Big(\text{vertical strip }B_2\Big)
$$
which does not share a common vertical edge with $B_1$. Iterating this argument shows that integrating over the various remaining variables $y_3,...,y_t$ amounts to removing vertical strips from $\bla$, no two sharing a vertical edge. We obtain:
$$
\text{RHS of \eqref{eqn:pushsmooth}} \Big |_\bla = \sum^{\text{decompositions}}_{\text{as in \eqref{eqn:turner}}} \frac {M(...,z_{ia},...)}{\#_{B_1,...,B_t}} \prod_{1\leq s < s' \leq t}^{w\in B_s, z\in B_{s'}} \frac {\left(1-\frac wz\right)^{\delta^{\col z}_{\col w}}\left(\frac 1q -\frac {zq}{w}\right)^{\delta^{\col z}_{\col w}}}{\left(1-\frac {w}{z}\right)^{\delta^{\col z}_{\col w+1}}\left(\frac 1q -\frac {zq}w \right)^{\delta^{\col z}_{\col w-1}}}
$$
$$
\cdot \prod_{s=1}^t \left( \prod_{z,z' \in B_s} \left(1 - \frac {z}{z'} \right)^{\delta^{\col z}_{\col z'} - \delta^{\col z}_{\col z'-1}} \prod_{z\in B_s} \zeta \left( \frac z{\chi_{\bnu^{s-1}}} \right)\tau_+(z) \right) \Big|^{z \mapsto \text{weight of }s-\text{th strip}}_{\text{if }z \in B_s} 
$$
In the formula above, the variables $z \in B_s$ become specialized to the weights of the boxes in the strips $B_s$, $\forall s \in \{1,...,t\}$. If we let $\bmu = \bnu_t$, hence: 
$$
\bnu_{s-1} = \bmu \sqcup B_s \sqcup ... \sqcup B_t
$$
for all $s$, then the expression above becomes:
\begin{equation}
\label{eqn:phish}
\text{RHS of \eqref{eqn:pushsmooth}} \Big|_\bla = \sum^{\bmu \text{ such that}}_{\blamu = B_1 \sqcup ... \sqcup B_t} \frac {M(\blamu)}{\#_{B_1,...,B_t}} \cdot 
\end{equation}
$$
\prod_{\bsq,\sq \in \blamu} \left(\frac 1q - \frac {\chi_\bsq q}{\chi_{\sq}} \right)^{\delta_{\col \bsq}^{\col \sq} - \delta_{\col \bsq +1}^{\col \sq}}  \prod_{\bsq \in \blamu} \zeta \left( \frac {\chi_\bsq}{\chi_\bmu} \right)\tau_+(\chi_\bsq)
$$
Besides $\#_{B_1,...,B_t}$, the only difference between formulas \eqref{eqn:mish} and \eqref{eqn:phish} is that: \\

\begin{itemize}[leftmargin=*]
	
\item In \eqref{eqn:mish}, $\blamu$ is decomposed into strips $S_1,...,S_t$ with no common edge \\

\item In \eqref{eqn:phish}, each summand corresponds to a decomposition of $\blamu$ into vertical strips $B_1,...,B_t$ which may have a horizontal common edge, but no vertical common edges (the factors in the numerator of the first line of \eqref{eqn:tempy} kill poles corresponding to vertical strips $B_i$ which may share vertical common edges). \\

\end{itemize}

\noindent It is clear that any skew diagram $\blamu$ can be decomposed at most once as in the first bullet. Meanwhile, any decomposition as in the second bullet can only be a refinement of a decomposition in the first bullet (meaning that every $S_a$ splits up into several of the $B_i$'s). Therefore, the fact that \eqref{eqn:mish} equals \eqref{eqn:phish} (which would complete the proof of the Proposition) follows from the fact that for any collection of vertical strips $S_1,...,S_t$ with no common edge, we have the equality:
\begin{equation}
\label{eqn:equality}
\sum^{B_1 \sqcup ... \sqcup B_{t'} = }_{= S_1 \sqcup ... \sqcup S_t} \frac 1{\#_{B_1,...,B_{t'}}} = 1 
\end{equation}
with the understanding that the decomposition $\sqcup B_i$ is a refinement of the decomposition $\sqcup S_a$. Let us prove \eqref{eqn:equality} by induction on the total number of boxes in the strips $B_a$. Our algorithm prescribes that for any first chosen starting variable $z_{ia}$, a box of color $i$ must be at the bottom of the vertical strip $B_1$, which in turn must be at the top of the vertical strip $S_a$ for some $a \in \{1,...,k\}$. Therefore:
$$
S_a = S_{a}' \sqcup_i B_1
$$
where the symbol $\sqcup_i$ means that the vertical strip $S_{a}'$ is directly below $B_1$, and the bottom-most box of $B_1$ has color $\equiv i$ mod $n$. Therefore, the LHS of \eqref{eqn:equality} equals:
\begin{equation}
\label{eqn:equalness}
\frac 1{g} \sum^{1\leq a \leq t}_{S_a = S_a' \sqcup_i B_1} \quad \sum^{B_2 \sqcup ... \sqcup B_{t'} = }_{= S_1 \sqcup ... \sqcup S_a' \sqcup ... \sqcup S_t} \frac 1{\#_{B_2,...,B_{t'}}}
\end{equation}
where $g$ is the number of boxes of color $\equiv i$ mod $n$ in the collection $S_1 \sqcup ... \sqcup S_t$. By the induction hypothesis, all inner sums are equal to 1. Since the number of outer sums is precisely equal to $g$, we conclude that the entire \eqref{eqn:equalness} equals 1, as desired.

\end{proof}

\end{document}